\documentclass{amsart}

\usepackage{amsmath,amssymb,cite}
\usepackage{amsthm}
\usepackage{amsrefs}
\usepackage{mathtools}
\usepackage{mathrsfs}
\usepackage{tikz}
\usetikzlibrary{patterns}
\usepackage{diagbox}
\usepackage{array}
\usepackage{graphicx}
\usepackage{relsize}
\usepackage{enumitem}
\usepackage[titletoc]{appendix}

\usepackage{caption}
\usepackage{subcaption}  
\usepackage{multirow}
\usepackage{tabularx}
\usepackage{tikz-cd}
\usepackage[plainpages=false,pdfborder={0 0 0}]{hyperref}

\newtheorem{thm}{Theorem}[section]
\newtheorem{lem}[thm]{Lemma}
\newtheorem{prop}[thm]{Proposition}
\newtheorem{cor}[thm]{Corollary}
\theoremstyle{definition}

\theoremstyle{remark}
\newtheorem{rmk}[thm]{Remark}
\newcommand{\p}{\partial}
\newcommand{\R}{\mathbb{R}}

\numberwithin{equation}{section}

\begin{document}

\title[Localized Deformation of Curvatures]{Localized Deformation of the Scalar Curvature and the Mean Curvature}


\author{Hongyi Sheng}
\address{Institute for Theoretical Sciences, Westlake Institute for Advanced Study, Westlake University, Hangzhou, Zhejiang Province 310024, China}
\email{shenghongyi@westlake.edu.cn}
\thanks{}


\date{}

\dedicatory{}

\begin{abstract}
On a compact manifold with boundary, the map consisting of the scalar curvature in the interior and the mean curvature on the boundary is a local surjection at generic metrics \cite{C-V}. We prove that this result may be localized to compact subdomains in an arbitrary Riemannian manifold with boundary. This result is a generalization of Corvino's result \cite{C} about localized scalar curvature deformations; however, the existence part needs to be handled delicately since the linearized problem is non-variational. We also discuss generic conditions that guarantee localized deformations, and related geometric properties.
\end{abstract}

\maketitle


\section{Introduction}\label{sec1}
An initial data set for a spacetime $(\mathbb{L}, \gamma)$ satisfying the Einstein field equations is a triple $(M,g,k)$, where $(M,g)$ is a spacelike hypersurface in $\mathbb{L}$, and $k$ is its second fundamental form. Physically plausible initial data sets must adhere to both the Einstein constraint equations and the dominant energy condition (DEC). Consequently, deformations aimed at satisfying the DEC serve as crucial analytical tools in the study of initial data sets. Broadly, two major approaches have been developed: \emph{conformal methods} and \emph{localized deformations}.

The conformal method, which involves a global rescaling of the metric, has a long and successful history. It transforms the deformation problem into a tractable PDE problem. This approach was famously used by Schoen and Yau in their proof of the Positive Mass Theorem \cites{S-Y1, S-Y3}. In the time-symmetric case ($k=0$), where the DEC implies non-negative scalar curvature, this method underpins the extensive study of prescribing scalar curvature. This includes the classic Yamabe problem on closed manifolds \cites{Trudinger, Aubin, Schoen} and its natural extension to manifolds with boundary, which involves simultaneously prescribing the interior scalar curvature $R$ and the boundary mean curvature $H$ (the latter corresponding to trapping conditions in general relativity; see \cites{Dain-J-K}). This boundary problem, initiated by Escobar \cites{E, E1, E2, E3, E4}, has seen significant development \cites{Han-Li, Han-Li1, Mayer-N, Brendle-C, Marques}.

In contrast, localized deformations—which involve \emph{compactly supported} variations of the metric—present a different set of challenges and goals. Their key feature, the preservation of geometry outside the region of interest, makes them essential for gluing constructions \cites{C, C-E-M, C-H, C-S, C-D, Delay} and for proving rigidity type results \cite{C}*{Theorem 8}. For interior domains, this field was pioneered by Corvino \cite{C}, who established the local surjectivity of the scalar curvature operator. This was later extended to the full constraint map with Schoen \cite{C-S} and developed into a systematic theory using weighted spaces by Chru\'sciel-Delay \cite{C-D}, Carlotto-Schoen \cites{Carlotto-S}, Corvino-Huang \cites{C-H}, and others.

Despite this progress, previous work on localized deformations has focused almost exclusively on the interior setting (on closed manifolds or compact subdomains). \emph{The problem of constructing localized deformations near the boundary of a manifold has remained largely unaddressed.}

This paper provides the first systematic study of this problem. We introduce a new perspective using Fredholm theory to analyze localized deformations that are supported near \emph{a portion} of the manifold's boundary. Our main focus is on the coupled system of simultaneously deforming the scalar curvature $R$ in the interior and the mean curvature $H$ on the boundary. \emph{This coupling fundamentally alters the problem's structure:} the linearized problem is no longer variational, as it is in the interior-only case. This change also requires a more delicate analysis of a non-elliptic operator.

Our main result, stated below, generalizes Corvino's work \cite{C} by proving the local surjectivity of the map $\Psi'_{g_0}(a) = (R(g_0+a), H(g_0+a))$ for symmetric $(0,2)$-tensors $a$ with compact support near a boundary patch.

\begin{thm}\label{main}
Let $M^n (n \ge 2)$ be a Riemannian manifold (not necessarily compact) with boundary $\p M$ and a $\mathcal{C}^{k+4,\alpha}$-metric $g_0$. Let $\overline{\Omega^n}\subset\overline{M}$ be a compact, connected subdomain with $C^{k+4,\alpha}$ boundary such that:
\begin{itemize}
\item $\Omega$ is contained in a neighborhood of $\p M$,
\item part of $\partial\Omega$ coincides with $\p M$, and
\item the interior of this common boundary, $\varnothing\neq\Sigma^{n-1}\triangleq\operatorname{int}\left(\p\Omega\cap\p M\right)$, is a $C^{k+4,\alpha}$ hypersurface.
\end{itemize}
Assume the generic conditions from Section~\ref{sec3} hold. Then there exists $\epsilon > 0$ such that for all pairs $(R',H')$ satisfying:
\begin{itemize}
\item[1.] $R'\in C^{k, \alpha}(\Omega\cup\Sigma)$, with $R' - R(g_0)\in \mathcal{B}_{0}(\Omega)$ and $\operatorname{supp}(R' - R(g_0))\subset\overline{\Omega}$,
\item[2.] $H'\in C^{k+1, \alpha}(\Sigma)$, with $H' - H(g_0)\in \mathcal{B}_{1}(\Sigma)$ and $\operatorname{supp}(H' - H(g_0))\subset\overline{\Sigma}$,
\item[3.] $\|R' - R(g_0)\|_{\mathcal{B}_{0}(\Omega)} + \|H' - H(g_0)\|_{\mathcal{B}_{1}(\Sigma)} < \epsilon,$
\end{itemize}
there exists a metric $g\in \mathcal{C}^{k+2,\alpha}$ on $M$ with $R(g) = R'$ in $\Omega$,  $H(g) = H'$ on $\Sigma$ and $g \equiv g_0$ outside $\Omega\cup\Sigma$.
\end{thm}
(The precise definitions of the weighted Banach spaces $\mathcal{B}_{0}(\Omega)$ and $\mathcal{B}_{1}(\Sigma)$ are given in Section~\ref{sec2}, and the generic conditions are discussed in Section~\ref{sec3}.)

\subsection*{Key technical innovations}
Our result is not a trivial generalization of the interior case. The presence of the partial boundary $\Sigma$ and its ``corner" $\partial\Sigma$ (a codimension-$2$ submanifold) introduces significant new difficulties:
\begin{itemize}[leftmargin=*]
\item[1.] \textbf{Non-Variational Structure:} Unlike the interior problem, the linearized coupled system lacks a variational structure; its analysis thus relies on Fredholm theory. We show the space of weak solutions decomposes into a principal component ($a=\rho L^*u$) and a finite-dimensional subspace. While the principal component satisfies an elliptic boundary value problem, allowing for the derivation of weighted Schauder estimates, the finite-dimensional subspace does not. As its underlying tensor boundary value problem (\ref{solution2}) is not elliptic, this subspace demands a careful construction of its basis to establish optimal regularity.
\item[2.] \textbf{Delicate Analysis at the ``Corner":} The localization requires weight functions dependent on the distance to $\partial\Omega\setminus\Sigma$. The analysis of weighted Sobolev spaces near the codimension-$2$ region $\partial\Sigma$ is highly subtle, as the distance function is not smooth there and its higher derivatives can blow up.
\item[3.] \textbf{New Embedding Theorems:} We establish new embedding theorems between weighted Sobolev spaces, proving a converse to results in \cite{C-H}. This requires a careful application of Hardy's inequality, especially near $\partial\Sigma$.
\end{itemize}

We note that the global deformation result of Cruz-Vit\'orio \cite{C-V}*{Theorem 3.5} can be seen as a special case of our theorem, where $\Omega = M, \Sigma = \partial M$, and the weight function is identically $1$.

\subsection*{Applications and implications}
Our work has several important implications that we shall defer to a subsequent publication. First, it yields a density theorem for a class of asymptotically flat, scalar-flat manifolds with non-compact minimal boundary that are Schwarzschild at infinity (a boundary version of the density theorem in \cite{C}; see \cite{Sheng-Zhao}). More broadly, it provides a systematic methodology that can be iterated to extend other landmark interior gluing results \cites{C-S, Carlotto-S, C-H} to settings with boundary.

Second, as an application, we demonstrate how our method resolves a limitation in the problem of regularizing area-minimizing hypersurfaces in initial data sets. It is well-known that in dimensions $n \geq 8$, such hypersurfaces can develop singularities. For $n=8$, Smale \cite{Sm} showed that a local metric perturbation can generically produce a smooth minimizer. However, this perturbation does not control scalar curvature, potentially violating the DEC and limiting its physical applicability.

Our localized deformation technique overcomes this obstacle (see Figure \ref{fig: deformation}). Beginning with an asymptotically flat, scalar-flat manifold $(M^8, g_0)$ containing a singular area-minimizing boundary $\Sigma_0$, we first apply Smale's result to obtain a metric $g_1$ for which the minimizer $\Sigma_1$ is smooth. We then use our main theorem to construct a second metric $g_2$, close to $g_1$, such that the scalar curvature vanishes in a neighborhood of $\Sigma_1$ and the mean curvature of $\Sigma_1$ is zero. This ensures $\Sigma_1$ is mean convex, so the new minimizer $\Sigma_2$ lies outside $\Sigma_1$. An application of the Hardt-Simon regularity theorem \cite{H-S} then implies $\Sigma_2$ is smooth. Crucially, $(M, g_2)$ remains asymptotically flat and scalar-flat in the exterior region, preserving the ADM mass. This provides a method for regularizing a singular apparent horizon to a smooth one in dimension $8$ while maintaining the DEC.

\begin{figure}
\begin{center}
						
\begin{tikzpicture}
\draw (0, 1.67) .. controls (0.3, 0)  .. (0, -1.67); \draw [loosely dashed] (0, 1.67) .. controls (-0.3, 0) .. (0, -1.67);
\draw (-1, 1.47) .. controls (-0.7, 0)  .. (-1, -1.47); \draw [loosely dashed] (-1, 1.47) .. controls (-1.3, 0) .. (-1, -1.47);
\draw (-2, 1.4) .. controls (-1.7, 0)  .. (-2, -1.4); \draw (-2, 1.4) .. controls (-2.3, 0) .. (-2, -1.4);

\draw (-2, 1.4) .. controls (-1, 1.45) and (0, 1.5) .. (2.5, 2.5);
\draw (-2, -1.4) .. controls (-1, -1.45) and (0, -1.5) .. (2.5, -2.5);

\fill[fill=gray!50]
(-2, 1.4) .. controls (-1.7, 0)  .. (-2, -1.4) .. controls (-2.3, 0) ..(-2, 1.4);
\fill[fill=gray!50]
(0, 1.67) .. controls (0.3, 0)  .. (0, -1.67) .. controls (-0.3, 0) .. (0, 1.67);
\fill[fill=gray!50]
(-1, 1.47) .. controls (-0.7, 0) .. (-1, -1.47) .. controls (-1.3, 0) .. (-1, 1.47);

\node at (-2,1.7) {$\Sigma_0$}; 
\node at (-1,1.85) {$\Sigma_1$};
\node at (0,2) {$\Sigma_2$};
\draw (-2, 0) -- (-2.5,0); \node[right] at (-4.05, 0) {singular};
\node at (2,0) {$M$};

\draw (5, 1.47) .. controls (5.3, 0)  .. (5, -1.47); \draw [loosely dashed] (5, 1.47) .. controls (4.7, 0) .. (5, -1.47);
\draw (4, 1.4) .. controls (4.3, 0)  .. (4, -1.4); \draw (4, 1.4) .. controls (3.7, 0) .. (4, -1.4);
\draw (4.12, 0.9) .. controls (6, 0.3) and (6,-0.3) .. (4.12, -0.9);

\draw (4, 1.4) .. controls (5, 1.45) and (6, 1.5) .. (8.5, 2.5);
\draw (4, -1.4) .. controls (5, -1.45) and (6, -1.5) .. (8.5, -2.5);

\fill[fill=gray!50]
(4, 1.4) .. controls (4.3, 0)  .. (4, -1.4) .. controls (3.7, 0) ..(4, 1.4);
\fill[fill=gray!50]
(5, 1.47) .. controls (5.3, 0) .. (5, -1.47) .. controls (4.7, 0) .. (5, 1.47);

\draw[pattern=north west lines]
(4.12, 0.9) .. controls (6, 0.3) and (6, -0.3).. (4.12, -0.9).. controls (4.15, 0) .. (4.12, 0.9);

\node at (4,1.7) {$\Sigma_0$}; 
\node at (5,1.75) {$\Sigma_1$};
\node at (8,0) {$M$};
\draw (4.5, -0.5) -- (5,-1.9); \node[below right] at (4.3,-1.8) {neg. scalar curv.};

\end{tikzpicture}

						
\begin{tikzpicture}
\draw (0, 3) .. controls (0.5, 0)  .. (0, -3);
\draw (0, 3) .. controls (-0.5, 0) .. (0, -3);
\draw (0.17, 2) .. controls (3, 0.5) and (3,-0.5) .. (0.17, -2);
\draw (0, 3) .. controls (4, 1.5) and (4,-1.5) .. (0, -3);

\fill[fill=gray!50]
(0, 3) .. controls (0.5, 0)  .. (0, -3) .. controls (-0.5, 0) ..(0,3); 

\draw[pattern=north west lines]
(0.17, 2) .. controls (3, 0.5) and (3,-0.5).. (0.17, -2).. controls (0.45, 0) .. (0.17, 2);

\draw (-1, -1) -- (0,0); \node[left] at (-1,-1) {$\Sigma$}; \node[left] at (-0.2,-1.4){(part of $\Sigma_1$)};
\draw (3, 1) -- (2.3,1); \node[right] at (3, 1) {$\Omega$};
\draw (1.25, -1) -- (3,-1.5); \node[right] at (3, -1.5) {neg. scalar curv.};

\end{tikzpicture}
\caption{Smale's perturbation and localized deformation}
\label{fig: deformation}


\end{center}
\end{figure}

\subsection*{Structure of the paper}
The paper is organized as follows. 
\begin{itemize}[leftmargin=*]
\item In Section~\ref{sec2}, we introduce basic notation and establish preliminary results in weighted spaces, including trace, density, and embedding theorems. 
\item In Section~\ref{sec3}, we define the generic conditions that guarantee localized deformations and discuss their geometric interpretation. 
\item Section~\ref{sec4} is devoted to the proof of our main theorem:
\begin{itemize}
\item[\ref{sec4.1}:] We establish Poincar\'e-type inequalities in weighted spaces, which are critical for the Dirichlet problems (\ref{Dirichlet eqn2 weak}), (\ref{Dirichlet}) and the self-adjoint problem (\ref{eqn:self-adjoint}).
\item[\ref{sec4.2}:] We analyze two Dirichlet problems (Theorems~\ref{Dirichlet1}, \ref{ext}). Theorem~\ref{ext} allows us to simplify boundary value problems by transforming them into problems on boundary function spaces. This leads to the construction of $\mathcal{D}$, a weighted Hilbert space that replaces $H^{3/2}(\Sigma)$ in the classical setting.
\item[\ref{sec4.3}:] We construct a self-adjoint problem by adding a lower-order term to (\ref{prob: 2u}). This new problem is then solved using variational methods (Theorem~\ref{self-adjoint soln weak}).
\item[\ref{sec4.4}:] We show that the added term is a compact operator. This allows us to apply Fredholm theory to solve (\ref{prob: 2u}) in Theorem~\ref{Fredholm}.
\item[\ref{sec: weak soln}:] We study the densely defined operators $\Psi$ and $B$ to prove the existence of weak solutions (Theorems~\ref{Surjectivity of B}, \ref{non-homogeneous BVP}). We show the tensor solution space decomposes into a principal component ($a=\rho L^*u$) and a finite-dimensional subspace. Because the boundary value problem for tensors (\ref{solution2}) is not elliptic, we need to carefully construct a basis for the finite-dimensional subspace and establish its optimal regularity (Theorem~\ref{basis}).
\item[\ref{sec4.6}:] The principal component satisfies an elliptic boundary value problem, allowing us to derive weighted Schauder estimates (Theorems~\ref{Global Schauder estimates}, \ref{Global Schauder2}). These estimates, combined with Theorem~\ref{basis}, yield the optimal regularity.
\item[\ref{sec4.7}:] Finally, we employ a Picard iteration scheme to extend this linear analysis and construct the desired solution to the full nonlinear problem.
\end{itemize}
\end{itemize}

\section*{Acknowledgments}
I would like to thank my PhD advisor Rick Schoen for his wisdom and patience during the preparation of my dissertation. I would like to thank Kai-Wei Zhao for helpful discussions.

\section{Preliminaries}\label{sec2}
\subsection{Basic notation}
Let $M^n \,(n \ge 2)$ be a smooth Riemannian manifold with boundary $\p M$. Let $\overline{\Omega^n}\subset \overline{M}$ be a compact, connected subdomain with smooth boundary, situated in a neighborhood of $\p M$, such that the interior of the intersection $\p\Omega\cap\p M$ is a non-empty, smooth hypersurface; we denote this shared portion by $\Sigma^{n-1}$. Let $\Sigma' \triangleq \partial\Omega\setminus\Sigma$ be the remaining part of the boundary, and let $\Gamma^{n-2} \triangleq \partial \Sigma$ be the (possibly empty) interface between them. For example, one can think of $\Omega$ as a ball and $\Sigma$ as the lower hemi-sphere, see Figure \ref{fig: weight function} below.

We list here some notation and function spaces, where $k \in \mathbb{N}$ and $\alpha\in(0,1)$. Note that we are following the notation of \cite{C} closely but with minor differences. 

\begin{itemize}[leftmargin=*]
\item $\operatorname{Ric}_g = R_{ij}$ and $R_g = R(g) = g^{ij} R_{ij}$ denote the Ricci and scalar curvatures, respectively, of a Riemannian metric $g$ on $M$; we use the Einstein summation convention throughout.
\item Let $D$ and $\nabla$ denote the Levi-Civita connections of $(\Omega, g)$ and $(\p\Omega, \hat{g})$, respectively, with $\hat{g}$ the induced metric on the boundary.
\item Let $d\mu_g$ denote the volume measure induced by $g$, and $d\sigma_g$ the induced surface measure on submanifolds. Note $d\sigma_g = d\mu_{\hat{g}}$ on $\p\Omega$.
\item $\mathcal{S}^{(0,2)}$ denotes the space of symmetric $(0,2)$-tensor fields.
\item $\mathcal{C}^\infty$ denotes the subspace of $\mathcal{S}^{(0,2)}$ consisting of smooth tensors.
\item $\mathcal{H}^k$ denotes the subspace of $\mathcal{S}^{(0,2)}$ consisting of those measurable tensors which are square integrable along with the first $k$ weak covariant derivatives; with the standard $\mathcal{H}^k$-inner product induced by the metric $g$, $\mathcal{H}^k$ becomes a Hilbert space. $H^k$ is defined similarly for functions, and the spaces $\mathcal{H}^k_\text{loc} (H^k_\text{loc})$ are defined as the spaces of tensors (functions) which are in $\mathcal{H}^k(H^k)$ on each compact subset.
\item $\mathcal{M}^k (k > \frac n2)$ denotes the open subset of $\mathcal{H}^k$ of Riemannian metrics, and $\mathcal{M}^{k,\alpha}$ denotes the open subset of metrics in $\mathcal{C}^{k,\alpha}$.
\item Let $\rho$ be a smooth positive function on $\Omega$. Define $L^2_\rho(\Omega)$ to be the set of locally $L^2(d\mu_g)$ functions $f$ such that $f\rho^{\frac12} \in L^2(\Omega, d\mu_g)$. The pairing
$$
\langle f_1, f_2\rangle_{L_{\rho}^{2}(\Omega)}\triangleq\left\langle f_1 \rho^{\frac12}, f_2 \rho^{\frac12}\right\rangle_{L^{2}(\Omega, d\mu_g)}
$$
makes $L^2_\rho(\Omega)$ a Hilbert space. Define $\mathcal{L}^2_\rho(\Omega)$ similarly for tensor fields.
\item Let $H_\rho^k(\Omega)$ be the Hilbert space of $L^2_\rho(\Omega)$ functions whose covariant derivatives up through order $k$ are also $\mathcal{L}^2_\rho(\Omega)$, i.e. the following norm is finite
$$\|u\|^2_{H_\rho^{k}(\Omega)}\triangleq\sum_{j=0}^{k}\left\|D^{j} u\right\|^2_{\mathcal{L}^2_\rho(\Omega)}.$$
Define $\mathcal{H}_\rho^k(\Omega)$ similarly for tensor fields.
\item For integer $0 \leq j \leq k - 1$, let $\gamma_j : H^k(\Omega) \rightarrow H^{k- j -\frac12}(\partial\Omega)$ denote the trace map of the extension of $\gamma_j(u) \triangleq  (D_\nu^j u) \rvert_{\partial\Omega}, \, u\in C^\infty(\overline{\Omega})$. For simplicity, we often use $\hat u$ to denote $\gamma_0 u$ and $u_\nu$ for $\gamma_1 u$.
\item Define
$$
    H_{0}^k(\Omega) \triangleq \text{the closure of }C^\infty_c(\Omega)\text{ in }H^k(\Omega).
$$
Thus if $u\in H_0^k(\Omega)$, $\gamma_j(u) = 0$ for all $0 \leq j \leq k - 1$. Define $\mathcal{H}_{0}^k(\Omega)$ similarly for tensor fields.
\item Define
\begin{align*}
    C^\infty_c(\Omega\cup\Sigma) & \triangleq \left\{u\in C^\infty(\overline{\Omega}): \operatorname{supp} u \subset\joinrel\subset \Omega\cup\Sigma \right\},\\
    H^k(\Omega, \Sigma') & \triangleq \text{the closure of }C^\infty_c(\Omega\cup\Sigma)\text{ in }H^k(\Omega).
\end{align*}
Note that if $u\in C^\infty_c(\Omega\cup\Sigma)$, then any continuous extension of $u$ vanishes at $\Sigma'$, and hence for any $u\in H^k(\Omega, \Sigma')$, $\gamma_j \rvert_{\Sigma'}(u) = 0$ for all $0 \leq j \leq k - 1$. Define $\mathcal{H}^k(\Omega, \Sigma')$ similarly for tensor fields.
\item Define
$$
    H_{0, \rho}^k(\Omega) \triangleq \text{the closure of }C^\infty_c(\Omega)\text{ in }H_\rho^k(\Omega).
$$
Define $\mathcal{H}_{0, \rho}^k(\Omega)$ similarly for tensor fields.
\item Define
$$
    H_{\nu, \rho}^2(\Omega) \triangleq \left\{u\in H^2_\rho(\Omega): u_\nu = 0 \text{ on } \Sigma\right\}.
$$
\item Let $\left(H_{\nu, \rho}^{2}(\Omega)\right)^*$ denote the dual space of $H_{\nu, \rho}^2(\Omega)$, and $\left(H_{0, \rho}^{k}(\Omega)\right)^*$ the dual space of $H_{0, \rho}^k(\Omega)$. 
\end{itemize}

\subsection{Linearized equations and integration by parts formula}
The linearization of the scalar curvature $L_g$ is given by \cite{F-M}*{Lemma 2}, and the linearization of the mean curvature $\dot{H}_g$ can be found in \cite{M-T}*{Eq. 33}. In the appendix, we use local coordinates to derive an equivalent formula for $\dot{H}_g$. As a summary, we have
\begin{prop}\label{linearized RH}
Let $l+2>\frac{n}{2}+1$ and $k \geq 0$.

The scalar curvature map is a smooth map of Banach manifolds, as a map $R: \mathcal{M}^{l+2}(\Omega) \rightarrow H^{l}(\Omega)$, or $R: \mathcal{M}^{k+2, \alpha}(\overline\Omega) \rightarrow C^{k, \alpha}(\overline\Omega)$. The linearization $L_g$ of the scalar curvature operator is given by
$$L_g(a)=-\Delta_{g}\left(\operatorname{tr}_{g}a\right)+\operatorname{div}_{g}(\operatorname{div}_{g}a)-\left<a, \operatorname{Ric}_g\right>_{g}$$
in the above spaces.

The mean curvature map is a smooth map of Banach manifolds, as a map $H: \mathcal{M}^{l+2}(\Omega) \rightarrow H^{l+\frac12}(\p\Omega)$, or $H: \mathcal{M}^{k+2, \alpha}(\overline\Omega) \rightarrow C^{k+1, \alpha}(\p\Omega)$. The linearization $\dot{H}_g$ of the mean curvature operator of $\p\Omega$ is given by
$$\dot{H}_g(a) = \frac12 \nu(\operatorname{tr}_{\p\Omega} a)-\operatorname{div}_{\p\Omega} a(\nu, \cdot) - \frac12 a(\nu, \nu) H_g$$
in the above spaces.
\end{prop}
We often suppress $g$ from the notation when it is clear from the context.


Now let us derive an integration by parts formula between $L_g$ and its formal $L^2$-adjoint $L^*_g$, i.e.
$$\int_\Omega L_g(a)u\, d \mu_g = \int_\Omega \left<a, L_g^*u\right>d \mu_g + \text{Bdry},$$
where
$$
L_g^*(u)=-\left(\Delta_{g} u\right) g+\operatorname{Hess}_g(u) - u \operatorname{Ric}_g
$$
and
\begin{align*}
\text{Bdry} & = \int_{\p\Omega} \sum_\alpha u((D_\alpha a)_{n\alpha} - \p_n(\operatorname{tr} a))d\sigma_g - \int_{\p\Omega} \sum_\alpha(a_{n\alpha}\p_\alpha u - u_n\operatorname{tr} a)d\sigma_g\\
		& = \int_{\p\Omega} \sum_i u((D_i a)_{ni} + (D_n a)_{nn} - \p_n(\operatorname{tr}_{\p\Omega} a) - \p_n a_{nn})d\sigma_g \\
		& \quad - \int_{\p\Omega} \sum_i (a_{ni}\p_i u - u_n\operatorname{tr}_{\p\Omega} a)d\sigma_g.
\end{align*}
Here $e_1, \dots, e_{n-1}$ are tangent to $\p\Omega$ and $e_n=\nu$ is outward normal to $\p\Omega$; $i, j = 1, \dots, n-1; \, \, \alpha = 1, \dots, n$. Note that in Fermi coordinates $D_n e_n = 0$, so
$$(D_n a)_{nn} = \p_n a_{nn} - 2a(e_n, D_n e_n) = \p_n a_{nn}.$$
Since
\begin{align*}
\sum_i (D_i a)_{ni} & = \sum_i \left(\p_i a_{ni} - a(e_n, D_i e_i) - a(D_i e_n, e_i)\right)\\
		& = \sum_i\left(\p_i a_{ni} - a(e_n, \nabla_{e_i} e_i)\right) + \sum_i a(e_n, h_{ii}e_n) - \sum_{i,j}a(h_{ij} e_j, e_i)\\
		& = \operatorname{div}_{\p\Omega} a(\nu, \cdot) + a(\nu, \nu)H - \left<a, h \right>_{\hat{g}},
\end{align*}
we then have
\begin{align*}
\text{Bdry} & = \int_{\p\Omega} \sum_i u((D_i a)_{ni} - \nu(\operatorname{tr}_{\p\Omega} a))d\sigma_g + \int_{\p\Omega} u\operatorname{div}_{\p\Omega} a(\nu, \cdot)d\sigma_g + \int_{\p\Omega} u_\nu\operatorname{tr}_{\p\Omega} a\,d\sigma_g\\
		& = \int_{\p\Omega} 2u\left(- \frac12\nu(\operatorname{tr}_{\p\Omega} a) + \operatorname{div}_{\p\Omega} a(\nu, \cdot) + \frac12 a(\nu, \nu)H\right)d\sigma_g - \int_{\p\Omega} u\left<a, h \right>_{\hat{g}}d\sigma_g\\
		& \quad + \int_{\p\Omega} u_\nu\operatorname{tr}_{\p\Omega} a\,d\sigma_g\\
		& = -\int_{\p\Omega} 2u\dot{H}(a)d\sigma_g - \int_{\p\Omega} u\left<a, h \right>_{\hat{g}}d\sigma_g + \int_{\p\Omega} u_\nu\operatorname{tr}_{\p\Omega} a\,d\sigma_g.
\end{align*}

\begin{prop}[Green’s formula]\label{IBP}
For $a\in\mathcal{C}^\infty(\overline{\Omega}), u\in C^\infty(\overline{\Omega})$, we have
$$\int_\Omega L_g(a)u \, d \mu_g + \int_{\p\Omega} 2\dot{H}_g(a)u\,d\sigma_g = \int_\Omega \left<a, L_g^*u\right> d \mu_g + \int_{\p\Omega} \left(- u\left<a, h \right>_{\hat{g}} + u_\nu\operatorname{tr}_{\p\Omega} a\right)d\sigma_g.$$
\end{prop}

\subsection{Approximate distance and weight function}\label{Weight function}
We would like to define a weight function $\rho$ on $\Omega\cup\Sigma$ whose behavior near $\Sigma'$ is determined by the distance to $\Sigma'$. We will define $0\le\rho\le1$  to tend monotonically to zero with decreasing distance to $\Sigma'$, and use the weight $\rho^{-1}$ which blows up at $\Sigma'$ to form weighted $L^2$-spaces of tensors, which by design will be forced to decay suitably at $\Sigma'$. 

We will work with uniformly equivalent metrics in a bounded open set $\mathcal{U}_0\subset\mathcal{M}^{k,\alpha}(\overline{\Omega})\, (k \ge 2)$. For fixed $g\in \mathcal{U}_0$ and $x\in\Omega\cup\Sigma$, define
\begin{align*}
d_g(x) & = \mathrm{dist}_{(\overline{\Omega}, g)} (x, \Sigma')\\
	& = \inf\{\ell_g(\gamma): \text{any piecewise smooth path $\gamma \subset \overline{\Omega}$ connecting $x$ to a point in }\Sigma'\},
\end{align*}
where $\ell_g(\gamma)$ is the length of $\gamma$ with respect to $g$. 
In most of the existing literature on weighted Sobolev spaces, only the distance function $d(x, \partial \Omega)$, which is smooth near $\partial \Omega$, is considered. However, our $d_g$ is not $C^2$ near $\Gamma$, introducing additional technical difficulties in the analysis. Since our primary interest lies in functions (or tensor fields) that decay as $d_g$ decreases, we instead use a smooth function $\theta$ that is $\varepsilon$-close to $d_g$ in the $C^0$-sense when $d_g$ is small. Note that $d_g$ is smooth away from $\Gamma$, so we only need to take extra care of the regions near $\Gamma$. This function, referred to as an \textbf{approximate distance} to $\Sigma'$, simplifies the analysis while preserving essential geometric properties.

\begin{figure}
\begin{center}
\begin{tikzpicture}
\draw[thick, fill=gray!40] (3,0) arc(0:180:3);
\draw[fill=white] (2.5,0) arc(0:180:2.5);
\draw[] (2.5,0) arc(180:265:0.5);
\draw[] (-2.5,0) arc(360:275:0.5);
\draw[dashed] (2.5,0) arc(180:95:0.5);
\draw[dashed] (-2.5,0) arc(0:85:0.5);

\fill[fill=gray!40] (3, 0) -- (2.5, 0) -- (2.5, 0) arc(180:265:0.5) -- (3, 0)  arc(360:350:3);
\fill[fill=gray!40] (-3, 0) -- (-2.5, 0) -- (-2.5,0) arc(360:275:0.5) -- (-3, 0)  arc(180:190:3);

\draw[blue] (-3,0) arc(180:360:3);

\fill (2.5,0) circle (1pt);
\fill (-2.5,0) circle (1pt);
\fill (3,0) circle (1pt);
\fill (-3,0) circle (1pt);

\draw (-2.7, -2.7) -- (-2.15,-2.1); \node[left] at (-2.7, -2.7) {$\Sigma$};
\draw (-2.7, 2.5) -- (-2.15,2.1); \node[left] at (-2.7, 2.5) {$\Sigma'$};
\draw (2.3, -0.8) -- (3.2,-1); \node[right] at (3.2,-1) {$\Omega_{r_0}$};
\draw (-3.5, 0) -- (-3.1,0); \node[left] at (-3.5, 0) {$\Gamma$};

\end{tikzpicture}
\caption{Collar neighborhood of $\Sigma'$}
\label{fig: weight function}
\end{center}
\end{figure}

Let us construct the approximate distance function $\theta$ in the following way. For $p\in\partial \Omega$, let $x = (x_1, \dots, x_{n-1}, \tau) = (\hat{x}, \tau)$ denote the Fermi coordinates around $p$. By abuse of notation, we will use $\hat{x}$ to denote points and a coordinate system on $\partial \Omega$. When $\hat x \in \partial \Omega$  is near $\Gamma$, we let $\hat x = (x_1, \dots, x_{n-2}, \sigma) = (x', \sigma)$ denote the Fermi coordinates around $\hat p\in\Gamma \subset \partial \Omega$, where $\sigma$ is the sign distance to $\Gamma$ such that $\{\hat x: \sigma < 0\}\subset \Sigma$ and $\{\hat x: \sigma > 0\}\subset \Sigma'$.
Let $\eta:\R_+\to [0, 1]$ be a smooth increasing cutoff function such that $\eta(t) = 0$ for $t\leq \frac12$ and $\eta(t) = 1$ for $t\geq 1$. Let $\varepsilon>0$ be small. 
Define the approximate distance function $\theta$ by
\begin{align*}
    \theta(\hat x, \tau) = \begin{cases}
        \eta\big(\tfrac{-\sigma}{\varepsilon \sqrt{\sigma^2 + \tau^2}}\big)\sqrt{\sigma^2 + \tau^2} + \Big(1- \eta\big(\tfrac{-\sigma}{\varepsilon\sqrt{\sigma^2 + \tau^2}}\big)\Big) \lvert \tau \rvert, & \mbox{if $\hat x = (x', \sigma)\in \Sigma$}\\
        \lvert \tau \rvert, & \mbox{if $\hat x\in \Sigma'$}
    \end{cases}
\end{align*}
in a neighborhood $V_\Omega\triangleq \{x\in \Omega\cup\Sigma: d_g(x)< 4 r_0\}$ of $\Sigma'$ for sufficiently small $r_0 > 0$ such that for any $x\in V_\Omega$,
\begin{align}
    (1 - \varepsilon) d_g(x) \leq  \theta(x) \leq (1 + \varepsilon) d_g(x), \label{eq: theta}
\end{align}
and moreover, for $j = 1, 2, \ldots,$
\begin{align}
    C_j^{-1} \leq \theta^{j-1}(x)\left| D^j \theta(x)\right| \leq C_j,\label{eq: d theta}
\end{align}
where $C_j$'s depend only on the cutoff function $\eta$ and the geometry of $\Gamma$ and $\partial \Omega$.
Therefore, for any $k>0$, the $\theta$-weighted H\"older norm is bounded:
\begin{align*}
    \|\theta\|_{C^k_{\theta}(V_\Omega)} \triangleq \sup_{x\in V_\Omega} \sum_{j=0}^k \theta^{j-1}(x) \left| D^j \theta(x)\right| \leq C_k,
\end{align*}
where $C_k$'s are uniform over $\mathcal{U}_0$. By shrinking $\mathcal{U}_0$, we can make this $\theta$ independent of the choice of $g\in \mathcal{U}_0$. Now the smaller neighborhood $\tilde{V}_\Omega \triangleq \{x\in \Omega\cup\Sigma: \theta(x)< r_0\}$ is foliated by smooth regular level sets of $\theta$. 
\begin{itemize}[leftmargin=*]
\item For $(\hat x, \tau)\in \tilde{V}_\Omega$ with $\hat x \in \Sigma'$, $\operatorname{Hess}(\theta) = \operatorname{Hess}(-\tau) \approx h_{\Sigma'}$ is uniformly bounded. 
\item For $(\hat x, \tau)\in \tilde{V}_\Omega$ with $\hat x \in \Sigma$, the level set of $\theta$ is approximately a cylinder around $\Gamma$, and the second fundamental form of the level set blows up at the rate of $\theta^{-1}$. Thus, unlike the uniform bound on $\operatorname{Hess}(d_g)$ in Corvino's setting \cite{C}, we can only require \eqref{eq: d theta} for $j = 2$.
\end{itemize}

We then define the weight function as a polynomial of $\theta$. Let $0 < r_1 < r_0$ be fixed. Define a smooth positive monotone function $\tilde{\rho}: (0, \infty) \rightarrow \mathbb{R}$ such that $\tilde{\rho}(t)=t$ for $t \in (0,r_1)$ and $\tilde{\rho}(t)=1$ for $t >r_0$. For $N > 0$, let $\rho$ be the positive function on $\Omega\cup\Sigma$ defined by
$$
\rho(x)=\left(\tilde{\rho} \circ \theta(x)\right)^{N}.
$$
We will eventually fix $N$ to be a suitably large number. We also denote $\Omega_\epsilon\subset\Omega$ as the region where $\theta > \epsilon$ for some small $\epsilon\ge0$, so that $\Omega_{r_0}\subset\Omega$ is a region where $\rho = 1$.

\subsection{Weighted Sobolev spaces}\label{weighted Sobolev spaces}
First, let us note a useful lemma from \cite{C-S} concerning the density of $H^k(\Omega)$ in $H_\rho^k(\Omega)$.
\begin{lem}[Corvino-Schoen \cite{C-S}*{Lemma 2.1}]\label{lem: density lemma}
Assume that $\rho$ is bounded from above. For $k \ge 1$, the subspace $H^k(\Omega)$ (and hence $C^\infty(\overline{\Omega})$) is dense in $H_\rho^k(\Omega)$.
\end{lem}
In fact, we can see that this density lemma holds in our setting as well.

Moreover, from (\ref{eq: theta}) and (\ref{eq: d theta}) we still have similar estimates on $\rho$ as in \cite{C-H}. In particular, we have the following proposition regarding weighted Sobolev norms.
\begin{prop}\label{ineq weighted Sobolev norms}
    Let $u\in H^k_{\rho}(\Omega)$. Then $u\rho^{\frac12} \in H^k(\Omega)$, and there is a constant $C$ uniform in $\mathcal{U}_0$ such that
$$
\left\|u \rho^{\frac12}\right\|_{H^k(\Omega)} \leq C \|u\|_{H^k_{\rho}(\Omega)}.
$$
\end{prop}
\begin{proof}
By direct computation and induction, we have
$$
\left|D^k\left(\rho^{\frac{1}{2}}\right)\right| \leq C N^k \rho^{\frac{1}{2}} \theta^{- k}.
$$
    For $0 < \theta(x) < r_1$, 
    \[
    \Delta \rho = N(N-1) \rho\left(\theta^{-2}|D\theta|^2 + (N-1)^{-1} \theta^{-1} \Delta \theta\right).
    \]
    If $r_2$ is sufficiently small, for $0 < \theta(x) \leq r_2$,
    $$
    \theta^{-2}|D\theta|^2 + (N-1)^{-1} \theta^{-1} \Delta \theta \ge  \theta^{-2}C_1^{-2} - (N-1)^{-1} C_2 \geq \frac12\theta^{-2}C^{-2}_1.
    $$
So we have 
$$
\frac{1}{2} N(N-1) C^{-2}_1\theta^{-2} \rho \leq \Delta \rho.
$$
Then similar argument as in \cite{C-H}*{Proposition 2.10} gives the desired estimate.
\end{proof}

Combined with the trace theorem \cite{L-M}*{p. 39} and the compactness theorem \cite{L-M}*{p. 99}, we have
\begin{cor}\label{trace cor}
Let $u\in H^k_{\rho}(\Omega)$. Then the trace map is a continuous linear map, denoted as
\begin{align*}
H^k(\Omega) & \rightarrow \prod_{j=0}^{k-1}H^{k- j -\frac12}(\partial\Omega);\\
u \rho^{\frac12} & \mapsto \prod_{j=0}^{k-1}\gamma_j(u \rho^{\frac12}).
\end{align*}
This map is surjective and there exists a continuous linear right inverse.
\end{cor}

\begin{cor}\label{compact cor}
Let $u_l\in H^k_{\rho}(\Omega)$ be a bounded sequence. Then for each $0\leq j\leq k-1$, there is a convergent subsequence of $\gamma_j(u_l \rho^{\frac12})$ in $H^{k- j - 1}(\partial\Omega)$.
\end{cor}

Let us further explore the properties of density and trace in weighted Sobolev spaces.
\begin{lem}\label{lem: density in L2}
    $C^\infty_c(\Omega\cup\Sigma)$ is dense in $L^2_\rho(\Omega)$.
\end{lem}
\begin{proof}
    Let $u\in L^2_\rho(\Omega)$, then $u\rho^{\frac12}\in L^2(\Omega)$. Since $C^\infty_c(\Omega\cup\Sigma)$ is dense in $L^2(\Omega)$, there exists a sequence $\varphi_j \in C^\infty_c(\Omega\cup\Sigma)$ such that $\varphi_j \to u\rho^{\frac12}$ in $L^2(\Omega)$. This implies that the sequence $\varphi_j \rho^{-\frac12}\in C^\infty_c(\Omega\cup\Sigma)$ converges to $u$ in $L^2_\rho(\Omega)$:
$$
        \left\|\varphi_j \rho^{-\frac12} - u\right\|_{L^2_\rho(\Omega)} = \left\|\varphi_j - u \rho^{\frac12}\right\|_{L^2(\Omega)} \longrightarrow 0.
$$
\end{proof}

\begin{prop}\label{lem: density cpt supp}
For any integer $k \geq 1$, $C^\infty_c(\Omega\cup\Sigma)$ is dense in $H^k_\rho(\Omega)$.
\end{prop}
\begin{proof}
By Lemma~\ref{lem: density lemma}, it suffices to show that $C^\infty(\overline{\Omega})$ lies in the closure of $C^\infty_c(\Omega\cup\Sigma)$ in $H^k_\rho(\Omega)$.
    Let $\varphi \in C^\infty(\overline{\Omega})$ and let $\eta:\R_+\to [0,1]$ be an increasing smooth cutoff function such that $\eta(t) = 0$ for $t\leq \frac14$ and $\eta(t) = 1$ for $t\geq \frac34$.
    
    Consider the sequence of functions $\varphi_j(x) \triangleq \eta\big(j \theta(x)\big) \varphi(x) \in C^\infty_c(\Omega\cup\Sigma)$. Then
$$
\|\varphi - \varphi_j\|^2_{L^2_\rho(\Omega)} = \int_\Omega (1 - \eta(j\theta))^2\varphi^2\rho \,d\mu_g \le \mu_g(K_j)\longrightarrow 0,
$$    
where $K_j$ is the support of $(1 - \eta(j\theta))$. It remains to show that the covariant derivatives of $(\varphi - \varphi_j)$ up to order $k$ converge to $0$ in $L^2_\rho(\Omega)$. For any multi-index $\alpha$ with $\lvert \alpha \rvert \leq k$,
$$
D_\alpha (\varphi_j - \varphi) = D_\alpha \left[(1 - \eta(j\theta))\varphi\right] = (1 - \eta(j\theta)) D_\alpha \varphi + \sum_{\lvert \beta \rvert = 1}^{\lvert \alpha \rvert} C_\beta D_{\beta}(1 - \eta(j\theta))  \cdot D_{(\alpha-\beta)} \varphi.  
$$
    Since $\varphi \in C^\infty(\overline\Omega)$, $D_\alpha \varphi$ and $D_{(\alpha-\beta)} \varphi$ are uniformly bounded. Thus the first term $(1 - \eta(j\theta)) D_\alpha \varphi$ converges to $0$ in $L^2_\rho(\Omega)$.  On the other hand, 
$$
D_{\beta}(1 - \eta(j\theta)) =   - \eta^{(\lvert \beta \rvert)}(j\theta)\cdot j^{\lvert \beta \rvert}\cdot D_{\beta}\theta.
$$ 
Using the fact that $\eta^{(\lvert \beta \rvert)}(j\theta)$ is bounded and supported in $\{x:\frac{1}{4j} \leq \theta(x) \leq \frac{3}{4j} \}$ and using estimates \eqref{eq: d theta}, we have for $N>4k-2$,
    \begin{align*}
        & \int_\Omega C^2_\beta \lvert D_{(\alpha-\beta)} \varphi\rvert^2 \lvert \eta^{(\lvert \beta \rvert)}(j\theta)\rvert^2 j^{2\lvert \beta \rvert} \lvert D_{\beta}\theta \rvert^2 \rho \,d\mu_g\\
        \leq & C \cdot \sup \lvert D_{(\alpha-\beta)} \varphi\rvert^2\cdot j^{2\lvert \beta \rvert} \int_{\{x:\frac{1}{4j} \leq \theta(x) \leq \frac{3}{4j} \}} \theta^{2-2\lvert \beta \rvert}\cdot\theta^N d\mu_g\\
        \leq & C \cdot \sup \lvert D_{(\alpha-\beta)} \varphi\rvert^2\cdot j^{4\lvert \beta \rvert - 2 - N} \longrightarrow 0,
    \end{align*}
where $C$ depends only on $\mu_g(\Omega)$, the constants in \eqref{eq: d theta} and the cutoff function $\eta$.
    Combining all the above estimates implies that $\|D_\alpha(\varphi_j - \varphi)\|_{L^2_\rho(\Omega)}$ converges to $0$ for all $\lvert \alpha \rvert \leq k$. Therefore, $\varphi_j$ converges to $\varphi$ in $H^k_\rho(\Omega)$.
\end{proof}

\begin{cor}\label{trace Sigma' zero}
For any integer $k\geq 1$, if $u\in H^k_\rho(\Omega)$, then $u\rho^{\frac12} \in H^k(\Omega, \Sigma')$. Thus for $0\leq j\leq k-1$,
$$
\gamma_j\left(u \rho^{\frac12}\right) = 0\quad\text{on }\Sigma'.
$$
\end{cor}
\begin{proof}
    Let $u\in H^k_\rho(\Omega)$. By Proposition~\ref{lem: density cpt supp}, there exists a sequence $\varphi_j\in C^\infty_c(\Omega\cup\Sigma)$ such that $\varphi_j\to u$ in $H^k_\rho(\Omega)$. By Proposition~\ref{ineq weighted Sobolev norms},
$$
\left\|\varphi_j \rho^{\frac{1}{2}} - u \rho^{\frac{1}{2}}\right\|_{H^k(\Omega)} \leq C \left\|\varphi_j - u\right\|_{H^k_\rho(\Omega)} \longrightarrow 0.
$$
Since $\varphi_j \rho^{\frac12} \in C^\infty_c(\Omega\cup\Sigma)$, this shows that $u\rho^{\frac12} \in H^k(\Omega, \Sigma')$.
\end{proof}

\begin{prop}\label{trace H0rho}
For any integer $k\geq 1$, if $u\in H^k_{0, \rho}(\Omega)$, then $u\rho^{\frac12} \in H_0^k(\Omega)$. Moreover, for $0\leq j\leq k-1$,
$$
\gamma_j\left(u \right) = 0\quad\text{on }\Sigma.
$$
Thus, the trace map extends by continuity to a continuous linear map of 
$$\left.\gamma_{j}\right|_\Sigma(u) : H^k_{0, \rho}(\Omega) \rightarrow H^{k- j -\frac12}(\Sigma).$$
\end{prop}
\begin{proof}
    Let $u\in H^k_{0, \rho}(\Omega)$. Then there exists a sequence $\varphi_j\in C^\infty_c(\Omega)$ such that $\varphi_j\to u$ in $H^k_\rho(\Omega)$. By Proposition~\ref{ineq weighted Sobolev norms},
$$
\left\|\varphi_j \rho^{\frac{1}{2}} - u \rho^{\frac{1}{2}}\right\|_{H^k(\Omega)} \leq C \left\|\varphi_j - u\right\|_{H^k_\rho(\Omega)} \longrightarrow 0.
$$
Since $\varphi_j \rho^{\frac12} \in C^\infty_c(\Omega)$, this shows that $u\rho^{\frac12} \in H_0^k(\Omega)$, and for $0\leq j\leq k-1$, $\gamma_j(u \rho^{\frac12}) = 0$ on $\Sigma$.

Note that the condition $\left.\gamma_0\right|_\Sigma(u \rho^{\frac12}) = 0$ immediately implies $\left.\gamma_0\right|_\Sigma(u) = 0$. Now, assume by induction that $\left.\gamma_{j}\right|_\Sigma(u) = 0$ for all $0\leq j\leq k-2$. Consider the next term:
\begin{align*}
0 & = \left.\gamma_{k-1}\right|_\Sigma\left(u \rho^{\frac12}\right) = \left.D^{k-1}_\nu\left(u \rho^{\frac12}\right)\right|_\Sigma\\
	& = \left.D^{k-1}_\nu u \cdot \rho^{\frac12}\right|_\Sigma + \left.C^{k-1}_1 D^{k-2}_\nu u \cdot D_\nu\left(\rho^{\frac12}\right)\right|_\Sigma + \dots + \left.u \cdot D^{k-1}_\nu\left(\rho^{\frac12}\right)\right|_\Sigma\\
	& = \left.D^{k-1}_\nu u \cdot \rho^{\frac12}\right|_\Sigma,
\end{align*}
where we used the inductive hypothesis and $|D^{k}_\nu(\rho^{\frac12})|\le CN^k\rho^{\frac12}\theta^{-k}\le CN^k\theta^{\frac{N}2-k}$.
Thus $\left.\gamma_{k-1}\right|_\Sigma(u) = 0$.

Therefore, by induction, $\gamma_j\left(u \right) = 0$ on $\Sigma$ for all $0\leq j\leq k-1$.
\end{proof}


Having established the inequality $\|u \rho^{\frac12}\|_{H^k(\Omega)} \leq C \|u\|_{H^k_{\rho}(\Omega)}$ for all $u \in H^k_{\rho}(\Omega)$ in Proposition~\ref{ineq weighted Sobolev norms}, we now turn to the question of norm equivalence. This leads us to investigate the reverse inequality. Specifically, we would like to show that if $w \in H^1(\Omega, \Sigma')$, then $w \rho^{-\frac12}\in H^1_{\rho}(\Omega)$, and
$$
\left\|w \rho^{-\frac12}\right\|_{H^1_{\rho}(\Omega)} \le C\|w\|_{H^1(\Omega)}.
$$
Note that
\begin{align*}
\left\|w \rho^{-\frac12}\right\|^2_{H^1_{\rho}(\Omega)} & = \int_\Omega w^2d \mu_g + \int_\Omega \left|D \left(w \rho^{-\frac12}\right)\right|^2\rho\,d \mu_g\\
	& = \int_\Omega w^2d \mu_g + \int_\Omega \left|D w - \frac{N}2w\theta^{-1}D\theta\right|^2d \mu_g\\
	& \le \int_\Omega w^2d \mu_g + 2\int_\Omega |D w|^2d \mu_g + \frac{N^2(1+\epsilon)^2}2\int_\Omega w^2\theta^{-2}d \mu_g.
\end{align*}
So it suffices to show 
$$\int_\Omega w^2\theta^{-2}d \mu_g \le C\|w\|^2_{H^1(\Omega)},$$
or equivalently, 
\begin{equation}\label{Hardy}
\int_{\Omega} u^{2} \theta^{-2} \rho \leq C\left\|u \rho^{\frac12}\right\|_{H^{1}(\Omega)}^{2}
\end{equation}
for $u \in H^1_{\rho}(\Omega)$.

The validity of inequality (\ref{Hardy}) is intimately connected to Hardy's inequality and, unfortunately, is not guaranteed. It depends critically on the geometric relationship between the approximate distance $\theta$ and the domain's boundary. Specifically:
\begin{itemize}[leftmargin=*]
\item If $\theta$ is equivalent to the distance to the \textbf{entire boundary} $\partial\Omega$, then (\ref{Hardy}) holds for all $u \in H^1_{\rho}(\Omega)$. This is the setting considered by Corvino \cite{C}, which differs from ours.
\item Likewise, (\ref{Hardy}) holds for $u \in H^1_{\rho}(\Sigma)$ if $\theta$ is equivalent to the distance to the entire boundary $\partial\Sigma$.
\item In contrast, if $\theta$ is equivalent to the distance to only a \textbf{partial boundary} $\Sigma'$, then (\ref{Hardy}) may fail for general $u \in H^1_{\rho}(\Omega)$. The failure occurs near the boundary $\Gamma \triangleq \partial\Sigma'$ of the partial boundary itself. This is a codimension-2 singularity, which induces the critical case where Hardy's inequality no longer holds (see \cite{Ku84}*{p. 35}). Nevertheless, the inequality remains valid for functions $u \in H^1_{0, \rho}(\Omega)$ that are supported away from the full boundary $\partial\Omega$.
\end{itemize}

The proofs of the preceding three results follow analogous arguments, relying on the following Hardy's inequality for $F\in C^1_c((0, \infty))$:
$$
\int_0^\infty t^{-2}F^2(t)dt \le 4\int_0^\infty \left|F'(t)\right|^2dt.
$$
We now summarize results (ii) and (iii) into the following propositions for subsequent reference. For brevity, we prove only Proposition~\ref{inverse ineq weighted Sobolev norms 2}.


\begin{prop}\label{inverse ineq weighted Sobolev norms 1}
Let $u\in H^k_{\rho}(\Sigma)$. Then there is a constant $C$ uniform in $\mathcal{U}_0$ such that
$$
C^{-1} \|u\|_{H^k_{\rho}(\Sigma)} \le \left\|u \rho^{\frac12}\right\|_{H^k(\Sigma)} \leq C \|u\|_{H^k_{\rho}(\Sigma)}.
$$
\end{prop}

\begin{prop}\label{inverse ineq weighted Sobolev norms 2}
Let $u\in H^k_{0, \rho}(\Omega)$. Then there is a constant $C$ uniform in $\mathcal{U}_0$ such that
$$
C^{-1} \|u\|_{H^k_{\rho}(\Omega)} \le \left\|u \rho^{\frac12}\right\|_{H^k(\Omega)} \leq C \|u\|_{H^k_{\rho}(\Omega)}.
$$
\end{prop}
\begin{proof}
In this proof, we denote by $\bar{d}$ the distance to the \textbf{entire boundary} $\partial\Omega$ in $\tilde{V}_\Omega$. Note that $\bar{d}\le\theta$ in $\tilde{V}_\Omega$.

Let us denote $w=u \rho^{\frac12}$, and by density, we may assume $w \in C^\infty_c(\Omega)$ with $\operatorname{supp}(w)\subset\{\bar{d}\ge\epsilon\}$. For fixed $\delta>0$ depending only on $\Omega$, we have
\begin{equation}\label{Hardy est}
\int_{\{\bar{d}\ge\delta\}} w^2\theta^{-2}d \mu_g \le \delta^{-2}\int_{\{\bar{d}\ge\delta\}} w^2d \mu_g\le \delta^{-2}\int_{\Omega} w^2d \mu_g.
\end{equation}
Thus, it suffices to consider the region where $\operatorname{supp}(w)\subset\{\epsilon \le \bar{d} \le \delta\}$. 

By the co-area formula,
$$
\int_{\{\epsilon \le \bar{d} \le \delta\}} w^2\theta^{-2}d \mu_g = \int^\delta_\epsilon\int_{\{\bar{d}=t\}}w^2\theta^{-2}d\sigma_t dt \le \int^\delta_\epsilon t^{-2}\left(\int_{\{\bar{d}=t\}}w^2d\sigma_t\right) dt,
$$
where $d\sigma_t$ denotes the surface measure on the level set $\{\bar{d}=t\}$. Denote $F(t) = \left(\int_{\{\bar{d}=t\}}w^2d\sigma_t\right)^{\frac12}\in C^1_c((0, \delta])$. Then
$$
F'(t) = \frac12\left(\int_{\{\bar{d}=t\}}w^2d\sigma_t\right)^{-\frac12}\left(\int_{\{\bar{d}=t\}} 2w\p_t w d \sigma_t+\int_{\{\bar{d}=t\}} w^2 H_t d \sigma_t\right),
$$
where $H_t$ denotes the mean curvature of the level set $\{\bar{d}=t\}$. Note that $H_t$ is uniformly bounded in $\{\epsilon \le \bar{d} \le \delta\}$. Hence, by Hardy’s inequality,
\begin{align*}
& \int_{\{\epsilon \le \bar{d} \le \delta\}} w^2\theta^{-2}d \mu_g\\
\le & \int^\delta_\epsilon \left(\int_{\{\bar{d}=t\}}w^2d\sigma_t\right)^{-1}\left(\int_{\{\bar{d}=t\}} 2w\p_t w d \sigma_t+\int_{\{\bar{d}=t\}} w^2 H_t d \sigma_t\right)^2 dt\\
\le & 2\int^\delta_\epsilon \left(\int_{\{\bar{d}=t\}}w^2d\sigma_t\right)^{-1} \left[\left(\int_{\{\bar{d}=t\}} 2w\p_t w d \sigma_t\right)^2+\left(\int_{\{\bar{d}=t\}} w^2 H_t d \sigma_t\right)^2\right] dt\\
\le & 2\int_\epsilon^\delta \left(\int_{\{\bar{d}=t\}} 4\left(\partial_t w\right)^2+w^2 H_t^2 d \sigma_t\right) d t\\
\le & C\int_{\{\epsilon \le \bar{d} \le \delta\}}\left(|D w|^2 + w^2 \right) d \mu_g\\
\le & C\int_\Omega\left(|D w|^2 + w^2 \right) d \mu_g.
\end{align*}
Combining this with (\ref{Hardy est}), we conclude
$$\int_{\Omega} w^2\theta^{-2}d \mu_g \le C\int_\Omega\left(w^2 + |D w|^2 \right) d \mu_g,$$
which is equivalent to inequality (\ref{Hardy}).

Estimates for the higher-order derivatives can be obtained by induction.
\end{proof}

\begin{figure}
\begin{center}
\begin{tikzpicture}[>=stealth] 

    \draw[->] (-3.5, 0) -- (6.5, 0);
    \node[below] at (6.5, 0) {$\Sigma'$};
    \node[below] at (-3.5, 0) {$\Sigma$};

    \draw[->] (0, -0.5) -- (0, 4.2) node[left] {$\Omega$};

    \fill (0,0) circle (2pt) node[below left] {$\Gamma$};

    \draw[thick] (-3,0) arc (180:90:3); 
    \draw[thick] (0,3) -- (6.3,3) node[right] {$\theta$}; 

    \draw[dashed, thick] (-3,1.5) -- (6.3,1.5) node[right] {$\bar{d}$};

\end{tikzpicture}
\caption{Level sets of $\theta$ and $\bar{d}$ in Fermi coordinates}
\label{fig: level sets}
\end{center}
\end{figure}

This proof also shows that inequality (\ref{Hardy}) fails for general $u\in C^\infty_c(\Omega\cup\Sigma)$ in the partial boundary case. The application of Hardy’s inequality requires the foliation $\{\theta = t\}$, but the mean curvature of these level sets blows up near $\Gamma$.

A direct consequence is a precise characterization of the space $H^k_{0, \rho}(\Omega)$.
\begin{cor}[Characterization of $H^k_{0, \rho}(\Omega)$]\label{H0rho}
For any integer $k\geq 1$, 
$$H^k_{0, \rho}(\Omega) = \left\{u\in H^k_{\rho}(\Omega): \left.\gamma_{j}\right|_\Sigma\left(u\rho^{\frac12}\right) = 0,\,\text{ for }0\leq j\leq k-1\right\}.$$
\end{cor}
\begin{proof}
Since Proposition~\ref{trace H0rho} establishes one inclusion, it suffices to prove the converse:
$$H^k_{0, \rho}(\Omega) \supset \left\{u\in H^k_{\rho}(\Omega): \left.\gamma_{j}\right|_\Sigma\left(u\rho^{\frac12}\right) = 0,\,\text{ for }0\leq j\leq k-1\right\}.$$

Suppose $u \in H^k_{\rho}(\Omega)$ satisfies $\left.\gamma_{j}\right|_\Sigma(u\rho^{\frac12}) = 0$ for $0 \le j \le k-1$. Corollary~\ref{trace Sigma' zero} implies $\left.\gamma_{j}\right|_{\Sigma'}(u \rho^{\frac12}) = 0$. Hence, the trace $\gamma_{j}(u \rho^{\frac12})$ vanishes on the entire boundary, yielding
$$
u \rho^{\frac12} \in H^k_{0}(\Omega).
$$

Consequently, there exists a sequence $w_j\in C^\infty_c(\Omega)$ such that $w_j\to u \rho^{\frac12}$ in $H^k(\Omega)$. By the proof of Proposition~\ref{inverse ineq weighted Sobolev norms 2}, we have
$$
\left\|w_j \rho^{-\frac12} - u\right\|_{H^k_{\rho}(\Omega)} \le C\left\|w_j - u \rho^{\frac12}\right\|_{H^k(\Omega)}\longrightarrow0.
$$
Since each $w_j \rho^{-\frac12}$ belongs to $C^\infty_c(\Omega)$, this implies $u \in H^k_{0,\rho}(\Omega)$.
\end{proof}

Finally, we would like to remark that this norm equivalence is a principal reason for choosing the power weight over the exponential weight considered in \cite{C-H}, as the same argument fails for exponential weights.

\subsection{Weighted H\"older spaces}\label{weighted Holder space}
In this subsection, we define the weighted H\"older space and discuss several of its key properties. Following the idea of \cite{C-D}, we consider weighted H\"older norms in small balls $B_{\phi(x)}(x)$ that cover $\Omega$. The weight function $\phi = \phi_g$ satisfies certain uniform estimates for all metrics $g$ in a $\mathcal{C}^{k+4}(\overline{\Omega})$ neighborhood $\mathcal{U}_0$. Recall the neighborhood $V_\Omega$ from Section~\ref{Weight function}, and suppose we have chosen a suitable $N$. 

As noted in \cite{C-H}*{Proposition 2.11}, the weight function enjoys useful interior properties. With minor modifications, these properties extend naturally to our current setting:
\begin{prop}\label{C-H estimate}
For $g \in \mathcal{U}_0$, we define $\phi(x) = \frac12\theta(x)$ in $V_\Omega$. There exists a constant $C > 0$, uniform across $\mathcal{U}_0$, such that we can extend $\phi$ to $\Omega$ with $0 < \phi < 1$ and with the following properties.\\
(i) $\phi$ has a positive lower bound on $\Omega\backslash V_\Omega$ uniformly in $g\in\mathcal{U}_0$, and for each $x$, $\phi(x)<d(x)$, so that $\overline{B_{\phi(x)}(x)}\subset\Omega$.\\
(ii) For $x\in\Omega$ and $k\le m$, we have $\left|\phi^{k} \rho^{-1} D^{k} \rho\right| \leq C.$\\
(iii) For $x\in\Omega$ and for $y\in B_{\phi(x)}(x)$, we have
\begin{align*}\begin{aligned}
    C^{-1}\rho(y)&\le\rho(x)\le C\rho(y)\\
    C^{-1}\phi(y)&\le\phi(x)\le C\phi(y).
\end{aligned}\end{align*}
\end{prop}

Let $r, s\in \mathbb{R}$ and $\varphi = \phi^r\rho^s$. For $u \in C_{\mathrm{loc}}^{k, \alpha}(\Omega)$, we define the weighted H\"older norm following \cite{C-H} as
$$
\|u\|_{C_{\phi, \varphi}^{k, \alpha}(\Omega)} \triangleq \sup _{x \in \Omega}\left(\sum_{j=0}^{k} \varphi(x) \phi^{j}(x)\left\|D_{g}^{j} u\right\|_{C^{0}\left(B_{\phi(x)}(x)\right)}+\varphi(x) \phi^{k+\alpha}(x)\left[D_{g}^{k} u\right]_{0, \alpha ; B_{\phi(x)}(x)}\right).
$$
The weighted H\"older space $C_{\phi, \varphi}^{k, \alpha}(\Omega)$ consists of all $C_{\mathrm{loc}}^{k, \alpha}(\Omega)$ functions or tensor fields for which this norm is finite. If $u\in C_{\phi, \varphi}^{k, \alpha}(\Omega)$, then $u$ is dominated by $\varphi^{-1}$ in the sense that $u = O(\varphi^{-1})$ and $D^j u = O(\varphi^{-1}\phi^{-j})$ near the boundary.

\begin{prop}[Corvino-Huang \cite{C-H}]\label{operator Holder}
We have some properties for weighted H\"older spaces:
\begin{itemize}
\item[(i)] The differentiation is a continuous map:
$$
D: C_{\phi, \varphi}^{k, \alpha}(\Omega) \to C_{\phi, \phi\varphi}^{k-1, \alpha}(\Omega).
$$
\item[(ii)] For $u\in C_{\phi, \varphi}^{k, \alpha}(\Omega)$, $v\in C^{k, \alpha}(\overline{\Omega})$, we have $uv\in C_{\phi, \varphi}^{k, \alpha}(\Omega)$ and $C=C(k)$ with
$$
\|u v\|_{C_{\phi, \varphi}^{k, \alpha}(\Omega)} \leq C\|u\|_{C_{\phi, \varphi}^{k, \alpha}(\Omega)}\|v\|_{C^{k, \alpha}(\overline{\Omega})}.
$$
\item[(iii)] The multiplication by $\rho$ is a continuous map from $C_{\phi, \varphi}^{k, \alpha}$ to $C_{\phi, \varphi\rho^{-1}}^{k, \alpha}$.
\end{itemize}
\end{prop}

We now extend this framework to the setting $\Omega\cup\Sigma$, particularly near the boundary $\Sigma$. Again, we consider weighted H\"older norms in small balls $B_{\phi(x)}(x)$ that cover $\Omega\cup\Sigma$. It is easy to see that the same properties continue to hold, with $\Omega$ replaced by $\Omega\cup\Sigma$, and $B_{\phi(x)}(x)$ replaced by $B^+_{\phi(x)}(x) \triangleq B_{\phi(x)}(x)\cap(\Omega\cup\Sigma)$. 

For $u \in C_{\mathrm{loc}}^{k, \alpha}(\Omega\cup\Sigma)$, we define its weighted H\"older norm $\|u\|_{C_{\phi, \varphi}^{k, \alpha}(\Omega\cup\Sigma)}$ by
$$
\sup _{x \in \Omega\cup\Sigma}\left(\sum_{j=0}^{k} \varphi(x) \phi^{j}(x)\left\|D_{g}^{j} u\right\|_{C^{0}\left(B^+_{\phi(x)}(x)\right)}+\varphi(x) \phi^{k+\alpha}(x)\left[D_{g}^{k} u\right]_{0, \alpha ; B^+_{\phi(x)}(x)}\right).
$$

\begin{rmk}
With this norm, the properties stated in Propositions \ref{C-H estimate} and \ref{operator Holder} continue to hold near the boundary $\Sigma$.
\end{rmk}

For integer $k\ge0$, we use the following Banach spaces for functions or tensor fields:
\begin{align*}
    \mathcal{B}_{0}(\Omega) & =C_{\phi, \phi^{4+\frac{n}{2}} \rho^{-\frac{1}{2}}}^{k, \alpha}(\Omega\cup\Sigma) \cap \left(H_{\nu, \rho}^{2}(\Omega)\right)^*,\\
   \mathcal{B}_{2}(\Omega) & =C_{\phi, \phi^{2+\frac{n}{2}} \rho^{-\frac{1}{2}}}^{k+2, \alpha}(\Omega\cup\Sigma) \cap L_{\rho^{-1}}^{2}(\Omega),\\
   \mathcal{B}_{4}(\Omega) & =C_{\phi, \phi^{\frac{n}{2}} \rho^{\frac{1}{2}}}^{k+4, \alpha}(\Omega\cup\Sigma) \cap H_{\rho}^{2}(\Omega),
\end{align*}
with the Banach norms:
\begin{align*}
    \|u\|_{\mathcal{B}_{0}(\Omega)} & =\|u\|_{C^{k, \alpha}_{\phi, \phi^{4+\frac{n}{2}}\rho^{-\frac{1}{2}}}(\Omega\cup\Sigma)}+\|u\|_{\left(H_{\nu, \rho}^{2}(\Omega)\right)^*},\\
    \|u\|_{\mathcal{B}_{2}(\Omega)} & =\|u\|_{C^{k+2, \alpha}_{\phi, \phi^{2+\frac{n}{2}}\rho^{-\frac{1}{2}}}(\Omega\cup\Sigma)}+\|u\|_{L_{\rho^{-1}}^{2}(\Omega)},\\
     \|u\|_{\mathcal{B}_{4}(\Omega)} & =\|u\|_{C^{k+4, \alpha}_{\phi, \phi^{\frac{n}{2}}\rho^{\frac{1}{2}}}(\Omega\cup\Sigma)}+\|u\|_{H_{\rho}^{2}(\Omega)}.
\end{align*}
It is clear that these Banach spaces contain the smooth functions with compact supports in $\Omega\cup\Sigma$. We can similarly define Banach spaces $\mathcal{B}_l(\Sigma)$, but note that the weight and the regularity may be different on the boundary. 
Throughout this paper, we will only use 
$$
\mathcal{B}_{1}(\Sigma) =C_{\phi, \phi^{3+\frac{n}{2}} \rho^{-\frac{1}{2}}}^{k+1, \alpha}(\Sigma) \cap \mathcal{D}^*.
$$
We will clarify the definitions of the spaces $\mathcal{D}$ and $\mathcal{D}^*$ more clear in Section~\ref{sec4}.

We will sometimes use the product norms, e.g.
$$
    \|(f, \psi)\|_{\mathcal{B}_{0}(\Omega)\times\mathcal{B}_{1}(\Sigma)} = \|f\|_{\mathcal{B}_{0}(\Omega)} + \|\psi\|_{\mathcal{B}_{1}(\Sigma)}.
$$


\subsection{A basic lemma}
Finally, we would like to introduce a basic lemma, which will be used frequently in later discussions. Also recall in \cite{C}*{Proposition 2.5} Corvino found that $H_{\mathrm{loc}}^{2}(\Omega)$ elements in $\operatorname{ker} L^*$ are actually in $C^2(\overline{\Omega})$.
\begin{lem}\label{no kernel}
There is no non-trivial solution in $H_{\operatorname{loc}}^{2}(\Omega)$ to the following equations:
$$
\left\{\begin{aligned} 
L^*u & = 0 \quad&\text{in } &\Omega\\ 
u = u_\nu & = 0 &\text{on } &\Sigma.
\end{aligned}\right.
$$
\end{lem}
\begin{proof}
Define the set
$$
Z = \{p\in\Omega\cup\Sigma: u(p)=du(p)=0\}.
$$
By the boundary conditions, $\Sigma\subset Z$, so $Z$ is non-empty. We will prove that $Z$ is both open and closed in the connected domain $\Omega\cup\Sigma$, forcing $Z$ to be the entire domain and thus $u\equiv0$.

It is easy to see that $Z$ is closed. Now let $p\in Z$. If $p\in\Sigma$, let $\gamma$ be the unit-speed geodesic starting at $p$ and normal to $\Sigma$. If $p\in\Omega$, let $\gamma$ be any unit-speed geodesic starting at $p$. Consider the function $f(t) = u(\gamma(t))$ along this geodesic. Then $f(t)$ satisfies the linear second order differential equation
$$
\begin{aligned}
f^{\prime \prime}(t) &=\left.(\operatorname{Hess}_g u)\right|_{\gamma(t)} \left(\gamma^{\prime}(t), \gamma^{\prime}(t)\right) \\
&=\left[\left(\operatorname{Ric}_g-\frac{R_g}{n-1} g\right)\left(\gamma^{\prime}(t), \gamma^{\prime}(t)\right)\right] f(t).
\end{aligned}
$$
Since $p\in Z$, the initial conditions are $f(0)=u(p)=0 \text { and } f^{\prime}(0)=du(p) \cdot \gamma^{\prime}(0)=0$. By the uniqueness theorem, the only solution is the trivial one: $f(t)\equiv0$. Therefore, $u$ vanishes identically along every geodesic starting at $p$. Since an open neighborhood of $p$ can be covered by such geodesics, it follows that $u\equiv0$ in that neighborhood. Hence, $Z$ is open.
\end{proof}

\begin{rmk}
For some results, we will work with metrics $g$ in a neighborhood of a metric $g_0$. For the H\"older spaces $C^{k,\alpha}(\Omega\cup\Sigma)$ that appear in these results, we have the flexibility to construct the norm using any chosen metric, such as $g_0$, as any pair of such norms will be equivalent. Regarding the weighted norms defining $\mathcal{B}_{l}(\Omega)$, we could work with norms based on the fixed metric $g_0$, or alternatively, when working at a metric $g$, we could employ weighted norms tailored to $g$. It's worth noting that while the weighted norms for two different metrics may not be equivalent, relevant estimates below are appropriately formulated so long as we consistently utilize the same metric in defining the weighted norms on both sides of the inequality. Therefore, we just need to use a metric consistently in this manner when interpreting the stated results.

We also observe from the analysis of the proofs presented in \cite{C-H} (with some modification regarding the $\Sigma$ part) that we could actually work with respect to a fixed weight function, such as $\rho_{g_0}$. Although it appears intuitive to utilize $\rho_{g}$ when working at a metric $g$, what's essential concerning the weight function is that it behaves as noted in \cite{C-H} in terms of a defining function for $\Sigma'$, such as $\theta_{g_0}$.
\end{rmk}

\section{Generic conditions}\label{sec3}
As previously mentioned, we must assume certain generic conditions to enable localized deformations. Now let us take a closer look at these conditions.

\subsection{The kernels of $\Psi^*$ and $\Phi^*$}\label{sec3.1}
Let $(M^n,g)$ be a Riemannian manifold ($n \geq 2$), and let $\Omega \subset M$ be a domain such that its closure $\overline{\Omega} = \Omega \cup \partial\Omega$ is a compact smooth $n$-manifold with smooth boundary $\Sigma = \partial\Omega$.

Define the operator $\Psi: \mathcal{C}^\infty(\Omega\cup\Sigma) \longrightarrow C^{\infty}(\Omega\cup\Sigma) \times C^{\infty}(\Sigma)$ by
$$
\Psi(a) \triangleq (L(a), B(a)),
$$
where $B(a) = 2\dot{H}(a)$, and let $\Psi^*$ denote its formal $L^2$-adjoint. We investigate the surjectivity of $\Psi$, which is equivalent to the solvability of the boundary value problem
\begin{align}\label{BVP}
\left\{\begin{aligned} 
L(a) & = f \quad & \text{in } &\Omega\\ 
B(a) & = \psi & \text{on } &\Sigma,
\end{aligned}\right.
\end{align}
for all smooth data $(f, \psi)$.

Let us first characterize the space $\operatorname{ker}\Psi^*$. Let $(u,v)\in C^{\infty}(\Omega\cup\Sigma) \times C^{\infty}(\Sigma)$ be in $\operatorname{ker}\Psi^*$. Then for any $a\in \mathcal{C}^\infty(\Omega\cup\Sigma)$,
\begin{align}\label{kernel phi}
0 & = \left<a, \Psi^*(u,v)\right> = \left<\Psi(a), (u,v)\right>\nonumber\\
& = \left<(L(a), B(a)), (u, v)\right>\nonumber\\
& = \int_\Omega L(a)u  \, d \mu_g+ \int_\Sigma B(a)v \, d\sigma_g\nonumber\\
& = \int_\Omega \left<a, L^*u\right> d \mu_g+  \int_\Sigma \left<a, u_\nu \hat{g}-u h\right>_{\hat{g}} d\sigma_g + \int_\Sigma B(a)(v-u) \, d\sigma_g,
\end{align}
where we used Proposition \ref{IBP} for the integration by parts.
\begin{itemize}[leftmargin=*]
\item[1.] If we take $a$ to be compactly supported in $\Omega$, then the boundary term vanishes, and we are left with
$$
0 = \int_\Omega \left<a, L^*u\right> d \mu_g.
$$
Since this holds for all such smooth, compactly supported tensors $a$, it follows that $L^*u = 0$ in $\Omega$. 
\item[2.] Now, suppose $a = 0$ on $\Sigma$. It follows that
$$
0=\int_\Sigma B(a)(v-u) \, d\sigma_g.
$$
From Proposition~\ref{linearized RH}, we recall that
$$
B(a) = 2\dot{H}(a) = \nu(\operatorname{tr}_\Sigma a).
$$
To proceed, we show that for any function $\varphi\in C^{\infty}(\Sigma)$, there is a tensor $a\in \mathcal{C}^\infty(\Omega\cup\Sigma)$ such that $a = 0$ and $B(a) = \varphi$ on $\Sigma$. 

\textbf{Construction:} Let $d\triangleq \operatorname{dist}(\cdot, \Sigma)$ denote the distance to $\Sigma$. Note that $d=0$ and $\nu(d)=1$ on $\Sigma$. Let $(x', x_n)$ be Fermi coordinates near a point $p \in \Sigma$, with $x_n = d$. With a minor abuse of notation, we may take 
$$a(x)=\frac1{n-1}\varphi(x') (d\cdot g)(x)$$ 
in a neighborhood of $\Sigma$, and extend it smoothly to the rest of $\Omega\cup\Sigma$. By construction, $a = 0$ on $\Sigma$, and
$$
B(a) = \nu(\operatorname{tr}_\Sigma a) = \nu(\varphi d) = \varphi\qquad\text{on }\Sigma. 
$$

Therefore, we conclude that $u = v$ on $\Sigma$.
\item[3.] Finally, from (\ref{kernel phi}) we have
$$
0 = \int_\Sigma \left<a, u_\nu \hat{g}-u h\right>_{\hat{g}} d\sigma_g,
$$
which means $u_\nu \hat{g}-u h=0$ on $\Sigma$.
\end{itemize}
In conclusion, $\operatorname{ker}\Psi^*$ is the space of solutions to the following equations:
\begin{align*}
\left\{\begin{aligned}
L^*u & = 0 \quad &\text{in } &\Omega\\ 
u_\nu \hat{g}-u h & = 0 &\text{on } &\Sigma\\
u & = v &\text{on } &\Sigma.
\end{aligned}\right.
\end{align*}

On the other hand, from Proposition \ref{IBP} the formal adjoint problem of (\ref{BVP}) with respect to Green's formula is
\begin{align*}
\left\{\begin{aligned}
L^*u & = a \quad &\text{in } &\Omega\\ 
u_\nu \hat{g}-u h & = b &\text{on } &\Sigma.
\end{aligned}\right.
\end{align*}
This leads us to consider the operator $\Phi^*: C^{\infty}(\Omega\cup\Sigma) \longrightarrow \mathcal{C}^\infty(\Omega\cup\Sigma) \times \mathcal{C}^\infty(\Sigma)$ defined by
$$
\Phi^*(u) \triangleq (L^*u, u_\nu \hat{g}-u h).
$$ 
Although $\Phi^*$ differs from $\Psi^*$, their kernels are closely related (see \cite{L-M}*{p. 162, Proposition 5.4}). In fact, their kernels are almost identical and can be characterized as the space of solutions to the following equations:
\begin{align}\label{generic}
\left\{\begin{aligned} 
L^*u & = 0 \quad & \text{in } &\Omega\\ 
u_\nu \hat{g} & = u h & \text{on } &\Sigma.
\end{aligned}\right.
\end{align}
Since our primary interest lies in $\ker \Psi^*$, we will focus on the operator $\Phi^*$ instead, whose simpler structure (especially in weighted spaces) makes it easier to analyze.

\subsection{Generic conditions}
We say that generic conditions are satisfied on $(\Omega, \Sigma)$ when $\operatorname{ker}\Phi^*$ is trivial. The definition of generic conditions also extends to the non-compact setting, and as the results of the rest of this section are local, they extend as well.

We begin by analyzing the boundary condition. At an umbilical point $x_0 \in \Sigma$, the second fundamental form satisfies $h_{ij}(x_0) = \frac{H(x_0)}{n-1}\hat{g}_{ij}(x_0)$, which implies the boundary condition 
$$
u_\nu = \frac{H}{n-1}u \qquad \text{at } x_0.
$$
Conversely, if $x_0 \in \Sigma$ is a non-umbilic point, then we must have 
$$u_\nu = u = 0 \qquad \text{at } x_0.$$
Furthermore, since the non-umbilic condition is open, there exists a neighborhood $O \subset \Sigma$ of $x_0$ consisting entirely of non-umbilic points. Consequently, $u$ and $u_\nu$ vanish on $O$. By a unique continuation argument (as in the proof of Lemma~\ref{no kernel}), it follows that $u\equiv0$ on all of $\Omega$.

On the other hand, it is well known \cite{F-M} that the scalar curvature is constant in case $\operatorname{ker}L^*_g$ is nontrivial on a connected manifold. As noted in \cite{C}, this condition is related to static spacetimes in general relativity. Recall that a static spacetime is a four-dimensional Lorentzian manifold which possesses a timelike Killing field and a spacelike hypersurface which is orthogonal to the integral curves of this Killing field. We recall the following proposition in \cite{C}*{Proposition 2.7}:
\begin{prop}
Let $(M^n,g)$ be a Riemannian manifold. Then $0\ne f\in\operatorname{ker} L^*_g$ if and only if the warped product metric $\bar{g} \equiv -f^2 dt^2 + g$ is Einstein (away from the zeros of $f$).
\end{prop}

In view of the preceding discussion of the boundary condition, we have
\begin{thm}\label{generic thm}
Let $(\Omega, \Sigma)$ be a domain in $(M^n,g)$. If $\operatorname{ker}L^*$ is trivial, or there is any non-umbilical point on $\Sigma$, then $\operatorname{ker}\Phi^*$ is trivial.
\end{thm}

Many authors call a non-trivial element $u\in\operatorname{ker}L^*$ a static potential. With the boundary term, however, we will call a non-trivial element $u\in\operatorname{ker}L^*$ a \textbf{possible} static potential on $(\Omega, \Sigma)$, and if it satisfies the boundary condition on $\Sigma$ as well, we then call it a static potential on $(\Omega, \Sigma)$ (see \cites{A-dL, A-dL2} for similar definitions of static potentials and static manifolds with boundary). That is, a non-trivial element $u\in\operatorname{ker}\Phi^*$ is called a static potential in our setting, in which case generic conditions are not satisfied. Let us study the static potentials more carefully and derive some geometric properties of $\Omega$ and $\Sigma$. In fact, we have the following result:

\begin{thm}\label{const mean}
If $\operatorname{ker}\Phi^*$ is non-trivial, then the scalar curvature of $\Omega$ is constant, the boundary $\Sigma$ is umbilic, and the mean curvature is locally constant on $\Sigma$.
\end{thm}

\begin{proof}
From the preceding discussion, we know that $\Sigma$ is umbilic and its second fundamental form satisfies $h = \frac{H}{n-1}\hat{g}$. Let $\{e_1, \dots, e_{n-1}, \nu\}$ be a local orthonormal frame adapted to $\Sigma$, where $e_i$'s are tangent to $\Sigma$ and $\nu$ is the outward unit normal.

We begin by differentiating the second equation from (\ref{generic}) in the direction of $e_i$.
\begin{align*}
0 & = \nabla_{e_i} \left( u_\nu \hat{g}_j^i - u h_j^i \right)\\
		& = \nabla_{e_i} \left( u_\nu \hat{g}_j^i - u \frac{H}{n-1}\hat{g}_j^i \right)\\
		& = e_i \left( u_\nu - u \frac{H}{n-1} \right)\hat{g}_j^i\\
		& = \left(u_{i\nu} + (D_i \nu)u - u_i \frac{H}{n-1} - \frac{u}{n-1}\nabla_i H\right)\hat{g}_j^i\\
		& = \left(u_{i\nu} + h_i^k u_k - u_i \frac{H}{n-1} - \frac{u}{n-1}\nabla_i H\right)\hat{g}_j^i\\
		& = u_{j\nu} + h_j^k u_k - u_i h_j^i - \frac{u}{n-1}\nabla_j H\\
		& = u_{j\nu} - \frac{u}{n-1}\nabla_j H.
\end{align*}

Next, we examine the $i\nu$-component of the equation $L^*u = 0$. Since $L^*u = -\left(\Delta_{g} u\right) g+\operatorname{Hess}(u) - u \operatorname{Ric}_g$, we have
$$
0 = u_{i\nu} - u R_{i\nu}.
$$

Finally, by the Codazzi equation,
\begin{align*}
\nabla_i H & = \nabla_i (h_{jk}\hat{g}^{jk}) = \nabla_i h_{jk}\hat{g}^{jk} = \left(\nabla_j h_{ik} - R_{\nu kij}\right)\hat{g}^{jk}\\
	& = \nabla_j \left(\frac{H}{n-1}\delta_{i}^j\right) - R_{i\nu} = \frac{1}{n-1}\nabla_i H - R_{i\nu}.
\end{align*}

Combining the above three equations, we obtain
$$
\frac{1}{n-1}u\nabla_i H = u_{i\nu} = u R_{i\nu} = -\frac{n-2}{n-1}u\nabla_i H,
$$
and this means $u\nabla_i H = 0$.

We now conclude that $\nabla H = 0$ on $\Sigma$, meaning the mean curvature is locally constant. Suppose, for contradiction, that $\nabla H(x_0) \neq 0$ at some $x_0 \in \Sigma$. Then $\nabla H$ is non-zero in an open neighborhood $O \subset \Sigma$ of $x_0$. The above equation then forces $u \equiv 0$ on $O$. Furthermore, since $\Sigma$ is umbilic, the boundary condition $u_\nu = \frac{H}{n-1}u$ implies $u_\nu$ also vanishes on $O$. A unique continuation argument, identical to the proof of Lemma~\ref{no kernel}, now implies that $u$ vanishes identically on $\Omega$, which contradicts our assumptions.
\end{proof}

If generic conditions are not satisfied, then $(\Omega, \Sigma)$ is called a non-generic domain in $(M^n,g)$. We have particular interest in non-generic domains and the space of static potentials on them. One of the reasons is that, non-generic domains are often related to certain rigidity results for the positive mass theorem (e.g. \cites{A-B-dL, A-dL, A-dL2, H-M-R, M, S-T}). 

In general, for non-generic domains $(\Omega, \Sigma)$, the following two results hold as a direct consequence of \cite{C}*{Corollary 2.4, Proposition 2.5}. The key idea is that the equation $L^*u = 0$ reduces to a second-order linear ODE along geodesics, following the proof of \cite[Corollary 2.4]{C}, cf. \cite{F-M}. The map which takes a possible static potential $u$ to its one-jet $(u(p), du(p))$ at a point is injective, so that the dimension of the space of possible static potentials is at most $n+1$. The boundary condition here shows the map $u \mapsto \left(u(p), \left.du\right|_\Sigma(p)\right)$ is injective on static potentials.

\begin{prop}
For non-generic domains $(\Omega, \Sigma)$ in $(M^n,g)$, we have $\operatorname{dim} \operatorname{ker} \Phi^{*} \le n$ in $H_{\operatorname{loc}}^{2}(\Omega\cup\Sigma)$.
\end{prop}

\begin{prop}\label{kernel C2}
$H_{\operatorname{loc}}^{2}(\Omega\cup\Sigma)$ elements in $\operatorname{ker} \Phi^*$ are actually in $C^2(\overline{\Omega})$.
\end{prop}

We will take a closer look at non-generic domains in another paper \cite{Sheng}.

\section{Localized Deformation}\label{sec4}
This section is devoted to the proof of our main theorem. The proof proceeds in several stages, beginning with a detailed analysis of the central linearized problem.

The linear analysis focuses on the following \emph{non-homogeneous boundary value problem (BVP)}:
\begin{align}\label{prob: 0}
\left\{\begin{aligned}
L(a) & = f \quad &\text{in } &\Omega\\ 
B(a) & = \psi &\text{on } &\Sigma.
\end{aligned}\right.
\end{align}
We decompose this into two sub-problems: the interior problem
\begin{equation}\label{prob: 1}
    L(a) = f \qquad \mbox{in $\Omega$},
\end{equation}
and the \emph{semi-homogeneous BVP}
\begin{align}\label{prob: 2}
\left\{\begin{aligned}
L(a) & = 0 \quad &\text{in } &\Omega\\ 
B(a) & = \psi &\text{on } &\Sigma.
\end{aligned}\right.
\end{align}
Because the BVP \eqref{prob: 0} lacks ellipticity for tensor solutions, we follow the approach in \cite{C} to seek a special solution of the form $a = \rho L^*u$, which also provides desirable decay rates near the boundary. This substitution transforms the interior problem \eqref{prob: 1} into a fourth-order homogeneous Dirichlet problem, which is solved in Theorem~\ref{Dirichlet1}:
\begin{align*}
\left\{\begin{aligned}
L(\rho L^*u) & = f \quad &\text{in } &\Omega\\ 
u & = 0 &\text{on } &\Sigma\\
u_\nu & = 0 &\text{on } &\Sigma.
\end{aligned}\right.
\end{align*}

The main difficulty is thus transferred to the semi-homogeneous problem \eqref{prob: 2}. Its solvability is equivalent to the surjectivity of the boundary operator $B$ restricted to the subspace of tensors $S = \ker L$. We decompose this subspace $S$ into
\begin{equation*}
    S = \overline{S'} \oplus (S')^\perp,
\end{equation*}
where $S' = \{a = \rho L^* u: L(a) = 0,\, u_\nu = 0\}$ is called the \emph{principal component} of tensors in $S$. By perturbing a self-adjoint operator $P$ (defined in Section~\ref{sec4.3}) by a compact operator, Fredholm theory (Theorem~\ref{Fredholm}) implies that the image $B(S')$ has finite codimension. This means the corresponding fourth-order BVP
\begin{align}\label{prob: 2u}
\left\{\begin{aligned}
L(\rho L^*u) & = 0 \quad &\text{in } &\Omega\\ 
B(\rho L^*u) & = \psi &\text{on } &\Sigma\\
u_\nu & = 0 &\text{on } &\Sigma
\end{aligned}\right.
\end{align}
has a weak solution $u$ for all $\psi$ in a finite-codimension subspace.
We then show that the full problem \eqref{prob: 2} is solvable by weak tensor solutions for all $\psi$ (under generic conditions) by complementing this image. Specifically, in Theorem~\ref{basis}, we find a set of tensors $a_1, \dots, a_p \in S\setminus S'$ such that their images form a finite-dimensional space $W \triangleq \operatorname{span}\{B(a_1), \dots, B(a_p) \}$ that algebraically complements $B(S')$.

A crucial step is establishing the regularity of these solutions. The principal component $\rho L^* u$ in \eqref{prob: 2u} satisfies an elliptic problem, so we can establish its regularity using Schauder estimates. However, this standard strategy fails for tensors $a_1, \dots, a_p$ in the complementary space, as \eqref{prob: 2} lacks this elliptic structure. Therefore, these tensors must be carefully selected from the outset (using density theorems) to ensure they possess the required regularity.

Finally, we employ a Picard iteration scheme to extend this comprehensive linear analysis to the full nonlinear problem, thereby proving the main theorem.

\subsection{Poincar\'e-type inequalities}\label{sec4.1}
We will first focus on the following two operators on weighted Sobolev spaces:
$$
L^*: H_{\rho}^2(\Omega) \rightarrow \mathcal{L}_{\rho}^2(\Omega),
$$
and
\begin{align*}
\Phi^*: H_{\rho}^2(\Omega) & \rightarrow \mathcal{L}_{\rho}^2(\Omega)\times \mathcal{L}_{\rho}^2(\Sigma);\\
u & \mapsto (L^*u, u_\nu \hat{g}-u h)
\end{align*}
with norm 
$$\left\|\Phi^{*} u\right\|^2_{\mathcal{L}^{2}_\rho} =  \left\|L^{*} u\right\|^2_{\mathcal{L}^{2}_\rho(\Omega)}+\|u_\nu \hat{g}-u h\|^2_{\mathcal{L}^2_\rho(\Sigma)}.$$ 
The two Poincar\'e-type inequalities for $L^*$ and $\Phi^*$ in this subsection will be useful for later proof.

The following basic estimate, which relies solely on the overdeterminedness of $L^*$ and is proved in \cite{C}, will be important. Moreover, the constant $C$ can be chosen uniformly for all nearby metrics.
\begin{lem}\label{basic est}
There is a constant $C = C(n, g, \Omega)$, uniform for metrics $\mathcal{C}^{k+4}$-near $g$, so that for $u \in H_{\operatorname{loc}}^{2}(\Omega)$,
$$
\|u\|_{H^2(\Omega)} \leq C\left(\left\|L^{*} u\right\|_{\mathcal{L}^{2}(\Omega)}+\|u\|_{L^{2}(\Omega)}\right).
$$
\end{lem}

In fact, for any $\epsilon\ge0$ small we have the following estimate, where the constant $C$ is independent of $\epsilon$:
$$
\|u\|_{H^2(\Omega_\epsilon)} \leq C(n, g, \Omega)\left(\left\|L^{*} u\right\|_{\mathcal{L}^{2}(\Omega_\epsilon)}+\|u\|_{L^{2}(\Omega_\epsilon)}\right).
$$
Recall $\Omega_\epsilon$ is defined in Section \ref{Weight function}. Then by a standard integration technique discussed in \cite{C}*{Theorem 3}, we have the following corollary. 
\begin{cor}\label{basic est2}
There is a constant $C = C(n, g, \Omega, \rho)$, uniform for metrics $\mathcal{C}^{k+4}$-near $g$, so that for $u \in H^2_{\rho} (\Omega)$,
$$
\|u\|_{H_\rho^2(\Omega)} \leq C\left(\left\|L^{*} u\right\|_{\mathcal{L}_\rho^{2}(\Omega)}+\|u\|_{L_\rho^{2}(\Omega)}\right).
$$
\end{cor}

Using Lemma \ref{no kernel} and the variational method, we get our first Poincar\'e-type inequality.

\begin{thm}[Poincar\'e-type inequality for $L^*$]\label{main est}
There is a constant $C = C(n, g, \Omega, \Sigma, \rho)$, uniform for metrics $\mathcal{C}^{k+4}$-near $g$, so that for $u \in H^2_{0, \rho} (\Omega)$, we have
$$
\|u\|_{H^2_\rho\left(\Omega\right)} \leq C\left\|L^* u\right\|_{\mathcal{L}^2_\rho\left(\Omega\right)}.
$$
\end{thm}
\begin{proof}
Recall from Proposition \ref{trace H0rho}, if $u \in H_{0, \rho}^2(\Omega)$, then $u = u_\nu = 0$ on $\Sigma$.

If $\|u\|_{L_\rho^{2}(\Omega)} = 0$, then $u=0$ and the inequality is trivial. Now let us assume $\|u\|_{L_\rho^{2}(\Omega)} > 0$. We may further assume $\|u\|_{L_\rho^{2}(\Omega)} = 1$ by normalization. Consider the functional $\mathcal{F}_0$ defined on $H_{0, \rho}^2(\Omega)$ by
$$
\mathcal{F}_0(u)= \|L^*u\|^2_{\mathcal{L}_\rho^{2}(\Omega)},
$$
and minimize it over
$$
\mathcal{A} = \{u\in H_{0, \rho}^2(\Omega): \|u\|_{L_\rho^2(\Omega)} = 1\}.
$$
By Corollary \ref{basic est2}, $\left\|L^{*} u\right\|_{\mathcal{L}_\rho^{2}(\Omega)} \ge C^{-1}\|u\|_{H_\rho^{2}(\Omega)} - 1$, so when $\|u\|_{H_\rho^{2}(\Omega)} >> 1$ this implies
$$
\mathcal{F}_0(u) \ge \left(C^{-1}\|u\|_{H_\rho^{2}(\Omega)} - 1\right)^2.
$$
Thus $\mathcal{F}_0$ satisfies the coercivity condition, and attains its infimum in $\mathcal{A}$. 

Suppose the functional is minimized by $\mathcal{F}_0(u_0) = \lambda$ over $\mathcal{A}$. It is obvious that $\lambda \ge 0$. Furthermore, if $\lambda = 0$, we would have $\|L^*u_0\|^2_{\mathcal{L}_\rho^2} = 0$; by Lemma \ref{no kernel}, $u_0\equiv0$, which is a contradiction to $\|u_0\|_{L_\rho^2} = 1$. Thus we must have $\lambda > 0$. So,
$$
\|u\|_{H_\rho^{2}} \leq C\left(\left\|L^{*} u\right\|_{\mathcal{L}_\rho^{2}}+1\right) \le C\left(\left\|L^{*} u\right\|_{\mathcal{L}_\rho^{2}}+\frac1{\sqrt{\lambda}} \left\|L^{*} u\right\|_{\mathcal{L}_\rho^{2}}\right) \le C'\left\|L^{*} u\right\|_{\mathcal{L}_\rho^{2}}.
$$
\end{proof}

The following theorem using the no-kernel condition is our second Poincar\'e-type inequality. 

\begin{thm}[Poincar\'e-type inequality for $\Phi^*$]\label{Poincare}
Assuming $\operatorname{ker}\Phi^*=\{0\}$, there is a constant $C = C(n, g, \Omega, \Sigma, \rho)$, uniform for metrics $\mathcal{C}^{k+4}$-near $g$, so that for $u \in H^2_{\rho} (\Omega)$,
$$
\|u\|_{H^{2}_\rho\left(\Omega\right)} \leq C\left\|\Phi^* u\right\|_{\mathcal{L}_\rho^{2}}.
$$
\end{thm}

The argument follows the strategy of \cite{C}, with necessary adaptations. Specifically, the main theorem is a direct consequence of the integration technique in \cite{C}*{Theorem 3}, once the following uniform estimate is established.
\begin{prop}
Assuming $\operatorname{ker}\Phi^*=\{0\}$, for $\epsilon\ge0$ sufficiently small, and for any $u\in H^2_{\operatorname{loc}}(\Omega\cup\Sigma)$,
$$
\|u\|_{H^{2}\left(\Omega_{\epsilon}\right)} \leq C\left\|\Phi^* u\right\|_{\mathcal{L}^{2}\left(\Omega_{\epsilon}\right)}
$$
where $C$ is independent of $\epsilon$ small and uniform for metrics $\mathcal{C}^{k+4}$-near $g$.
\end{prop}
\begin{proof}
Since $\Phi^*$ has a trivial $H^2_{\operatorname{loc}}$-kernel in $\Omega$, it follows from the same reasoning as in \cite{C}*{Proposition 3.2} that for sufficiently small $\epsilon > 0$, the operator $\Phi^*$ also has no $H^2$-kernel in $\Omega_{\epsilon}$. Recall that in our setting $\varnothing\neq\Sigma\cap\Omega_{\epsilon}\subset\Omega_{\epsilon}$, so the restrictions and extensions in the proof keep the boundary condition. Furthermore, Proposition~\ref{kernel C2} guarantees the existence of pointwise boundary values.

By a standard application of Rellich’s lemma, we have for all $\epsilon\ge0$ small enough and $u\in H^2(\Omega_{\epsilon})$,
$$
\|u\|_{H^{2}\left(\Omega_{\epsilon}\right)} \leq C\left\|\Phi^* u\right\|_{\mathcal{L}^{2}\left(\Omega_{\epsilon}\right)}.
$$
We claim furthermore that the constant $C$ above is independent of $\epsilon$ small. Suppose the contrary. Then there is a sequence $\epsilon_i \downarrow 0$ and $u_i \in H^2(\Omega_i)$ such that 
$$
\|u_i\|_{H^{2}\left(\Omega_i\right)} > i\left\|\Phi^* (u_i)\right\|_{\mathcal{L}^{2}\left(\Omega_i\right)}.
$$
Here we denote $\Omega_i = \Omega_{\epsilon_i}$ for simplicity. By normalizing $\|u_i\|_{H^{2}\left(\Omega_i\right)}$ to 1, we have $\left\|\Phi^* (u_i)\right\|_{\mathcal{L}^{2}\left(\Omega_i\right)}<1/i$. Since $\p \Omega_i$ is $C^{1,1}$, we can extend $u_i$ to $H^2(\Omega)$ in such a way that for all $i$, $\|u_i\|_{H^{2}\left(\Omega\right)}\le C_1$.

Since the inclusion $H^2(\Omega) \hookrightarrow H^1(\Omega)$ is compact, by relabeling indices we have a function $\varphi\in H^1(\Omega)$ with $u_i\to\varphi$ in $H^1(\Omega)$; similarly, by Corollary \ref{compact cor} we have up to a subsequence $\left.u_i\right|_\Sigma$ also converges in $L^2(\Sigma)$, and by the uniqueness, it converges to $\left.\varphi\right|_\Sigma$. Moreover, for any $\eta\in\mathcal{C}_c^\infty(\Omega\cup\Sigma)$,
\begin{align*}
0=\lim _{i \rightarrow \infty}\left\langle\eta, \Phi^{*} (u_{i})\right\rangle_{\mathcal{L}^{2}} & = \lim _{i \rightarrow \infty}\left(\int_\Omega L(\eta)u_i \, d \mu_g + \int_\Sigma 2u_i\dot{H}(\eta) d\sigma_g\right)\\
		& = \int_\Omega L(\eta)\varphi \, d \mu_g + \int_\Sigma 2\varphi\dot{H}(\eta) d\sigma_g\\
		& = \left\langle\eta, \Phi^*(\varphi)\right\rangle_{\mathcal{L}^{2}}
\end{align*}
we have $\Phi^*(\varphi) = 0$ weakly. By elliptic regularity, $\varphi\in H^2_\text{loc}(\Omega\cup\Sigma)$. Then by the no-kernel condition of $\Phi^*$, $\varphi$ must be zero. 

To obtain a contradiction, we demonstrate that $\varphi$ cannot be zero. By Lemma \ref{basic est}, 
\begin{equation}\label{basic lemma}
\|u\|_{H^{2}(\Omega_\epsilon)} \leq C\left(\left\|L^{*} u\right\|_{\mathcal{L}^{2}(\Omega_\epsilon)}+\|u\|_{L^{2}(\Omega_\epsilon)}\right),
\end{equation}
where $C$ is independent of $\epsilon$ small and uniform for metrics $\mathcal{C}^{k+4}$-near $g$. By the $H^1$-convergence of $\{u_i\}$, there is an $i_1$ so that for $i \ge i_1$, 
$$\|u_i - \varphi\|_{L^2(\Omega)} < \frac1{4C}.$$
Moreover, we can choose $i_1$ large enough so that $i \ge i_1$ implies 
$$\|L^*(u_i - \varphi)\|_{\mathcal{L}^2(\Omega_i)} \le \|\Phi^*(u_i - \varphi)\|_{\mathcal{L}^2(\Omega_i)} =  \|\Phi^*(u_i)\|_{\mathcal{L}^2(\Omega_i)} < \frac1{4C}.$$
Plugging $(u_i - \varphi)$ into the above estimate (\ref{basic lemma}) we get for $i$ large
$$\|u_i - \varphi\|_{H^2(\Omega_i)} < \frac12.$$ 
Note we have strongly used the independence of $C$ on $i$ large. This shows that $\varphi$ cannot be zero by the normalization of $u_i$. Therefore, the constant $C$ must be independent of $\epsilon$.
\end{proof}

\subsection{The Dirichlet problems and the space $\mathcal{D}$}\label{sec4.2}
\begin{thm}[The Dirichlet problem 1]\label{Dirichlet1}
For $f \in \left(H_{0, \rho}^{2}(\Omega)\right)^*$, there is a unique solution $u \in H_{0, \rho}^2(\Omega)$ to
\begin{align}\label{Dirichlet eqn2 weak}
L(\rho L^*u)  = f \quad &\text{in } \Omega.
\end{align}
Moreover, there is a constant $C = C(n, g, \Omega, \Sigma, \rho)$ uniform for metrics $\mathcal{C}^{k+4}$-near $g$ such that 
\begin{equation}\label{DP}
    \|u\|_{H_{\rho}^2(\Omega)} \le C\|f\|_{\left(H_{0, \rho}^{2}(\Omega)\right)^*}.
\end{equation}
\end{thm}
\begin{proof}
From the density of $C^\infty_c(\Omega)$ in $H_{0, \rho}^2(\Omega)$, for $u \in H_{\rho}^2(\Omega)$ and $v \in H_{0, \rho}^2(\Omega)$, we have
$$
\left<L(\rho L^*u), v\right> = \int_\Omega \left<\rho L^*u, L^*v\right> d \mu_g.
$$
Note that this also implies $L(\rho L^*u) \in \left(H_{0, \rho}^{2}(\Omega)\right)^*$, as
$$
\left|\left<L(\rho L^*u), v\right>\right| \le \|L^*u\|_{L_{\rho}^2(\Omega)}\cdot \|L^*v\|_{L_{\rho}^2(\Omega)} \le C\|u\|_{H_{\rho}^2(\Omega)}\cdot \|v\|_{H_{\rho}^2(\Omega)},
$$
where $C$ depends only on the coefficients of $L^*$. 

For the uniqueness, assume $u_1, u_2 \in H_{0, \rho}^2(\Omega)$ are both solutions. Then $w=u_1 - u_2\in H_{0, \rho}^2(\Omega)$ satisfies $L(\rho L^*w) = 0$ in $\Omega$. Since
$$
0 = \left<L(\rho L^*w), w\right> = \int_{\Omega}\|L^*w\|^2\rho \, d \mu_g,
$$
we have $L^*w = 0$ in $\Omega$. Then by Lemma \ref{no kernel}, $w\equiv0$.

We will use the standard variational method to show the existence, which is similar to the proof of Theorem \ref{main est}. Consider the functional $\mathcal{F}_1$ defined on $H_{0, \rho}^2(\Omega)$ by
$$
\mathcal{F}_1(u)=\frac12 \|L^*u\|^2_{\mathcal{L}_\rho^{2}(\Omega)}-\left<f,u\right>.
$$
By Theorem \ref{main est}, $\left\|L^* u\right\|_{\mathcal{L}^2_\rho(\Omega)} \ge \frac1C\|u\|_{H^2_\rho(\Omega)}$, so we have
\begin{equation}\label{F1}
\mathcal{F}_1(u) \geq \frac1{2C^2}\|u\|^2_{H^2_\rho(\Omega)} - \|f\|_{\left(H_{0, \rho}^{2}(\Omega)\right)^*}\cdot\|u\|_{H^2_\rho(\Omega)}.
\end{equation}
Thus $\mathcal{F}_1$ satisfies the coercivity condition and attains its infimum. Suppose the functional is minimized by $\mathcal{F}_1(u_0) = \lambda$, and let $\eta\in H_0^2(\Omega)$. Then
\begin{align*}
0 & =\left.\frac{d}{d t}\right|_{t=0} \mathcal{F}_1(u_0+t \eta)\\
		& =\left.\frac{d}{d t}\right|_{t=0} \left[\int_{\Omega}\frac12\|L^{*}(u_0+t \eta)\|^{2}\rho \,d \mu_g- \left<f, u_0+t \eta\right>\right]\\
		& = \left.\frac{d}{d t}\right|_{t=0}\left[\int_{\Omega} \frac12\sum_{\alpha, \beta}\left(L^* u_0+t L^* \eta\right)_{\alpha\beta}^{2}\rho \, d \mu_{g}\right] - \left<f, \eta\right>\\
		& = \int_{\Omega}\left<\rho L^*u_0, L^* \eta\right> d \mu_g - \left<f, \eta\right>\\
		& = \left<L(\rho L^*u_0) - f, \eta\right>.
\end{align*}
This means $u_0\in H_{0, \rho}^2(\Omega)$ is a weak solution.

Finally, the infimum $\lambda = \mathcal{F}_1(u_0) \le\mathcal{F}_1(0) = 0$, so (\ref{F1}) gives us
$$
0\ge\mathcal{F}_1(u_0) \geq \frac1{2C^2}\|u_0\|^2_{H^2_\rho(\Omega)} - \|f\|_{\left(H_{0, \rho}^{2}(\Omega)\right)^*}\cdot\|u_0\|_{H^2_\rho(\Omega)}.
$$
This implies
$$
\|u_0\|_{H_{\rho}^2(\Omega)} \le 2C^2\|f\|_{\left(H_{0, \rho}^{2}(\Omega)\right)^*}.
$$
\end{proof}

Let us define 
$$\mathcal{D} \triangleq \left\{\hat{u}\in L^2_\rho(\Sigma): \exists u\in H^2_\rho(\Omega)\text{ such that }\gamma_0\left(u\rho^{\frac12}\right) = \hat{u}\rho^{\frac12} \text{ on }\Sigma\right\}.$$ 
By Corollary~\ref{trace cor}, if $\hat{u}\in\mathcal{D}$, then $\hat{u}\rho^{\frac12}\in H^{\frac32}(\Sigma)$. Furthermore, Corollary~\ref{trace Sigma' zero} implies that $\hat{u}\rho^{\frac12}$ can be extended by zero to all of $\p\Omega$ to a function in $H^{\frac32}(\p\Omega)$. The following Dirichlet problem shows that it suffices to work within the space $\mathcal{D}$, since there exists a one-to-one correspondence between functions defined on $\Sigma$ and those on $\Omega$.

\begin{thm}[The Dirichlet problem 2]\label{ext}
For each $\hat{u} \in \mathcal{D}$, there is a unique solution $u \in H_{\rho}^2(\Omega)$ to the following system of equations:
\begin{align}\label{Dirichlet}
\left\{\begin{aligned}
L(\rho L^*u) & = 0 \quad &\text{in } &\Omega\\ 
u & = \hat{u} &\text{on } &\Sigma\\
u_\nu & = 0 &\text{on } &\Sigma.
\end{aligned}\right.
\end{align}
\end{thm}
\begin{rmk}
Strictly speaking, the boundary conditions should be interpreted in the sense of trace. More precisely,
$$\left.\gamma_0\right|_\Sigma\left(u\rho^{\frac12}\right) = \hat{u}\rho^{\frac12} \quad\text{ and }\quad \left.\gamma_1\right|_\Sigma\left(u\rho^{\frac12}\right) = \hat{u}\left.D_\nu\left(\rho^{\frac12}\right)\right|_\Sigma.$$
\end{rmk}
\begin{proof}
Let $u_0 \in H_{\rho}^2(\Omega)$ be an extension of $\hat{u}$ satisfying the condition
$$\left.\gamma_1\right|_\Sigma(u_0\rho^{\frac12}) = \hat{u}\left.D_\nu\left(\rho^{\frac12}\right)\right|_\Sigma,$$ 
and denote $f \triangleq -L(\rho L^*u_0) \in \left(H_{0, \rho}^{2}(\Omega)\right)^*$. By Theorem~\ref{Dirichlet1}, there is a unique $v \in H_{0, \rho}^2(\Omega)$ satisfying $L(\rho L^*v) = f$. Then $u = u_0+v\in H_{\rho}^2(\Omega)$ solves (\ref{Dirichlet}).

For the uniqueness, assume $u_1, u_2 \in H_{\rho}^2(\Omega)$ are both solutions. Then $w=u_1 - u_2\in H_{\rho}^2(\Omega)$ satisfies $L(\rho L^*w) = 0$ in $\Omega$ and
$$
\left.\gamma_0\right|_\Sigma\left(w\rho^{\frac12}\right) = \left.\gamma_1\right|_\Sigma\left(w\rho^{\frac12}\right) = 0.
$$
It follows from Corollary~\ref{H0rho} that $w\in H_{0, \rho}^2(\Omega)$. Therefore, the same argument as in Theorem~\ref{Dirichlet1} gives $w\equiv0$.
\end{proof}

From now on, for $\hat{u}\in\mathcal{D}$ we will denote $u\in H_{\rho}^2(\Omega)$ to be the unique extension as a solution of the Dirichlet Problem (\ref{Dirichlet}).

Assuming $\operatorname{ker}\Phi^*=\{0\}$, let us define the inner product on $\mathcal{D}$ by
$$
\left<\hat{u},\hat{v}\right>_\mathcal{D} \triangleq \int_\Omega \left<L^*u,L^*v\right> \rho \, d\mu_g +  C_0\int_\Sigma \hat{u}\hat{v}\rho \, d\sigma_g,
$$
where $C_0 = \operatorname{sup}_\Sigma\|h\|^2 < \infty$. The inner product is constructed so that the self-adjoint operator discussed in the following subsection becomes unitary. Note that if $C_0 = \operatorname{sup}_\Sigma\|h\|^2= 0$, then $\Sigma$ is totally geodesic and the operator simplifies to $\Phi^*(u) = (L^*(u), 0)$. Consequently, we have $\ker\Phi^* = \ker L^*$ and $\|\Phi^*\| = \|L^*\|$. Therefore, all the discussion regarding the space $\mathcal{D}$ remains valid without the boundary term.


\begin{prop}
Assuming $\operatorname{ker}\Phi^*=\{0\}$, $\mathcal{D}$ is a Hilbert space.
\end{prop}
\begin{proof}
For any Cauchy sequence $\hat{u}_j$ in $\mathcal{D}$, $u_j$'s are solutions of (\ref{Dirichlet}), so
\begin{align*}
\|\Phi^*(u_j-u_k)\|^2_{\mathcal{L}_\rho^{2}} & = \|L^*(u_j-u_k)\|^2_{\mathcal{L}^2_\rho(\Omega)}+  \int_\Sigma \|(u_j-u_k)_\nu \hat{g}-(\hat{u}_j-\hat{u}_k) h\|^2_{\hat{g}} \rho \, d\sigma_g\\
		& = \|L^*(u_j-u_k)\|^2_{\mathcal{L}^2_\rho(\Omega)}+  \int_\Sigma \|h\|^2(\hat{u}_j-\hat{u}_k)^2 \rho \, d\sigma_g\\
		& \le \|L^*(u_j-u_k)\|^2_{\mathcal{L}^2_\rho(\Omega)}+  C_0\|\hat{u}_j-\hat{u}_k\|^2_{L^2_\rho(\Sigma)}\\
		& = \|\hat{u}_j-\hat{u}_k\|^2_\mathcal{D}.
\end{align*}
Together with Theorem \ref{Poincare}, we get
$$
\|u_j-u_k\|_{H^{2}_\rho\left(\Omega\right)} \leq C\left\|\Phi^* (u_j-u_k)\right\|_{\mathcal{L}_\rho^{2}} \leq C\|\hat{u}_j-\hat{u}_k\|_\mathcal{D}.
$$
Thus, $u_j$ is a Cauchy sequence in $H_{\rho}^2(\Omega)$. Since $H_{\rho}^2(\Omega)$ is a Hilbert space, there is $u\in H_{\rho}^2(\Omega)$ such that $u_j\rightarrow u$ in $H_{\rho}^2(\Omega)$. Define $\hat{u}\rho^{\frac12} = \gamma_0(u\rho^{\frac12})$ so that $\hat{u}\in\mathcal{D}$, then by Corollary \ref{trace cor}, $\|\hat{u}_j-\hat{u}\|^2_{L^2_\rho(\Sigma)} = \|\gamma_0((u_j-u)\rho^{\frac12})\|^2_{L^2(\Sigma)}\le C\|u_j-u\|^2_{H^{2}_\rho\left(\Omega\right)}$. So
\begin{align*}
\|\hat{u}_j-\hat{u}\|^2_\mathcal{D}  & = \|L^*(u_j-u)\|^2_{\mathcal{L}^2_\rho(\Omega)} + C_0\|\hat{u}_j-\hat{u}\|^2_{L^2_\rho(\Sigma)}\\
	& \le C'\|u_j-u\|^2_{H^{2}_\rho\left(\Omega\right)} + CC_0\|u_j-u\|^2_{H^{2}_\rho\left(\Omega\right)}\\
	& = C''\|u_j-u\|^2_{H^{2}_\rho\left(\Omega\right)}\longrightarrow0.
\end{align*}
This means $\mathcal{D}$ is a Hilbert space.
\end{proof}

The above proof also gives us a useful estimate for the Dirichlet problem (\ref{Dirichlet}).
\begin{cor}\label{Dirichlet estimate}
Assuming $\operatorname{ker}\Phi^*=\{0\}$, for the Dirichlet problem (\ref{Dirichlet}) there is a constant $C = C(n, g, \Omega, \rho)$ uniform for metrics $\mathcal{C}^{k+4}$-near $g$ such that 
$$
C^{-1}\|u\|_{H^{2}_\rho\left(\Omega\right)}\le \|\hat{u}\|_\mathcal{D}\le C\|u\|_{H^{2}_\rho\left(\Omega\right)}.
$$
\end{cor}
Moreover, since $C^\infty_c(\Omega\cup\Sigma)$ is dense in $H^2_\rho(\Omega)$ by Proposition~\ref{lem: density cpt supp}, it follows that
\begin{cor}
$C^\infty_c(\Sigma)$ is dense in $\mathcal{D}$.
\end{cor}

\begin{rmk}
Since $\mathcal{D}$ is a Hilbert space only if $\operatorname{ker}\Phi^*=\{0\}$, we will always assume $\operatorname{ker}\Phi^*=\{0\}$ from now on, and emphasize this assumption when needed.
\end{rmk}

\begin{rmk}\label{extension D}
It is worth noting that not every function in $H^{\frac32}(\Sigma)$ (or even in $H_0^{\frac32}(\Sigma)$) can be extended by zero to a function in $H^{\frac32}(\p\Omega)$; see \cite{L-M}*{p. 60, Theorem 11.4}. To address this issue, we will need to consider the finer space $H_{00}^{\frac32}(\Sigma)$; see \cite{L-M}*{p. 66, Theorem 11.7}. As discussed above, for any $\hat{u}\in\mathcal{D}$, the function $\hat{u}\rho^{\frac12}$ can be extended by zero to all of $\p\Omega$, and hence $\hat{u}\rho^{\frac12}\in H_{00}^{\frac32}(\Sigma)$.
\end{rmk}

\subsection{The self-adjoint problem}\label{sec4.3}
Define a self-adjoint operator $P: \mathcal{D}\to\mathcal{D}^*$ by
$$
\left<P(\hat{u}), \hat{v}\right> = \int_\Omega \left<L^*u, L^*v\right> \rho \, d\mu_g+ C_0\int_\Sigma \hat{u}\hat{v}\rho\, d\sigma_g = \left<P(\hat{v}), \hat{u}\right>,
$$
where $\hat{u}, \hat{v}\in\mathcal{D}$.

\begin{rmk}\label{P Cc form}
If $u, v \in C^\infty_c(\Omega\cup\Sigma)$ satisfy (\ref{Dirichlet}), then from Proposition~\ref{IBP},
\begin{align*}
\begin{aligned}
\left<P(\hat{u}), \hat{v}\right>	& = \int_\Omega \left<L^*u, L^*v\right> \rho \, d\mu_g+ C_0\int_\Sigma \hat{v}\hat{u}\rho\, d\sigma_g\\
	& = \int_\Sigma 2\hat{v}\dot{H}(\rho L^*u) \, d\sigma_g+ \int_\Sigma \hat{v}\left<\rho L^*u, h \right>_{\hat{g}} \, d\sigma_g+ C_0\int_\Sigma \hat{v}\hat{u}\rho\, d\sigma_g\\
	& = \left<2\dot{H}(\rho L^*u) + \left<\rho L^*u, h\right>_{\hat{g}} + C_0\hat{u}\rho, \hat{v}\right>.
\end{aligned}\end{align*}
\end{rmk}

\begin{prop}\label{unitary}
The operator $P$ is unitary.
\end{prop}
\begin{proof}
For $\hat{u}\in\mathcal{D}$, and for any $\epsilon>0$ there is $0\neq\hat{v}_0\in\mathcal{D}$ such that
\begin{align*}
\|P(\hat{u})\|_{\mathcal{D}^*}  & = \sup_{0\neq\hat{v}\in\mathcal{D}} \frac{|\left<P(\hat{u}), \hat{v}\right>|}{\|\hat{v}\|_{\mathcal{D}}}\le \frac{|\left<P(\hat{u}), \hat{v}_0\right>|}{\|\hat{v}_0\|_{\mathcal{D}}}+\epsilon\\
	& \le \left(\frac{\left(\int_\Omega \left<L^*u, L^*v_0\right> \rho\,d\mu_g + C_0\int_\Sigma \hat{v}_0\hat{u}\rho\,d\sigma_g\right)^2}{\|L^*v_0\|^2_{\mathcal{L}_\rho^{2}\left(\Omega\right)} + C_0\|\hat{v}_0\|_{L_\rho^{2}(\Sigma)}^{2}}\right)^{1/2}+\epsilon\\
	& \le \left(\frac{\left(\|L^*u\|_{\mathcal{L}_\rho^{2}\left(\Omega\right)} \|L^*v_0\|_{\mathcal{L}_\rho^{2}\left(\Omega\right)} + C_0\|\hat{u}\|_{L_\rho^{2}(\Sigma)}\|\hat{v}_0\|_{L_\rho^{2}(\Sigma)} \right)^2}{\|L^*v_0\|^2_{\mathcal{L}_\rho^{2}\left(\Omega\right)} + C_0\|\hat{v}_0\|_{L_\rho^{2}(\Sigma)}^{2}}\right)^{1/2}+\epsilon\\
	& \le \left(\frac{\left(\|L^*u\|^2_{\mathcal{L}_\rho^{2}\left(\Omega\right)} + C_0\|\hat{u}\|_{L_\rho^{2}(\Sigma)}^{2}\right)\left(\|L^*v_0\|^2_{\mathcal{L}_\rho^{2}\left(\Omega\right)} + C_0\|\hat{v}_0\|_{L_\rho^{2}(\Sigma)}^{2}\right)}{\|L^*v_0\|^2_{\mathcal{L}_\rho^{2}\left(\Omega\right)} + C_0\|\hat{v}_0\|_{L_\rho^{2}(\Sigma)}^{2}}\right)^{1/2}+\epsilon\\
	& = \|\hat{u}\|_{\mathcal{D}}+\epsilon.
\end{align*}
Let $\epsilon\to0$, we get $\|P(\hat{u})\|_{\mathcal{D}^*} \le \|\hat{u}\|_{\mathcal{D}}$.

On the other hand, for $\hat{u}\in\mathcal{D}$, $\hat{u}\neq0$,
$$
\|P(\hat{u})\|_{\mathcal{D}^*} \ge \frac{|\left<P(\hat{u}), \hat{u}\right>|}{\|\hat{u}\|_{\mathcal{D}}} = \frac{\left(\|L^*u\|^2_{\mathcal{L}_\rho^{2}\left(\Omega\right)} + C_0\|\hat{u}\|_{L_\rho^{2}(\Sigma)}^{2}\right)}{\left(\|L^*u\|^2_{\mathcal{L}_\rho^{2}\left(\Omega\right)} + C_0\|\hat{u}\|_{L_\rho^{2}(\Sigma)}^{2}\right)^{1/2}} = \|\hat{u}\|_{\mathcal{D}}.
$$
This means $\|P(\hat{u})\|_{\mathcal{D}^*} = \|\hat{u}\|_{\mathcal{D}}$.
\end{proof}


\begin{thm}[The self-adjoint problem]\label{self-adjoint soln weak}
Assuming $\operatorname{ker}\Phi^*=\{0\}$, for each $\psi\in\mathcal{D}^*$, there is a unique solution $\hat{u} \in \mathcal{D}$ such that
\begin{equation}\label{eqn:self-adjoint}
P(\hat{u}) = \psi,
\end{equation}
and we have
$$\|\hat{u}\|_{\mathcal{D}} = \|\psi\|_{\mathcal{D}^*}.$$
\end{thm}

\begin{proof}
Uniqueness and the estimate follow directly from Proposition \ref{unitary}.

Let us consider the functional $\mathcal{F}_2$ defined on $\mathcal{D}$ by
$$
\mathcal{F}_2(\hat{v})=\frac12 \|L^*v\|^2_{\mathcal{L}_\rho^{2}(\Omega)} + \frac{C_0}2\|\hat{v}\|_{L_\rho^{2}(\Sigma)}^{2} - \left<\psi, \hat{v}\right>.
$$
Here we make use of the one-to-one correspondence between $\hat{v} \in \mathcal{D}$ and $v\in H_\rho^2(\Omega)$ from the Dirichlet problem (\ref{Dirichlet}).

By Theorem \ref{Poincare} and Corollary \ref{Dirichlet estimate}, we have
$\|\hat{v}\|_\mathcal{D}\le C'\|v\|_{H_\rho^{2}(\Omega)} \le C\left\|\Phi^* v\right\|_{\mathcal{L}^2_\rho}$, then
$$
\begin{aligned}
\mathcal{F}_2(\hat{v}) & = \frac12\left\|\Phi^* v\right\|^2_{\mathcal{L}^2_\rho} - \frac12 \int_\Sigma \|h\|^2\hat{v}^2\rho\, d\sigma_g + \frac{C_0}2\|\hat{v}\|_{L_\rho^{2}(\Sigma)}^{2} - \left<\psi, \hat{v}\right>\\
	& \ge \frac{1}{2 C^2}\|\hat{v}\|_{\mathcal{D}}^{2} - \|\psi\|_{\mathcal{D}^*}\|\hat{v}\|_{\mathcal{D}}.
\end{aligned}
$$
Thus $\mathcal{F}_2(\hat{v})$ satisfies the coercivity condition and attains its infimum. Suppose the functional is minimized by $\mathcal{F}_2(\hat{v}_0) = \lambda$, and let $\hat{\eta} \in C_c^\infty(\Sigma)\subset\mathcal{D}$, then
\begin{align*}
0 & =\left.\frac{d}{d t}\right|_{t=0} \mathcal{F}_2(\hat{v}_0+t \hat{\eta})\\
		& =\left.\frac{d}{d t}\right|_{t=0} \left[\int_{\Omega}\frac12\|L^{*}(v_0+t \eta)\|^{2} \rho\,d \mu_g + \frac{C_0}2\|\hat{v}_0+t \hat{\eta}\|_{L_\rho^{2}(\Sigma)}^{2} -  \left<\psi, \hat{v}_0+t \hat{\eta}\right>\right]\\
		& = \int_{\Omega}\left<\rho L^*v_0, L^* \eta\right> d \mu_g + C_0\int_\Sigma \hat{v}_0\hat{\eta}\rho \,d\sigma_g- \left<\psi, \hat{\eta}\right>\\
		& = \left<P(\hat{v}_0) - \psi, \hat{\eta}\right>.
\end{align*}
This means $\hat{v}_0\in\mathcal{D}$ is a weak solution.
\end{proof} 

\begin{cor}
The operator $P$ is an isometric isomorphism.
\end{cor}

\subsection{The Fredholm operator}\label{sec4.4}
Define an operator $\hat{B}: \mathcal{D}\to\mathcal{D}^*$ by
\begin{equation}\label{hat B}
    \hat{B}(\hat{u}) \triangleq P(\hat{u}) - \left<\rho L^*u, h\right>_{\hat{g}} - C_0\hat{u}\rho.
\end{equation}
Then $\hat{B}$ differs from the self-adjoint operator $P$ by a lower order term $K = K_1 + K_2$. We will show that the lower order term is a compact operator, and hence $\hat{B}$ is Fredholm. This is not immediately obvious in weighted Sobolev spaces, particularly due to the use of (\ref{Hardy}) in the proof.

\begin{prop}\label{compact1}
The operator $K_1: \mathcal{D}\to\mathcal{D}^*$ defined by $K_1(\hat{u}) = C_0\hat{u}\rho$ is compact.
\end{prop}
\begin{proof}
For $\hat{v} \in \mathcal{D}$, 
$$
\left|\left<K_1(\hat{u}), \hat{v}\right>\right| = C_0\left|\int_\Sigma \hat{u}\hat{v}\rho\, d\sigma_g\right| \le C_0\|\hat{u}\|_{L_{\rho}^2(\Sigma)}\cdot \|\hat{v}\|_{L_{\rho}^2(\Sigma)} \le C^\frac12_0\|\hat{u}\|_{L_{\rho}^2(\Sigma)}\cdot \|\hat{v}\|_{\mathcal{D}}.
$$
Thus $K_1(\hat{u}) \in \mathcal{D}^*$, and $\|K_1(\hat{u})\|_{\mathcal{D}^*} \le C^\frac12_0\|\hat{u}\|_{L_{\rho}^2(\Sigma)}$.

Consider a bounded sequence $\hat{u}_j$ in $\mathcal{D}$ such that $\|\hat{u}_j\|_{\mathcal{D}} = 1$, then $u_j$'s are in $H_{\rho}^2(\Omega)$, and we have
$$
\|u_j\|_{H^{2}_\rho\left(\Omega\right)} \leq C\|\hat{u}_j\|_\mathcal{D} = C.
$$
From Corollary \ref{compact cor}, by relabeling indices we have that $\hat{u}_j\rho^{\frac12}\in L^2(\Sigma)$ is a Cauchy sequence. So
$$
\|K_1(\hat{u}_j) - K_1(\hat{u}_k)\|_{\mathcal{D}^*} \le C^\frac12_0\|\hat{u}_j - \hat{u}_k\|_{L_{\rho}^2(\Sigma)} = C^\frac12_0\|\hat{u}_j\rho^{\frac12} - \hat{u}_k\rho^{\frac12}\|_{L^2(\Sigma)} \longrightarrow 0.
$$
This means $K_1(\hat{u}_j)\in \mathcal{D}^*$ is also Cauchy, thus the operator $K_1$ is compact.

\end{proof}

\begin{prop}\label{compact2}
The operator $K_2: \mathcal{D}\to\mathcal{D}^*$ defined by $K_2(\hat{u}) = \left<\rho L^*u, h\right>_{\hat{g}}$ is compact.
\end{prop}
\begin{proof}
Consider a bounded sequence $\hat{u}_j$ in $\mathcal{D}$, so that $\|u_j\|_{H^{2}_\rho\left(\Omega\right)} \leq C$ for all $j$. Recall that $L^*u=-\left(\Delta_{g} u\right) g+\operatorname{Hess}_g(u) - u \operatorname{Ric}_g$. For any $\hat{v} \in \mathcal{D}$, we decompose the dual pairing $\langle K_2(\hat{u}_j), \hat{v} \rangle$ into three terms:
\begin{equation*}
    \begin{aligned}
    \text{I} & \triangleq \int_\Sigma \left<\rho \left(\Delta_{g} u_j\right) g, h\right>_{\hat{g}}\hat{v}\,d\sigma_g,\\
        \text{II} & \triangleq \int_\Sigma \left<\rho \operatorname{Hess}_g(u_j), h\right>_{\hat{g}}\hat{v}\,d\sigma_g,\\
        \text{III}  & \triangleq \int_\Sigma \left<\rho u_j \operatorname{Ric}_g, h\right>_{\hat{g}}\hat{v}\,d\sigma_g.
    \end{aligned}
\end{equation*}
Here, we slightly abuse notation by using integrals to represent dual pairings. Throughout the proof, we may occasionally omit the measures in integrals when they are clear from context.

\textbf{Term III:}
For
$$
\text{III} = \int_\Sigma \left<\operatorname{Ric}_g, h\right>_{\hat{g}}\rho \hat{u}_j\hat{v},
$$
since $\left<\operatorname{Ric}_g, h\right>_{\hat{g}}$ is bounded, we conclude that up to a subsequence III is Cauchy, which is similar to Proposition \ref{compact1}.

\textbf{Term II:}
On $\Sigma$, we have the decomposition:
$$
\operatorname{Hess}_g(u_j) = \operatorname{Hess}_\Sigma(u_j) + (u_j)_\nu h,
$$
so that
$$
\text{II} = \int_\Sigma \left< \operatorname{Hess}_{\Sigma}(\hat{u}_j), h\right>_{\hat{g}}\rho\hat{v}.
$$
In local Fermi coordinates, where $k,l = 1,\dots, n-1$, this becomes:
\begin{align*}
\text{II} & = \int_\Sigma \nabla_k\nabla_l \hat{u}_j h^{kl}\rho\hat{v} = -\int_\Sigma \nabla_l \hat{u}_j \nabla_k\left(h^{kl}\rho\hat{v}\right)\\
	& = -\int_\Sigma \nabla_l \hat{u}_j \nabla_k h^{kl}\rho\hat{v} -\int_\Sigma \nabla_l \hat{u}_j h^{kl}\rho\nabla_k\hat{v} -\int_\Sigma \nabla_l \hat{u}_j h^{kl}(\nabla_k\rho)\hat{v}.
\end{align*}
Note that on $\Sigma$, we have $|D u_j|^2 = |\nabla u_j|^2 + |(u_j)_\nu|^2 = |\nabla \hat{u}_j|^2$, and since $\|D u_j\|^2_{\mathcal{H}^1_\rho(\Omega)} \le \|u_j\|^2_{H^2_\rho(\Omega)} \le C$, it follows from Corollary~\ref{compact cor} that $\left(\nabla\hat{u}_j\right)\rho^{\frac12}\in\mathcal{L}^2(\Sigma)$ is Cauchy (up to a subsequence). Therefore, the first two integrals in II are Cauchy.

For the third integral, notice that both $\hat{v}$ and $\nabla\hat{v}$ are in $L^2_\rho(\Sigma)$, so $\hat{v}\in H_{\rho}^{1}(\Sigma)$, and we have by Proposition~\ref{ineq weighted Sobolev norms} and (\ref{Hardy}) that
$$\int_{\Sigma} \hat{v}^{2} \theta^{-2} \rho \leq C\left\|\hat{v} \rho^{\frac12}\right\|^2_{H^1(\Sigma)} \leq C\|\hat{v}\|_{H_{\rho}^{1}(\Sigma)}^{2}.$$ 
Hence,
\begin{align*}
& \left|\int_\Sigma \nabla_l \hat{u}_j h^{kl}\nabla_k\rho\hat{v} - \int_\Sigma \nabla_l \hat{u}_m h^{kl}\nabla_k\rho\hat{v}\right|\\
\le & C\int_\Sigma \left|\nabla \hat{u}_j - \nabla \hat{u}_m\right| \rho \theta^{-1}|\hat{v}|\\
\le & C\left(\int_\Sigma \left|\nabla \hat{u}_j - \nabla \hat{u}_m\right|^2\rho\right)^{\frac12}\left(\int_\Sigma \theta^{-2}\hat{v}^2 \rho\right)^{\frac12}\\
\le & C\|\hat{v}\|_{H_{\rho}^{1}(\Sigma)}\left(\int_\Sigma \left|\left(\nabla\hat{u}_j\right)\rho^{\frac12} - \left(\nabla\hat{u}_m\right)\rho^{\frac12}\right|^2\right)^{\frac12} \longrightarrow 0.
\end{align*}
Thus, the third integral is also Cauchy. Therefore, II is Cauchy up to a subsequence.

\textbf{Term I:}
For
$$
\text{I} = \int_\Sigma \Delta_{g} u_j H\rho\hat{v},
$$
let us choose a test function $w \in H^3_\rho(\Omega)$ such that $w=0$ on $\p\Omega$, $w_\nu = H\hat{v}$ on $\Sigma$, and $w_\nu = 0$ on $\Sigma'$ (such $w$ exists by the trace theorem). By Proposition~\ref{IBP} and Proposition~\ref{lem: density cpt supp},
\begin{align*}
0 & = \int_\Omega L(\rho L^*u_j) w\\
	& = \int_\Omega \left<L^*u_j, L^*w\right> \rho + \int_\Sigma w_\nu\operatorname{tr}_\Sigma (\rho L^*u_j)\\
	& = \int_\Omega \left<L^*u_j, \rho L^*w\right> + \int_\Sigma H\hat{v}\left(-(n-1)\rho\Delta_g u_j + \rho\Delta_\Sigma u_j - \hat{u}_j\rho\operatorname{tr}_\Sigma \operatorname{Ric}_g\right).
\end{align*}
Rearranging gives:
$$
(n-1)\text{I} = \int_\Omega \left<L^*u_j, \rho L^*w\right> + \int_\Sigma H\rho\hat{v}\Delta_\Sigma u_j - \int_\Sigma H\left(\operatorname{tr}_\Sigma \operatorname{Ric}_g\right)\hat{v}\hat{u}_j\rho.
$$
Note that
$$
H\operatorname{tr}_\Sigma \operatorname{Ric}_g = H(R_g-\operatorname{Ric}_g(\nu,\nu)) = \frac12H(R_g+R_\Sigma+\|h\|^2-H^2)
$$
is bounded. Hence, as in Proposition~\ref{compact1}, the third integral has a Cauchy subsequence.

For the second integral:
$$
\int_\Sigma H\rho\hat{v}\Delta_\Sigma u_j = -\int_\Sigma \nabla_\Sigma (H\rho\hat{v})\cdot\nabla_\Sigma u_j.
$$
So it also has a Cauchy subsequence, which is similar to Term II. 

For the first integral, since $L^*w \in \mathcal{H}^1_\rho(\Omega)$, integration by parts allows us to transfer a derivative from $L^*u$ to $\rho L^*w$ and get a Cauchy subsequence. As a result, I is Cauchy up to a subsequence.

\textbf{Conclusion:} We have shown that, up to a subsequence, $\left<K_2(\hat{u}_j), \hat{v}\right>$ is Cauchy for every $\hat{v} \in \mathcal{D}$. Now, for any $\epsilon>0$ there is $\hat{v}_0\in\mathcal{D}$ with $\|\hat{v}_0\|_{\mathcal{D}}=1$ such that
$$\|K_2(\hat{u}_j) - K_2(\hat{u}_m)\|_{\mathcal{D}^*} \le \left|\left<K_2(\hat{u}_j), \hat{v}_0\right> - \left<K_2(\hat{u}_m), \hat{v}_0\right>\right| + \epsilon.$$
Since $\left<K_2(\hat{u}_j), \hat{v}_0\right>$ is Cauchy, there is $N$ large such that for all $j, m>N$, 
$$
\left|\left<K_2(\hat{u}_j), \hat{v}_0\right> - \left<K_2(\hat{u}_m), \hat{v}_0\right>\right| < \epsilon.
$$
This implies
$$\|K_2(\hat{u}_j) - K_2(\hat{u}_m)\|_{\mathcal{D}^*} < 2\epsilon \quad\text{ for all }j, m>N.$$
Therefore, up to a subsequence, $K_2(\hat{u}_j)$ is Cauchy, and hence the operator $K_2$ is compact.
\end{proof}

\begin{thm}\label{Fredholm}
The operator $\hat{B}$ is a bounded Fredholm operator and $\operatorname{ind}(\hat{B}) = \operatorname{ind}(P) = 0$. Moreover, $\hat{B}\left(\mathcal{D}\right)$ is closed and has finite codimension in $\mathcal{D}^*$.
\end{thm}
\begin{proof}
From the above propositions, we have $\hat{B} = P + K$, where $K$ is a compact operator. The result then follows directly from standard properties of Fredholm operators \cites{bS, G}.
\end{proof}


\subsection{Solving the linearized equation}\label{sec: weak soln}
We now focus on the densely defined operator $\Psi$ and study its surjectivity.

The operator is given by
\begin{align*}
\Psi: \mathcal{L}_{\rho^{-1}}^2(\Omega) & \rightarrow \left(H_{\nu, \rho}^{2}(\Omega)\right)^*\times\mathcal{D}^*;\\
a & \mapsto (L(a), B(a)).
\end{align*}
Specifically, we consider the realization of $\Psi$ with the domain
\begin{align*}
    \operatorname{dom}(\Psi) \triangleq \Big\{a& \in \mathcal{L}_{\rho^{-1}}^2(\Omega):\: L(a)\in \left(H_{\nu, \rho}^{2}(\Omega)\right)^*, B(a)\in \mathcal{D}^*, \langle  a, h\rangle_{\hat g}\in \mathcal{D}^* \text{ such that}\\ &\forall  u\in H_{\nu, \rho}^{2}(\Omega),\: \langle L(a), u\rangle + \langle B(a), \hat u\rangle = \langle a, L^* u\rangle - \big\langle\langle a, h\rangle_{\hat g} ,\hat u \rangle \Big\}.
\end{align*}
Here, the brackets denote the appropriate dualities. The domain is endowed with the graph norm
\begin{align}\label{eq: norm}
    \|a\|_{\mathrm{dom(\Psi)}}^2 = \|a\|^2_{\mathcal{L}^2_{\rho^{-1}}(\Omega)}+ \|L(a)\|_{\left(H^2_{\nu, \rho}(\Omega)\right)^*}^2 + \|B(a)\|_{\mathcal{D}^*}^2 + \|\langle a, h \rangle_{\hat g}\|_{\mathcal{D}^*}^2.
\end{align}

We begin with a density lemma, which implies that the image $\Psi\left(\mathcal{C}^\infty_c(\Omega\cup \Sigma)\right)$ is dense in $\left(H_{\nu, \rho}^{2}(\Omega)\right)^*\times\mathcal{D}^*$ under generic conditions.
\begin{lem}\label{lem:density}
$\mathcal{C}_c^\infty(\Omega\cup \Sigma)$ is dense in $\operatorname{dom}(\Psi)$ with respect to the graph norm.
\end{lem}
{\begin{proof}
By the Hahn-Banach theorem, it suffices to show that any bounded linear functional $\ell$ on $\operatorname{dom}(\Psi)$ that vanishes on $\mathcal{C}_c^\infty(\Omega\cup \Sigma)$ must be identically zero.

Let $\ell$ be such a functional. There must exist $b\in \mathcal{L}_\rho^2(\Omega)$, $u \in H_{\nu, \rho}^{2}(\Omega)$, and $\hat{v}_1, \hat{v}_2 \in \mathcal{D}$ such that for any $a\in \operatorname{dom}(\Psi)$,
    \begin{align}\label{eq: functional}
        \ell(a) = &\langle a, b\rangle_{\mathcal{L}^2_{1/\rho}(\Omega), \mathcal{L}^2_\rho(\Omega)} + \langle L(a), u\rangle_{(H^2_{\nu, \rho}(\Omega))^*, H^2_{\nu, \rho}(\Omega)} \\
        + & \langle B(a), \hat{v}_1\rangle_{\mathcal{D}^*, \mathcal{D}} + \langle \langle a, h\rangle_{\hat{g}}, \hat{v}_2\rangle_{\mathcal{D}^*, \mathcal{D}},\notag
    \end{align}
where $\langle\cdot, \cdot\rangle$ denotes the duality pairing between the corresponding spaces. By the definition of $\operatorname{dom}(\Psi)$, 
\begin{align*}
    \langle L(a), u\rangle_{(H^2_{\nu, \rho}(\Omega))^*, H^2_{\nu, \rho}(\Omega)} = \langle a, L^* u\rangle_{\mathcal{L}^2_{1/\rho}(\Omega), \mathcal{L}^2_\rho(\Omega)} - \langle B(a), \hat u\rangle_{\mathcal{D}^*, \mathcal{D}}   
    - \big\langle\langle  a, h\rangle_{\hat g} ,\hat u \rangle_{\mathcal{D}^*, \mathcal{D}}.
\end{align*}
So we can rewrite (\ref{eq: functional}) as
    \begin{align}\label{eq: functional2}
        \ell(a) & = \langle a, b + L^* u\rangle_{\mathcal{L}^2_{1/\rho}(\Omega), \mathcal{L}^2_\rho(\Omega)} + \langle B(a), \hat{v}_1-\hat u\rangle_{\mathcal{D}^*, \mathcal{D}} + \langle \langle a, h\rangle_{\hat g}, \hat{v}_2-\hat u\rangle_{\mathcal{D}^*, \mathcal{D}}.
    \end{align}
        
First, consider an arbitrary test tensor $a\in \mathcal{C}_c^\infty(\Omega)$, then 
$$
0 = \ell(a) = \langle a, b + L^* u\rangle_{\mathcal{L}^2_{1/\rho}(\Omega), \mathcal{L}^2_\rho(\Omega)}.
$$
This means $b + L^* u = 0$, and (\ref{eq: functional2}) simplifies to 
$$\ell(a) = \langle B(a), \hat{v}_1-\hat u\rangle_{\mathcal{D}^*, \mathcal{D}} + \langle \langle a, h\rangle_{\hat g}, \hat{v}_2-\hat u\rangle_{\mathcal{D}^*, \mathcal{D}}.$$

As established in Section~\ref{sec3.1}, for any given $\varphi \in C_c^\infty(\Sigma)$, we can construct a tensor $a_\varphi \in \mathcal{C}_c^\infty(\Omega\cup \Sigma)$ such that $B(a_\varphi) = \varphi$ and $a_\varphi = 0$ on $\Sigma$. Since $C_c^\infty(\Sigma)$ is dense in $\mathcal{D}$, let $\varphi_j\in C_c^\infty(\Sigma)$ be a sequence converging to $\hat{v}_1-\hat u$ in $\mathcal{D}$. Let $a_j \in \mathcal{C}_c^\infty(\Omega\cup \Sigma)$ denote the tensor corresponding to each $\varphi_j$. Then we have
$$
0 = \lim_{j\to \infty}\ell(a_j) = \lim_{j\to \infty}\langle B(a_j), \hat{v}_1-\hat u\rangle_{\mathcal{D}^*, \mathcal{D}} = \lim_{j\to \infty}\int_\Sigma\varphi_j (\hat{v}_1-\hat u)d\sigma_g = \int_\Sigma(\hat{v}_1-\hat u)^2 d\sigma_g.
$$
This implies that $\hat{v}_1-\hat u = 0$, and (\ref{eq: functional2}) further simplifies to 
$$\ell(a) = \langle \langle a, h\rangle_{\hat g}, \hat{v}_2-\hat u\rangle_{\mathcal{D}^*, \mathcal{D}}.$$

Finally, for any $\varphi \in C_c^\infty(\Sigma)$, we can construct $a_\varphi\in \mathcal{C}_c^\infty(\Omega \cup \Sigma)$ such that $a_\varphi = \varphi h$ on $\Sigma$. Let $\varphi_j \in C_c^\infty(\Sigma)$ be a sequence converging to $\hat v_2 - \hat u$ in $\mathcal D$ and $a_j\in \mathcal{C}_c^\infty(\Omega \cup \Sigma)$ be the corresponding tensors. Then
$$
0 = \lim_{j\to \infty}\ell(a_j) = \lim_{j\to \infty}\langle \langle a_j, h\rangle_{\hat g}, \hat{v}_2 - \hat u\rangle_{\mathcal{D}^*, \mathcal{D}}  = \int_\Sigma (\hat{v}_2-\hat u)^2\| h\|^2\, d\sigma_g.
$$
Therefore, $\hat{v}_2 - \hat{u} = 0$ on the support of $h$, and from (\ref{eq: functional2}) we conclude that $\ell(a) = 0$ for all $a \in \operatorname{dom}(\Psi)$. 
\end{proof}





\subsubsection{Existence of solutions}
Up to now, we have focused on a special type of tensor solutions of the form $a = \rho L^*u$, which are elements of the space
$$S' = \left\{\rho L^*u: u \mbox{ solves \eqref{Dirichlet} for \:}\hat{u}\in \mathcal{D}\right\}.$$
Unfortunately, unlike in the case of variational problems, the operator $\hat{B}: \mathcal{D}\to\mathcal{D}^*$ is not necessarily surjective. We now turn our attention to tensor solutions that cannot, in general, be expressed in this form. Accordingly, we consider the space
$$
S = \left\{a\in\mathcal{L}_{\rho^{-1}}^2(\Omega): L(a) = 0\right\}.
$$
Observe that $S$ is closed in $\mathcal{L}_{\rho^{-1}}^2(\Omega)$, thus is also a Hilbert space.

It is easy to see that $S'\subset S$ since for $\hat{u}\in \mathcal{D}$, we have $\rho L^*u\in\mathcal{L}_{\rho^{-1}}^2(\Omega)$ and $L(\rho L^*u) = 0$ in $\Omega$. Let $\overline{S'}$ denote the closure of $S'$ in $\mathcal{L}_{\rho^{-1}}^2(\Omega)$, i.e., any $a\in\overline{S'}$ can be written as 
\begin{equation}\label{S' limit}
a = \lim_{j \to \infty} \left(\rho L^*u_j\right),
\end{equation}
where $\{\rho L^*u_j: \hat{u}_j\in \mathcal{D}\}$ is a Cauchy sequence in $\mathcal{L}_{\rho^{-1}}^2(\Omega)$. Then $\overline{S'}\subset S$, and we may split the space $S$ as
$$
S = \overline{S'}\oplus (S')^\perp.
$$

\bigskip
If $S$ were equal to $\overline{S'}$, then any tensor in $S$ could be represented as in (\ref{S' limit}). However, this is not the case; in fact, $(S')^\perp$ is infinite-dimensional. 

\begin{prop}\label{S'}
The space $S'$ has infinite codimension in $S$.
\end{prop}
\begin{proof}
Consider $a\in S$. Then $a\in (S')^\perp$ if and only if for any $\hat{u}\in\mathcal{D}$,
\begin{equation}\label{infinite S' perp}
0 = \left<a, \rho L^*u\right>_{\mathcal{L}_{\rho^{-1}}^2(\Omega)} = \int_\Omega \left<a, L^*u\right> d \mu_g.
\end{equation}
We notice that there is an infinite-dimensional space of tensors $a'\in\mathcal{C}^\infty_c(\Omega)$ with $L(a')=0$ (see \cite{B-E-M}). Then any such tensor $a'$ satisfies $a'\in\mathcal{L}_{\rho^{-1}}^2(\Omega)$ and hence $a'\in S$. Moreover, we have 
$$0=\int_\Omega L(a')u \, d \mu_g = \int_\Omega \left<a', L^*u\right> d \mu_g,$$
so (\ref{infinite S' perp}) is satisfied. Therefore, any such tensor $a'$ lies in $(S')^\perp$. As a result, the space $(S')^\perp$ is also infinite-dimensional. 
\end{proof}

The operator $B$ is initially defined as $B=2\dot{H}$ on $\mathcal{C}^\infty_c(\Omega\cup\Sigma)$, which is dense in $\mathcal{L}_{\rho^{-1}}^2(\Omega)$. Our goal is now to extend the domain of $B$ to the set $S\cap\operatorname{dom}(\Psi)$. To do this, we first use the fact that for any $a\in S'$, there is $\hat{u}\in \mathcal{D}$ such that $a = \rho L^*u$. We can then define $B(a) \triangleq \hat{B}(\hat{u})$. This definition is well-defined, as the result is independent of the choice of $\hat{u}\in \mathcal{D}$. This extension is consistent with the original operator, since for any $a\in S'\cap\mathcal{C}^\infty_c(\Omega\cup\Sigma)$, Remark~\ref{P Cc form} confirms that
$$B(a) = \hat{B}(\hat{u}) = 2\dot{H}(\rho L^*u)=2\dot{H}(a).$$

We summarize the key properties of the densely defined operator $B: S\to \mathcal{D}^*$:
\begin{itemize}[leftmargin=*]
\item $B$ is a linear operator;
\item The domain of $B$ is $\tilde{S}\triangleq S\cap\operatorname{dom}(\Psi)$;
\item $B(a) = 2\dot{H}(a)$ for all $a\in S\cap\mathcal{C}^\infty_c(\Omega\cup\Sigma)$, and extends by continuity to $S\cap\mathcal{H}_{\rho^{-1}}^2(\Omega)$;
\item $B(a) = \hat{B}(\hat{u})$ for all $a\in S'$, where $a = \rho L^*u$. Moreover, $B$ is bounded on $S'$:
\begin{equation}\label{B bounded}
\|B(a)\|_{\mathcal{D}^*} = \|\hat{B}(\hat{u})\|_{\mathcal{D}^*} \le C\|\hat{u}\|_{\mathcal{D}} \le C'\|u\|_{H_{\rho}^2(\Omega)} \le C''\|\rho L^*u\|_{\mathcal{L}_{\rho^{-1}}^2(\Omega)}.
\end{equation}
\end{itemize}

\begin{rmk}\label{rmk: B extension}
In fact, a similar bound for $B(\rho L^*u)$ holds even when $L(\rho L^*u) \ne 0$. More precisely, there is a constant $C$ uniform for metrics $\mathcal{C}^{k+4}$-near $g$ such that  
\begin{equation}\label{B bounded2}
\|B(\rho L^*u)\|_{\mathcal{D}^*} \le C\left(\|u\|_{H_{\rho}^2(\Omega)} + \|L(\rho L^*u)\|_{\left(H_{\nu, \rho}^{2}(\Omega)\right)^*}\right).
\end{equation}
\end{rmk}

\bigskip
By definition, we have
$$
\hat{B}(\mathcal{D}) = B(S') \subset B(\tilde{S}) \subset \mathcal{D}^*.
$$
Since $\hat{B}(\mathcal{D})$ has finite codimension in $\mathcal{D}^*$, it follows that $B(\tilde{S})$ also has finite codimension in $\mathcal{D}^*$, and is therefore closed.

\begin{thm}[Surjectivity of $B$]\label{Surjectivity of B}
Assuming $\operatorname{ker}\Phi^*=\{0\}$, we have $B(\tilde{S})=\mathcal{D}^*$.
\end{thm}
\begin{proof}
Assume $\operatorname{ker}\Phi^*=\{0\}$, then $\operatorname{Im}\Psi$ is dense in $\left(H_{\nu, \rho}^{2}(\Omega)\right)^*\times\mathcal{D}^*$. This means, for any $\psi\in\mathcal{D}^*$, there is a sequence $a_j\in \operatorname{dom}(\Psi)$ such that
$$
\Psi(a_j) = \left(L(a_j), B(a_j)\right)\longrightarrow (0, \psi).
$$
By Theorem~\ref{Dirichlet1}, for each $j$, there is a unique solution $u_j \in H_{0, \rho}^{2}(\Omega)$ to 
$$
L(\rho L^*u_j) = L(a_j) \quad \text{in } \Omega,
$$
and 
$$
 \|u_j\|_{H_{\rho}^{2}(\Omega)} \le C\|L(a_j)\|_{\left(H_{\nu, \rho}^{2}(\Omega)\right)^*}.
$$
Here the constant $C$ is independent of $j$. Then 
$$a_j' = a_j-\rho L^*u_j\in \operatorname{dom}(\Psi)$$
is another sequence satisfying
\begin{align*}
\left\{\begin{aligned}
L(a_j') & = 0 \quad &\text{in } &\Omega\\ 
B(a_j') & = B(a_j-\rho L^*u_j) &\text{on } &\Sigma.
\end{aligned}\right.
\end{align*}
In other words, $a_j'$ is a sequence in $\tilde{S}$. In addition, by \eqref{B bounded2},
\begin{align*}
 \|B(a_j') - \psi\|_{\mathcal{D}^*} & \le \|B(a_j) - \psi\|_{\mathcal{D}^*} +  \|B(\rho L^*u_j)\|_{\mathcal{D}^*}\\
 	& \le \|B(a_j) - \psi\|_{\mathcal{D}^*} + C \left(\|u_j\|_{H_{\rho}^2(\Omega)} + \|L(a_j)\|_{\left(H_{\nu, \rho}^{2}(\Omega)\right)^*}\right)\\
	& \le \|B(a_j) - \psi\|_{\mathcal{D}^*} + C'\|L(a_j)\|_{\left(H_{\nu, \rho}^{2}(\Omega)\right)^*}\longrightarrow 0,
\end{align*}
where the constant $C'$ is also independent of $j$. As a result, $B(\tilde{S})$ is dense in $\mathcal{D}^*$.

Since $B(\tilde{S})$ is also closed in $\mathcal{D}^*$, we have $B(\tilde{S})=\mathcal{D}^*$.
\end{proof}

\subsubsection{Structure of solutions}
Theorem~\ref{Surjectivity of B} tells us that for any $\psi\in\mathcal{D}^*$, there is some $a\in \mathcal{L}_{\rho^{-1}}^2(\Omega)$ such that
\begin{align}\label{solution1}
\left\{\begin{aligned}
L(a) & = 0 \quad &\text{in } &\Omega\\ 
B(a) & = \psi &\text{on } &\Sigma.
\end{aligned}\right.
\end{align}
Note that we are not claiming uniqueness of solutions. Having established the existence of weak solutions to (\ref{solution1}), we now turn to a more detailed analysis of their structure. 

First, for the space $\hat{B}(\mathcal{D})\subset\mathcal{D}^*$, we have the following as a direct consequence of Theorem~\ref{Fredholm}.
\begin{cor}\label{bounded inverse B}
The operator $\hat{B}^{-1}: \hat{B}(\mathcal{D})\to \mathcal{D}/\operatorname{ker}\hat{B}$ is linear and bounded.
\end{cor}
\begin{proof}
As a closed subspace of $\mathcal{D}^*$, $\hat{B}(\mathcal{D})$ is also a Hilbert space. Then it follows directly from the Bounded Inverse Theorem.
\end{proof}

For simplicity, we will always choose the representative element $\hat{u}$ to lie in the orthogonal complement of $\operatorname{ker}\hat{B}$. This ensures that the inverse operator $\hat{B}^{-1}$ is well-defined as an element of $\mathcal{L}(\hat{B}(\mathcal{D}), (\operatorname{ker}\hat{B})^\perp)$, and the mapping is given by:
$$
\psi = \hat{B}(\hat{u}) \xmapsto{\hat{B}^{-1}} \hat{u}.
$$
For a general $\hat{u}\in\mathcal{D}$, the inverse can be defined to be the projection of $\hat{u}$ onto $(\operatorname{ker}\hat{B})^\perp$. Denoting this projection by $\hat{u}_0 = \Pi_{(\operatorname{ker}\hat{B})^\perp} \hat{u}$, the mapping becomes:
$$
\psi = \hat{B}(\hat{u}) \xmapsto{\hat{B}^{-1}} \hat{u}_0.
$$
Furthermore, since it is an orthogonal projection, we have $\|\hat{u}_0\|_{\mathcal{D}} \le \|\hat{u}\|_{\mathcal{D}}$.

\bigskip
We now construct a basis for a subspace in $\mathcal{D}^*$ complementary to $\hat{B}(\mathcal{D})$. By Theorem~\ref{Fredholm}, there exists a constant $p \in \mathbb{N}$, uniform for all metrics $\mathcal{C}^{k+4}$-near $g$, such that $\operatorname{codim}\hat{B}(\mathcal{D}) = \operatorname{codim}B(S') = p$. (We assume $p>0$; otherwise, if $p=0$, then $\hat{B}$ is surjective and we are done.) Consequently, any such complementary subspace must be $p$-dimensional. The following theorem formalizes this construction.

\begin{thm}\label{basis}
Assuming $\operatorname{ker}\Phi^*=\{0\}$, there exist tensors $a_1, \dots, a_p\in (S\setminus S')\cap\mathcal{B}_{2}(\Omega)$ whose images under the operator $B$ form a basis for a complementary subspace to $\hat{B}(\mathcal{D})$.

Specifically, $B(a_1), \dots, B(a_p)\in \mathcal{B}_1(\Sigma)$ are linearly independent and provide the direct sum decomposition:
$$\hat{B}(\mathcal{D})\oplus\operatorname{span}\{B(a_1), \dots, B(a_p) \} = \mathcal{D}^*.$$
\end{thm}

\begin{proof}
Lemma~\ref{lem:density} allows us to choose a tensor $a\in \mathcal{C}_c^\infty(\Omega\cup\Sigma)$ with the following properties:
\begin{itemize}
\item $\|L(a)\|_{\left(H_{\nu, \rho}^{2}(\Omega)\right)^*}$ is arbitrarily small;
\item $B(a)\notin\hat{B}(\mathcal{D})$, and is normalized such that $\|\Pi_{(\hat{B}(\mathcal{D}))^\perp} B(a)\|_{\mathcal{D}^*} = 1$, where $\Pi_{(\hat{B}(\mathcal{D}))^\perp}$ denotes the orthogonal projection onto $(\hat{B}(\mathcal{D}))^\perp$.
\end{itemize}
The existence of such an $a$ is guaranteed, as its non-existence would imply that $\hat{B}(\mathcal{D}) = \mathcal{D}^*$, contradicting $\operatorname{codim}\hat{B}(\mathcal{D}) = p$.

According to Theorem~\ref{Dirichlet1}, the equation $L(\rho L^*u) = L(a)$ has a unique solution $u \in H_{0, \rho}^{2}(\Omega)$ that satisfies the estimate
$$
 \|u\|_{H_{\rho}^{2}(\Omega)} \le C\|L(a)\|_{\left(H_{\nu, \rho}^{2}(\Omega)\right)^*}.
$$
Furthermore, given that $L(a)\in C_c^\infty(\Omega\cup\Sigma)$, Theorem~\ref{Global Schauder estimates} grants the additional regularity $u\in C_{\phi, \phi^{\frac{n}2}\rho^{\frac12}}^{k+4, \alpha}(\Omega\cup\Sigma)$.

We then define $a' = a-\rho L^*u\in \mathcal{B}_{2}(\Omega)$. By construction, $L(a') = L(a)-L(\rho L^*u)=0$, which means $a'\in S$. In addition,
\begin{align*}
\left\|\Pi_{\left(\hat{B}(\mathcal{D})\right)^\perp} B(a')\right\|_{\mathcal{D}^*} & \ge \left\|\Pi_{\left(\hat{B}(\mathcal{D})\right)^\perp} B(a)\right\|_{\mathcal{D}^*} -  \left\|\Pi_{\left(\hat{B}(\mathcal{D})\right)^\perp} B(\rho L^*u)\right\|_{\mathcal{D}^*}\\
 	& \ge 1 - \left\| B(\rho L^*u)\right\|_{\mathcal{D}^*}\\
	& \ge 1 - C \left(\|u\|_{H_{\rho}^2(\Omega)} + \|L(a)\|_{\left(H_{\nu, \rho}^{2}(\Omega)\right)^*}\right)\\
	& \ge 1- C'\|L(a)\|_{\left(H_{\nu, \rho}^{2}(\Omega)\right)^*}>\frac12.
\end{align*}
By choosing $\|L(a)\|_{\left(H_{\nu, \rho}^{2}(\Omega)\right)^*}$ to be sufficiently small, we ensure this final inequality holds. This means $B(a')\notin\hat{B}(\mathcal{D})$, and consequently, $a'\notin S'$.

We have thus constructed a tensor $a' \in (S\setminus S')\cap\mathcal{B}_{2}(\Omega)$ such that $B(a')\notin\hat{B}(\mathcal{D})$. This procedure can be repeated. Let $a_1 = a'$. We can now apply the same argument to the subspace $\hat{B}(\mathcal{D})\oplus\operatorname{span}\{B(a_1)\}$, which has codimension $p-1$, to find a second tensor $a_2 \in (S\setminus S')\cap\mathcal{B}_{2}(\Omega)$. Iterating this process $p$ times yields a set $\{a_1, \dots, a_p\}\subset (S\setminus S')\cap\mathcal{B}_{2}(\Omega)$. By construction, the images $B(a_1), \dots, B(a_p)\in \mathcal{B}_1(\Sigma)$ are linearly independent and span a complement to $\hat{B}(\mathcal{D})$, thus completing the proof. 
\end{proof}

For simplicity, we denote by $W \triangleq \operatorname{span}\{B(a_1), \dots, B(a_p) \}$ the complementary subspace. Let $S_p \triangleq \operatorname{span}\{a_1, \dots, a_p \} \subset S$. We now define a linear map $T_0: W  \rightarrow S_p$ by setting $T_0(B(a_j)) = a_j$ for $j=1, \dots, p$, and extending linearly. Explicitly, for any $\sum c_j B(a_j)\in W$, we have
$$
T_0\left(\sum c_j B(a_j)\right) = \sum c_j a_j.
$$
This map $T_0$ is bounded.

\bigskip
Therefore, we can construct a bounded linear map $T: \mathcal{D}^*\to S$ by
\begin{center}
\begin{tikzcd}[column sep = 0]
  \mathcal{D}^*\arrow[d, "T"]
  &=
  & \hat{B}(\mathcal{D})\arrow[d, "\rho L^*\circ\hat{B}^{-1}"] 
  &\oplus 
  &W\arrow[d, "T_0"]\\
  S
  &\supset
  & S'
  &\oplus
  & S_p
\end{tikzcd}
\end{center}
More precisely:
\begin{itemize}[leftmargin=*]
\item[1.] For any $\psi\in\hat{B}(\mathcal{D})\subset\mathcal{D}^*$, the map $T$ is defined in the following way, where \textbf{DP} indicates the Dirichlet problem (\ref{Dirichlet}):
\begin{equation}\label{T operator}
    \psi = \hat{B}(\hat{u}) \xmapsto{\hat{B}^{-1}} \hat{u} \xmapsto{\text{DP}} u \xmapsto{\rho L^*} \rho L^*u,
\end{equation}
and it is easy to see that $L(\rho L^*u) = 0, B(\rho L^*u) = \hat{B}(\hat{u}) = \psi$. Moreover, by Corollary~\ref{bounded inverse B}, there is a constant $C = C(n, g, \Omega, \Sigma, \rho)$ uniform for metrics $\mathcal{C}^{k+4}$-near $g$ such that 
\begin{equation}\label{weak}
    \|\rho L^*u\|_{\mathcal{L}^2_{\rho^{-1}}(\Omega)} \le C'\|u\|_{H^2_{\rho}(\Omega)} \le C''\|\hat{u}\|_{\mathcal{D}}\le C\|\psi\|_{\mathcal{D}^*}.
\end{equation}

\item[2.] For any $\psi\in W\subset\mathcal{D}^*$, the map $T$ behaves like $T_0$, 
$$
\psi = \sum c_j B(a_j) \xmapsto{T_0} \sum c_j a_j,
$$
and it is easy to see that $L\left(\sum c_j a_j\right) = 0, B\left(\sum c_j a_j\right) = \psi$.
\end{itemize}

In conclusion, $B\circ T$ is the identity on $\mathcal{D}^*$, i.e. $T$ is a right inverse for $B$. As a result, for any $\psi\in\mathcal{D}^*$, $T\psi\in S$ is a solution to (\ref{solution1}). Moreover, there is a constant $C = C(n, g, \Omega, \Sigma, \rho)$ uniform for metrics $\mathcal{C}^{k+4}$-near $g$ such that 
\begin{equation}\label{T weak estimate}
    \|T\psi\|_{\mathcal{L}^2_{\rho^{-1}}(\Omega)} \le C\|\psi\|_{\mathcal{D}^*}.
\end{equation}

\bigskip
We are ready to solve the general case and get the surjectivity of $\Psi$.

\begin{thm}[Surjectivity of $\Psi$]\label{non-homogeneous BVP}
Assuming $\operatorname{ker}\Phi^*=\{0\}$, for any $(f, \psi) \in \left(H_{\nu, \rho}^{2}(\Omega)\right)^*\times\mathcal{D}^*$, there is a solution $a\in \mathcal{L}_{\rho^{-1}}^2(\Omega)$ such that
\begin{align}\label{solution2}
\left\{\begin{aligned}
L(a) & = f \quad &\text{in } &\Omega\\ 
B(a) & = \psi &\text{on } &\Sigma.
\end{aligned}\right.
\end{align}
Moreover, there is a constant $C = C(n, g, \Omega, \Sigma, \rho)$ uniform for metrics $\mathcal{C}^{k+4}$-near $g$ such that 
\begin{equation}\label{soln regularity}
    \|a\|_{\mathcal{L}^2_{\rho^{-1}}(\Omega)} \le C\left(\|f\|_{\left(H_{\nu, \rho}^{2}(\Omega)\right)^*} + \|\psi\|_{\mathcal{D}^*}\right).
\end{equation}
\end{thm}
\begin{proof}
We construct a solution to (\ref{solution2}) using the principle of superposition.

Let $a_0\in\mathcal{L}_{\rho^{-1}}^2(\Omega)$ be any solution to $L(a_0)=f$. For a concrete choice, we take $a_0 = \rho L^*u_0$, where $u_0$ is the unique solution to the Dirichlet problem with zero boundary data (\ref{Dirichlet eqn2 weak}). Then using the result we get for (\ref{solution1}), we can solve 
\begin{align*}
\left\{\begin{aligned}
L(a) & = 0 \quad &\text{in } &\Omega\\ 
B(a) & = \psi - B(a_0) &\text{on } &\Sigma
\end{aligned}\right.
\end{align*}
and obtain a solution $a_1 = T(\psi - B(a_0))\in S$. 

By linearity, $a_2 = a_0 + a_1\in\mathcal{L}_{\rho^{-1}}^2(\Omega)$ is a desired solution of (\ref{solution2}). 
Therefore, for any $(f, \psi) \in \left(H_{\nu, \rho}^{2}(\Omega)\right)^*\times\mathcal{D}^*$, a solution to (\ref{solution2}) is given by
\begin{equation}\label{weak solution}
    a_2 = \rho L^*u_0 + T\left(\psi - B(\rho L^*u_0)\right),
\end{equation}
where $u_0$ is the unique solution of the Dirichlet problem (\ref{Dirichlet eqn2 weak}). Moreover,
\begin{align*}
    \|a_2\|_{\mathcal{L}^2_{\rho^{-1}}(\Omega)} & \le \|\rho L^*u_0\|_{\mathcal{L}^2_{\rho^{-1}}(\Omega)} + \|T\left(\psi - B(\rho L^*u_0)\right)\|_{\mathcal{L}^2_{\rho^{-1}}(\Omega)}\\
    & \le C\left(\|u_0\|_{H^2_\rho(\Omega)} + \|\psi - B(\rho L^*u_0)\|_{\mathcal{D}^*}\right)\\
    & \le C\left(\|u_0\|_{H^2_\rho(\Omega)} + \|\psi\|_{\mathcal{D}^*} + \|u_0\|_{H^2_\rho(\Omega)} + \|f\|_{\left(H_{\nu, \rho}^{2}(\Omega)\right)^*}\right).
\end{align*}
Combined with Theorem \ref{Dirichlet1}, we get
\begin{equation*}
    \|a_2\|_{\mathcal{L}^2_{\rho^{-1}}(\Omega)} \le C\left(\|f\|_{\left(H_{\nu, \rho}^{2}(\Omega)\right)^*} + \|\psi\|_{\mathcal{D}^*}\right).
\end{equation*}
\end{proof}

\subsection{Weighted Schauder estimates}\label{sec4.6}
In this subsection, we improve the regularity of the weak solution (\ref{weak solution}) and get the following Schauder estimates.

\begin{thm}\label{weighted Schauder}
Let $g_0$ be a $\mathcal{C}^{k+4,\alpha}$-metric such that the operator $\Phi^*$ has trivial kernel in $H^2_{\operatorname{loc}}(\Omega\cup\Sigma)$. Then there is a constant $C$ uniform for metrics near $g_0$ in $\mathcal{C}^{k+4,\alpha}(\overline{\Omega})$ such that for $(f, \psi)\in \mathcal{B}_0(\Omega)\times\mathcal{B}_{1}(\Sigma)$, if $a\in\mathcal{L}_{\rho^{-1}}^2(\Omega)$ is a weak solution of (\ref{solution2}) given by (\ref{weak solution}), then $a\in\mathcal{B}_2(\Omega)$ and
\begin{equation}
    \|a\|_{\mathcal{B}_2(\Omega)} \le C \|(f, \psi)\|_{\mathcal{B}_0(\Omega)\times\mathcal{B}_{1}(\Sigma)}.
\end{equation}
\end{thm}

We will specify the dependence of the constant $C$ in the following discussion.

We note that system (\ref{solution2}) is not elliptic, so standard regularity theory does not directly apply. Nevertheless, we can obtain the desired estimates by leveraging the structure of the weak solutions. As shown in Section~\ref{sec: weak soln}, these solutions take the form $a=\rho L^*u$ modulo a finite-dimensional subspace. For the $a=\rho L^*u$ component, the system becomes an elliptic boundary value problem. Meanwhile, the finite-dimensional subspace can be chosen from $\mathcal{B}_{2}(\Omega)$. This decomposition provides a path to the required estimates.

Due to the complexity of the solution $(\ref{weak solution})$, we will improve its regularity term by term. Following the same strategy used for the weak solution, we will first analyze the Dirichlet problem $(\ref{Dirichlet eqn2 weak})$ and then study the operator $T$.

We adopt the local analysis framework from \cite{C-H} to define coordinate systems and function pullbacks. For any interior point $x\in\Omega$, consider the ball $B_{\phi(x)}(x)$ centered at $x$ with a radius given by the weight function $\phi(x)$ from Section~\ref{weighted Holder space}. We map the unit ball $B_1(0)\subset\mathbb{R}^n$ to this local neighborhood using the diffeomorphism $F_x: B_1(0) \to B_{\phi(x)}(x)$ defined by: 
$$
y = F_x(z) = x + \phi(x)z.
$$
For a boundary point (i.e., $x$ on or near $\Sigma$), we adapt this by considering the intersection $B^+_{\phi(x)}(x) = B_{\phi(x)}(x)\cap(\Omega\cup\Sigma)$ and use a similar map from the unit half-ball $B^+_1(0)$. In both cases, we identify the domain with its coordinate image. The pullback of a function $f$ from the local neighborhood to the corresponding unit (half-)ball is denoted by $\tilde{f}$ and defined as:
$$
\tilde{f}(z) = F_x^*(f)(z) = f \circ F_x(z).
$$

With a minor abuse of notation, we denote for $a \in (0, 1]$ and $x\in\Omega$,
$$
\|f\|_{C_{\phi, \varphi}^{k, \alpha}(B_{a\phi(x)}(x))} = \sum_{j=0}^{k} \varphi(x) \phi^{j}(x)\left\|D^{j} f\right\|_{C^{0}\left(B_{a\phi(x)}(x)\right)}+\varphi(x) \phi^{k+\alpha}(x)\left[D^{k} f\right]_{0, \alpha ; B_{a\phi(x)}(x)};
$$
and for $x$ near $\Sigma$,
$$
\|f\|_{C_{\phi, \varphi}^{k, \alpha}(B^+_{a\phi(x)}(x))} = \sum_{j=0}^{k} \varphi(x) \phi^{j}(x)\left\|D^{j} f\right\|_{C^{0}\left(B^+_{a\phi(x)}(x)\right)}+\varphi(x) \phi^{k+\alpha}(x)\left[D^{k} f\right]_{0, \alpha ; B^+_{a\phi(x)}(x)}.
$$

We have the following useful lemma.
\begin{lem}[Corvino-Huang \cite{C-H}]\label{scaling lemma}
Let $f$ and $g$ be functions defined on $B_{\phi(x)}(x)$. The following properties hold.\\
(i) $\widetilde{f+g} = \tilde{f} + \tilde{g}$ and $\widetilde{fg} = \tilde{f}\tilde{g}$.\\
(ii) $\widetilde{\p^\beta_y f} = (\phi(x))^{-|\beta|}\p^\beta_z \tilde{f}$, where $\beta = (\beta_1, \dots, \beta_k)$ is a multi-index.\\
(iii) For any $a\in(0,1]$,
\begin{align*}
    \|\varphi(x)\tilde{f}\|_{C^{k, \alpha}(B_{a}(0))} & = \|f\|_{C_{\phi, \varphi}^{k, \alpha}(B_{a\phi(x)}(x))}\\
    \|\varphi(x)\tilde{f}\|_{L^2(B_{a}(0))} & = \|f\|_{L_{\phi^{-n}\varphi^2}^2(B_{a\phi(x)}(x))}.
\end{align*}
\end{lem}
\begin{rmk}
For $x$ near $\Sigma$, we also have the properties mentioned above in the half ball.
\end{rmk}

\subsubsection{Weighted Schauder estimates for the Dirichlet problem}
Let us first consider the Dirichlet problem with zero boundary data:
\begin{align}\label{weighted DP}
\left\{\begin{aligned}
\rho^{-1}L(\rho L^*u) & = \rho^{-1}f \quad &\text{in } &\Omega\\ 
u & = 0 &\text{on } &\Sigma\\
u_\nu & = 0 &\text{on } &\Sigma.
\end{aligned}\right.
\end{align}
Following Schauder estimates by scaling \cite{lS4} and using both the properties outlined above and those discussed in Section \ref{weighted Holder space}, we are able to get the interior and boundary Schauder estimates.

The following interior Schauder estimates have already been discussed in \cites{C, C-E-M}:
\begin{lem}[Interior Schauder estimates]
For any $k\in\mathbb{N}$ and any $r, s\in\mathbb{R}$, there is a constant $C$ such that for any $x\in\Omega$,
\begin{equation}\label{interior A-D-N}
    \|u\|_{C_{\phi, \phi^r\rho^s}^{k+4, \alpha}(B_{\phi(x)/2}(x))} \le C\left(\|f\|_{C_{\phi, \phi^{r+4}\rho^{s-1}}^{k, \alpha}(B_{\phi(x)}(x))} + \|u\|_{L_{\phi^{-n}(\phi^{2r}\rho^{2s})}^2(B_{\phi(x)}(x))}\right).
\end{equation}
Here the constant $C$ depends only on $\Omega, g_0, k, n, \rho, \alpha, r, s$, the ellipticity constant and the coefficients of the operator $\rho^{-1}L(\rho L^*u)$, and the constant in Proposition \ref{C-H estimate}. In particular, the constant $C$ is uniform for metrics $\mathcal{C}^{k+4, \alpha}$-near $g_0$.
\end{lem}

Let us now derive the boundary Schauder estimates. To apply \cite{A-D-N} in obtaining these estimates, we must verify that (\ref{weighted DP}) satisfies the three conditions outlined in \cite{A-D-N}*{Section 7}.

\textit{1. Condition on the interior operator $\rho^{-1}L(\rho L^*u)$:}

After simple calculation, we can see that the leading order term is $L' = (n-1)\Delta_{g}^2$, which is uniformly elliptic. For point $x$ on the boundary and $\xi, \nu$ the unit tangent and normal vectors, the following equation
\begin{equation*}
    0 = L'(x; \xi+\tau\nu) = (n-1)\rho\sum_{i,j}\left(\xi+\tau\nu\right)_i^2\left(\xi+\tau\nu\right)_j^2 = (n-1)(1+\tau^2)^2
\end{equation*}
has exactly 2 roots ($\sqrt{-1}$ and $\sqrt{-1}$) with positive imaginary part.

\textit{2. Complementing condition:}

It's mentioned in \cite{A-D-N}*{p.~627} that the boundary operators for the Dirichlet problem satisfy the “complementing condition” relative to any operator for which \textit{Condition 1} holds.

\textit{3. Boundedness of the coefficients:}

We may rewrite the interior operator as
$$
\rho^{-1}L(\rho L^*u) = \sum_{|\beta|=0}^4 b_\beta\partial^\beta u,
$$
where $b_\beta$ is a degree one polynomial of $\rho^{-1} D^l \rho \,\,\, (l = 0, 1, 2; \, l+|\beta|\le 4)$ with coefficients depending locally uniformly on $g_0 \in C^{k+4, \alpha}(\Omega\cup\Sigma)$. By Lemma \ref{scaling lemma} and remark, we have the following estimate.

\begin{lem}
    The coefficients $b_\beta$ satisfy
    $$\|b_\beta\|_{C_{\phi, \phi^{4-|\beta|}}^{k, \alpha}(\Omega\cup\Sigma)} \le C,$$
    where the constant $C$ depends locally uniformly on $g_0 \in \mathcal{C}^{k+4, \alpha}(\Omega\cup\Sigma)$.
\end{lem}

For $x\in\Sigma$, in each $F^{-1}_x\left(B^+_{\phi(x)}(x)\right)$,
$$
    \widetilde{\rho^{-1}f} = \sum_{|\beta|=0}^4 \widetilde{b_\beta}\widetilde{\partial^\beta u} = \sum_{|\beta|=0}^4 \widetilde{b_\beta}(\phi(x))^{-|\beta|}\p^\beta_z \tilde{u}.
$$
Multiplying $(\phi(x))^4$ to the above identity, we have
\begin{equation*}
    (\phi(x))^4\widetilde{\rho^{-1}f} = \sum_{|\beta|=0}^4 \left((\phi(x))^{4-|\beta|}\widetilde{b_\beta}\right)\p^\beta_z\tilde{u}.
\end{equation*}
Here the power of $\phi$ is specifically chosen such that each coefficient $(\phi(x))^{4-|\beta|}\widetilde{b_\beta}$ is bounded in $C^{k, \alpha}(B_1^+(0))$:
$$\|(\phi(x))^{4-|\beta|}\widetilde{b_\beta}\|_{C^{k, \alpha}(B^+_1(0))} = \|b_\beta\|_{C_{\phi, \phi^{4-|\beta|}}^{k, \alpha}(B^+_{\phi(x)}(x))} \le C.$$

Thus for the following re-scaled equations, the three conditions are all satisfied:
\begin{align}\label{A-D-N}
\left\{\begin{aligned}
 \sum_{|\beta|=0}^4 \left((\phi(x))^{4-|\beta|}\widetilde{b_\beta}\right)\p^\beta_z\tilde{u} & = (\phi(x))^4\widetilde{\rho^{-1}f} \quad &\text{in } &B^+_1(0)\\ 
 \tilde{u} & = 0 &\text{on } &\Sigma_0\\
\p_n \tilde{u} & = 0 &\text{on } &\Sigma_0
\end{aligned}\right.
\end{align}
where $\Sigma_0$ is the flat boundary of $B^+_1(0)$. We apply the boundary Schauder estimate \cite{A-D-N} to the equation (\ref{A-D-N}) and obtain
\begin{lem}[Boundary Schauder estimates]
For any $k\in\mathbb{N}$ and any $r, s\in\mathbb{R}$, there is a constant $C$ such that for any $x\in\Sigma$,
\begin{equation}\label{boundary A-D-N}
    \|u\|_{C_{\phi, \phi^r\rho^s}^{k+4, \alpha}(B^+_{\phi(x)/2}(x))} \le C\left(\|f\|_{C_{\phi, \phi^{r+4}\rho^{s-1}}^{k, \alpha}(B^+_{\phi(x)}(x))} + \|u\|_{L_{\phi^{-n}(\phi^{2r}\rho^{2s})}^2(B^+_{\phi(x)}(x))}\right).
\end{equation}
Here the constant $C$ depends only on $\Omega, \Sigma, g_0, k, n, \rho, \alpha, r, s$, the ellipticity constant of the operator $\rho^{-1}L(\rho L^*u)$, $\|b_\beta\|_{C_{\phi, \phi^{4-|\beta|}}^{k, \alpha}(\Omega\cup\Sigma)}$, and the constant in Proposition \ref{C-H estimate}. In particular, the constant $C$ is uniform for metrics $\mathcal{C}^{k+4, \alpha}$-near $g_0$.
\end{lem}
\begin{proof}
    From the boundary Schauder estimate \cite{A-D-N} we get
$$
\|\tilde{u}\|_{C^{k+4, \alpha}(B^+_{1/2}(0))} \le C\left(\|(\phi(x))^4\widetilde{\rho^{-1}f}\|_{C^{k, \alpha}(B^+_1(0))} + \| \tilde{u}\|_{L^2(B^+_1(0))}\right).
$$
For $r, s\in\mathbb{R}$, multiplying $\phi^r(x)\rho^s(x)$ to the above inequality, we have
\begin{equation*}
\begin{aligned}
& \|\phi^r(x)\rho^s(x)\tilde{u}\|_{C^{k+4, \alpha}(B^+_{1/2}(0))} \\
& \qquad \le C\left(\|\phi^{r+4}(x)\rho^s(x)\widetilde{\rho^{-1}f}\|_{C^{k, \alpha}(B^+_1(0))} + \| \phi^r(x)\rho^s(x)\tilde{u}\|_{L^2(B^+_1(0))}\right).
\end{aligned}
\end{equation*}
Thus,
\begin{align*}
\begin{aligned}
    \|u\|_{C_{\phi, \phi^r\rho^s}^{k+4, \alpha}(B^+_{\phi(x)/2}(x))} & \le C\left(\|\rho^{-1}f\|_{C_{\phi, \phi^{r+4}\rho^{s}}^{k, \alpha}(B^+_{\phi(x)}(x))} + \|u\|_{L_{\phi^{-n}(\phi^{2r}\rho^{2s})}^2(B^+_{\phi(x)}(x))}\right)\\
    & \le C\left(\|f\|_{C_{\phi, \phi^{r+4}\rho^{s-1}}^{k, \alpha}(B^+_{\phi(x)}(x))} + \|u\|_{L_{\phi^{-n}(\phi^{2r}\rho^{2s})}^2(B^+_{\phi(x)}(x))}\right).
\end{aligned}
\end{align*}
\end{proof}


We may take the supremum over $x\in\Omega$ and $y\in\Sigma$ and combine (\ref{interior A-D-N}) and (\ref{boundary A-D-N}) to get the following global Schauder estimates.
\begin{thm}[Global Schauder estimates]\label{Global Schauder estimates}
For any $k\in\mathbb{N}$ and any $r, s\in\mathbb{R}$, there is a constant $C$ such that
\begin{equation}\label{global Schauder}
   \|u\|_{C_{\phi, \phi^r\rho^s}^{k+4, \alpha}(\Omega\cup\Sigma)} \le C\left(\|f\|_{C_{\phi, \phi^{r+4}\rho^{s-1}}^{k, \alpha}(\Omega\cup\Sigma)} + \|u\|_{L_{\phi^{-n}(\phi^{2r}\rho^{2s})}^2(\Omega)}\right). 
\end{equation}
Here the constant $C$ depends only on $\Omega, \Sigma, g_0, k, n, \rho, \alpha, r, s$, the ellipticity constant of the operator $\rho^{-1}L(\rho L^*u)$, $\|b_\beta\|_{C_{\phi, \phi^{4-|\beta|}}^{k, \alpha}(\Omega\cup\Sigma)}$, and the constant in Proposition \ref{C-H estimate}. In particular, the constant $C$ is uniform for metrics $\mathcal{C}^{k+4, \alpha}$-near $g_0$.
\end{thm}


In particular, if $u_0$ is the \textbf{unique} solution of the Dirichlet problem with zero boundary data (\ref{Dirichlet eqn2 weak}) for $f\in C_{\phi, \phi^{r+4}\rho^{s-1}}^{k, \alpha}(\Omega\cup\Sigma)$, then (\ref{global Schauder}) becomes
\begin{equation}
    \|u_0\|_{C_{\phi, \phi^r\rho^s}^{k+4, \alpha}(\Omega\cup\Sigma)} \le C\|f\|_{C_{\phi, \phi^{r+4}\rho^{s-1}}^{k, \alpha}(\Omega\cup\Sigma)}.
\end{equation}
Applying Proposition~\ref{operator Holder} to this result yields two key consequences:
\begin{equation}\label{Schauder1}
    \|\rho L^*u_0\|_{C_{\phi, \phi^{r+2}\rho^{s-1}}^{k+2, \alpha}(\Omega\cup\Sigma)} \le C\|f\|_{C_{\phi, \phi^{r+4}\rho^{s-1}}^{k, \alpha}(\Omega\cup\Sigma)},
\end{equation}
\begin{equation}\label{Schauder2}
    \|B(\rho L^*u_0)\|_{C_{\phi, \phi^{r+3}\rho^{s-1}}^{k+1, \alpha}(\Sigma)} \le C\|f\|_{C_{\phi, \phi^{r+4}\rho^{s-1}}^{k, \alpha}(\Omega\cup\Sigma)}.
\end{equation}

\subsubsection{Weighted Schauder estimates for the $T$ operator}
To derive the Schauder estimates for the operator $T$, we recall its definition from (\ref{T operator}) and consider the following boundary value problem:
\begin{align*}
\left\{\begin{aligned}
\rho^{-1}L(\rho L^*u) & = 0 \quad &\text{in } &\Omega\\ 
2\rho^{-1}\dot{H}(\rho L^*u) & = \rho^{-1}\psi &\text{on } &\Sigma\\
u_\nu & = 0 &\text{on } &\Sigma.
\end{aligned}\right.
\end{align*}
Like in the Dirichlet problem, let us start with the three conditions required by \cite{A-D-N}*{Section 7}. 

\textit{1. Condition on the interior operator $\rho^{-1}L(\rho L^*u)$:} Same as the Dirichlet problem.

\textit{2. Complementing condition:}

The principal symbols of the boundary operators are:
$$
B_1' = -n\sum_{i=1}^{n-1} \partial_i^2\partial_\nu + (1-n) \partial_\nu^3,\quad\text{ and }\quad B_2' = \partial_\nu.
$$
$B_1'(x; \xi+\tau\nu)$ mod $(\tau - \sqrt{-1})^2$ is
\begin{align*}
    B_1'(x; \xi+\tau\nu) & = -n\sum_{i=1}^{n-1} (\xi+\tau\nu)_i^2(\xi+\tau\nu)_n + (1-n) (\xi+\tau\nu)_n^3\\
    & = -n\sum_{i=1}^{n-1} \xi_i^2\tau + (1-n) \tau^3\\
    & = -n\tau + (1-n) (2\sqrt{-1}-3\tau)\\
    & = 2(n-1)(\tau - \sqrt{-1})-\tau,
\end{align*}
while $B_2'(x; \xi+\tau\nu)$ mod $(\tau - \sqrt{-1})^2$ is $\tau$. 

Thus $B_1', B_2'$ are linearly independent, and the \textit{Complementing condition} is satisfied.

\textit{3. Boundedness of the coefficients:} 

We may rewrite the boundary operator as
$$
2\rho^{-1}\dot{H}(\rho L^*u) = \sum_{|\gamma|=0}^3 c_\gamma\partial^\gamma u
$$
where $c_\gamma$ is a degree one polynomial of $\rho^{-1} D^l \rho \,\,\, (l = 0, 1; \, l+|\gamma|\le 3)$ with coefficients depending locally uniformly on $g_0 \in C^{k+4, \alpha}(\Omega\cup\Sigma)$. By Lemma \ref{scaling lemma} and remark, we have the following estimate.

\begin{lem}
    The coefficients $c_\gamma$ satisfy
    $$\|c_\gamma\|_{C_{\phi, \phi^{3-|\gamma|}}^{k+1, \alpha}(\Sigma)} \le C,$$
    where the constant $C$ depends locally uniformly on $g_0 \in \mathcal{C}^{k+4, \alpha}(\Omega\cup\Sigma)$.
\end{lem}

For $x\in\Sigma$, in each $F^{-1}_x\left(B^+_{\phi(x)}(x)\right)$,
$$
    \widetilde{\rho^{-1}\psi} = \sum_{|\gamma|=0}^3 \widetilde{c_\gamma}\widetilde{\partial^\gamma u} = \sum_{|\gamma|=0}^3 \widetilde{c_\gamma}(\phi(x))^{-|\gamma|}\p^\gamma_z \tilde{u}.
$$
Multiplying $(\phi(x))^3$ to the above identity, we have
\begin{equation*}
    (\phi(x))^3\widetilde{\rho^{-1}\psi} = \sum_{|\gamma|=0}^3 \left((\phi(x))^{3-|\gamma|}\widetilde{c_\gamma}\right)\p^\gamma_z\tilde{u}.
\end{equation*}
Here the power of $\phi$ is specifically chosen such that each coefficient $(\phi(x))^{3-|\gamma|}\widetilde{c_\gamma}$ is bounded in $C^{k+1, \alpha}(B_1^+(0))$:
$$\|(\phi(x))^{3-|\gamma|}\widetilde{c_\gamma}\|_{C^{k+1, \alpha}(\Sigma_0)} = \|c_\gamma\|_{C_{\phi, \phi^{3-|\gamma|}}^{k+1, \alpha}(B_{\phi(x)}(x)\cap\Sigma)} \le C.$$

Thus for the following re-scaled equations, the three conditions are all satisfied:
\begin{align}\label{A-D-N2}
\left\{\begin{aligned}
 \sum_{|\beta|=0}^4 \left((\phi(x))^{4-|\beta|}\widetilde{b_\beta}\right)\p^\beta_z\tilde{u} & = 0 \quad &\text{in } &B^+_1(0)\\ 
 \sum_{|\gamma|=0}^3 \left((\phi(x))^{3-|\gamma|}\widetilde{c_\gamma}\right)\p^\gamma_z\tilde{u} & = (\phi(x))^3\widetilde{\rho^{-1}\psi} &\text{on } &\Sigma_0\\
\p_n \tilde{u} & = 0 &\text{on } &\Sigma_0.
\end{aligned}\right.
\end{align}

We apply the boundary Schauder estimate \cite{A-D-N} to the equation (\ref{A-D-N2}) and obtain
\begin{lem}[Boundary Schauder estimates 2]
For any $k\in\mathbb{N}$ and any $r, s\in\mathbb{R}$, there is a constant $C$ such that for any $x\in\Sigma$,
\begin{equation}\label{boundary A-D-N 2}
\|u\|_{C_{\phi, \phi^r\rho^s}^{k+4, \alpha}(B^+_{\phi(x)/2}(x))} \le C\left(\|\psi\|_{C_{\phi, \phi^{r+3}\rho^{s-1}}^{k+1, \alpha}(B_{\phi(x)}(x)\cap\Sigma)} + \|u\|_{L_{\phi^{-n}(\phi^{2r}\rho^{2s})}^2(B^+_{\phi(x)}(x))}\right).
\end{equation}
Here the constant $C$ depends only on $\Omega, \Sigma, g_0, k, n, \rho, \alpha, r, s$, the ellipticity constant of the operator $\rho^{-1}L(\rho L^*u)$, $\|b_\beta\|_{C_{\phi, \phi^{4-|\beta|}}^{k, \alpha}(\Omega\cup\Sigma)}$, $\|c_\gamma\|_{C_{\phi, \phi^{3-|\gamma|}}^{k+1, \alpha}(\Sigma)}$, and the constant in Proposition \ref{C-H estimate}. In particular, the constant $C$ is uniform for metrics $\mathcal{C}^{k+4, \alpha}$-near $g_0$.
\end{lem}

Now we may take the supremum over $x\in\Omega$ and $y\in\Sigma$ and combine (\ref{interior A-D-N}) and (\ref{boundary A-D-N 2}) to get the following global Schauder estimates.
\begin{thm}[Global Schauder estimates 2]\label{Global Schauder2}
For any $k\in\mathbb{N}$ and any $r, s\in\mathbb{R}$, there is a constant $C$ as above such that
\begin{equation}\label{global Schauder 2}
   \|u\|_{C_{\phi, \phi^r\rho^s}^{k+4, \alpha}(\Omega\cup\Sigma)} \le C\left(\|\psi\|_{C_{\phi, \phi^{r+3}\rho^{s-1}}^{k+1, \alpha}(\Sigma)} + \|u\|_{L_{\phi^{-n}(\phi^{2r}\rho^{2s})}^2(\Omega)}\right). 
\end{equation}
\end{thm}

To complete the Schauder estimates, we set $r=\frac{n}2$ and $s=\frac12$. This yields the following estimate
\begin{equation*}
   \|u\|_{C_{\phi, \phi^{\frac{n}2}\rho^{\frac12}}^{k+4, \alpha}(\Omega\cup\Sigma)} \le C\left(\|\psi\|_{C_{\phi, \phi^{\frac{n}2+3}\rho^{-\frac12}}^{k+1, \alpha}(\Sigma)} + \|u\|_{L_{\rho}^2(\Omega)}\right). 
\end{equation*}
We now combine this with the estimate from $(\ref{weak})$ for weak solutions to control the full $\mathcal{B}_4(\Omega)$ norm of $u$.
\begin{equation*}
\begin{aligned}
    \|u\|_{\mathcal{B}_4(\Omega)} & = \|u\|_{C_{\phi, \phi^{\frac{n}2}\rho^{\frac12}}^{k+4, \alpha}(\Omega\cup\Sigma)} + \|u\|_{H^2_{\rho}(\Omega)}\\
    & \le C\left(\|\psi\|_{C_{\phi, \phi^{\frac{n}2+3}\rho^{-\frac12}}^{k+1, \alpha}(\Sigma)} + \|u\|_{L_{\rho}^2(\Omega)}\right) + \|u\|_{H^2_{\rho}(\Omega)}\\
    & \le C\left(\|\psi\|_{C_{\phi, \phi^{\frac{n}2+3}\rho^{-\frac12}}^{k+1, \alpha}(\Sigma)} + \|\psi\|_{\mathcal{D}^*}\right)\\
    & =C\|\psi\|_{\mathcal{B}_1(\Sigma)}. 
\end{aligned}
\end{equation*}
Then from Proposition~\ref{operator Holder}, we immediately get
\begin{equation*}
    \|\rho L^*u\|_{\mathcal{B}_2(\Omega)} \le C\|\psi\|_{\mathcal{B}_1(\Sigma)}. 
\end{equation*}

In other words, when restricted to the space $\hat{B}(\mathcal{D})\cap C_{\phi, \phi^{\frac{n}2+3}\rho^{-\frac12}}^{k+1, \alpha}(\Sigma)\subset\mathcal{B}_1(\Sigma)$, the operator $T$ has the following estimate:
\begin{equation*}
    \|T\psi\|_{\mathcal{B}_2(\Omega)} = \|\rho L^*u\|_{\mathcal{B}_2(\Omega)} \le C\|\psi\|_{\mathcal{B}_1(\Sigma)}.  
\end{equation*}

\bigskip
Finally, we extend this result to the entire space $\mathcal{B}_1(\Sigma)$. By Theorem~\ref{basis}, the complement 
$$W\cap C_{\phi, \phi^{\frac{n}2+3}\rho^{-\frac12}}^{k+1, \alpha}(\Sigma) = W \subset \mathcal{B}_1(\Sigma)$$ 
is finite-dimensional, which ensures the above estimate holds universally for a large constant $C$. This yields our final desired bound:
\begin{equation}\label{Schauder3}
     \|T\psi\|_{\mathcal{B}_2(\Omega)} \le C\|\psi\|_{\mathcal{B}_1(\Sigma)}\quad\text{ for any } \psi\in\mathcal{B}_1(\Sigma).
\end{equation}

We are now ready to prove the main result of this section.
\begin{proof}[Proof of Theorem \ref{weighted Schauder}]
Together with (\ref{soln regularity}), (\ref{Schauder1}), (\ref{Schauder2}), and (\ref{Schauder3}), we get
\begin{equation*}
\begin{aligned}
     \|a\|_{\mathcal{B}_2(\Omega)} & \le \|\rho L^*u_0\|_{\mathcal{B}_2(\Omega)} + \|T\left(\psi - B(\rho L^*u_0)\right)\|_{\mathcal{B}_2(\Omega)}\\
    & \le C\left(\|f\|_{\mathcal{B}_0(\Omega)} + \|\psi - B(\rho L^*u_0)\|_{\mathcal{B}_1(\Sigma)}\right)\\
    & \le C\left(\|f\|_{\mathcal{B}_0(\Omega)} + \|\psi\|_{\mathcal{B}_1(\Sigma)} + \|f\|_{\mathcal{B}_0(\Omega)}\right)\\
    & \le C\left(\|f\|_{\mathcal{B}_0(\Omega)} + \|\psi\|_{\mathcal{B}_1(\Sigma)}\right).
\end{aligned}
\end{equation*}
\end{proof}

\subsection{The nonlinear problem and iteration}\label{sec4.7}
We now use Picard's iteration to get a solution for the nonlinear problem. Given an initial metric $g_0$, 
\begin{itemize}
\item a function $R'$ close to $R(g_0)$ and equal to $R(g_0)$ outside $\Omega$,
\item a function $H'$ close to $H(g_0)$ on $\Sigma$ and equal to $H(g_0)$ outside $\Sigma$ (on $\partial M$),
\end{itemize}
we would like to find a tensor $a\in\mathcal{S}^{(0,2)}$ with $R(g_0+a) = R', H(g_0+a) = H'$ and $a \equiv 0$ outside $\Omega\cup\Sigma$, assuming $\Phi_{g_0}^*$ has trivial kernel.

Recall $L_{g_0}$ is the linearization of the scalar curvature at $g_0$, and $\dot{H}_{g_0} = \frac12B_{g_0}$ is the linearization of the mean curvature of $\Sigma$ at $g_0$. Consider the Taylor expansion of the scalar curvature and the mean curvature at $g_0$, where $(Q_{g_0,1}(a), Q_{g_0,2}(a))^t$ is the remainder term:
$$\left(
\begin{matrix} 
R(g_0+a)\\
H(g_0+a)
\end{matrix}
\right) = \left(
\begin{matrix} 
R(g_0) + L_{g_0}(a)\\
H(g_0) + \frac12B_{g_0}(a)
\end{matrix}
\right) + \left(
\begin{matrix} 
Q_{g_0,1}(a)\\
Q_{g_0,2}(a)
\end{matrix}
\right).$$
For a $C^{k+4,\alpha}$-metric $g_0$, it's easy to see that 
$$
\|Q_{g_0,1}(a)\|_{k,\alpha(\Omega)} + \|Q_{g_0,2}(a)\|_{k+1,\alpha(\Sigma)}\le C\|a\|^2_{k+2,\alpha}
$$
where $C$ is a constant uniform for metrics near $g_0$ in $C^{k+4,\alpha}$. In fact, we have the following estimates involving the weights, which are similar to Corvino's \cite{C} and results in some follow up works \cites{C-E-M, C-H, C-S} (with refined weighted spaces).
\begin{prop}\label{base}
Let $g_0$ be a $\mathcal{C}^{k+4, \alpha}$-metric such that $\operatorname{ker}\Phi_{g_0}^*=\{0\}$. Then there is a constant $C$ (as in Theorem \ref{weighted Schauder}) uniform for metrics near $g_0$ in $\mathcal{C}^{k+4, \alpha}$, and an  $\epsilon> 0$ (sufficiently small) such that if
\begin{itemize}
\item[1.] $R'\in C^{k, \alpha}(\Omega\cup\Sigma)$ with $(R' - R(g_0))\in \mathcal{B}_{0}(\Omega)$;
\item[2.] $H'\in C^{k+1, \alpha}(\Sigma)$ with $(H' - H(g_0))\in \mathcal{B}_{1}(\Sigma)$;
\item[3.] and $\|R' - R(g_0)\|_{\mathcal{B}_{0}(\Omega)} + \|H' - H(g_0)\|_{\mathcal{B}_{1}(\Sigma)} < \epsilon,$
\end{itemize}
then upon solving $\Psi_{g_0}(a_0) = (R' - R(g_0), 2(H' - H(g_0)))$ via the previous method (given by (\ref{weak solution})), and letting $g_1 = g_0 + a_0$, we have $$\left\|a_{0}\right\|_{\mathcal{B}_2(\Omega)} \leq C(\|R' - R(g_0)\|_{\mathcal{B}_{0}(\Omega)} + \|H' - H(g_0)\|_{\mathcal{B}_{1}(\Sigma)})<C \epsilon$$ 
and 
\begin{equation*}
\begin{aligned}
\|R' - R(g_1)\|_{\mathcal{B}_{0}(\Omega)} & + \|H' - H(g_1)\|_{\mathcal{B}_{1}(\Sigma)} \\
& \leq C(\|R' - R(g_0)\|_{\mathcal{B}_{0}(\Omega)} + \|H' - H(g_0)\|_{\mathcal{B}_{1}(\Sigma)})^{2}<C \epsilon^{2}.
\end{aligned}
\end{equation*}
Moreover, the metric $g_1$ is $\mathcal{C}^{k+2,\alpha}$.
\end{prop}
\begin{proof}
    The first inequality is a direct consequence of Theorem \ref{weighted Schauder}. If $g_1 = g_0 + a_0$, we have
    $$\left(
\begin{matrix} 
R(g_1)\\
H(g_1)
\end{matrix}
\right) = \left(
\begin{matrix} 
R(g_0) + L_{g_0}(a_0)\\
H(g_0) + \frac12B_{g_0}(a_0)
\end{matrix}
\right) + \left(
\begin{matrix} 
Q_{g_0,1}(a_0)\\
Q_{g_0,2}(a_0)
\end{matrix}
\right)= \left(
\begin{matrix} 
R'\\
H'
\end{matrix}
\right) + \left(
\begin{matrix} 
Q_{g_0,1}(a_0)\\
Q_{g_0,2}(a_0)
\end{matrix}
\right).$$
So
\begin{equation*}
\begin{aligned}
    \|R' - R(g_1)\|_{\mathcal{B}_{0}(\Omega)} & + \|H' - H(g_1)\|_{\mathcal{B}_{1}(\Sigma)} \leq C \left\|a_{0}\right\|^2_{\mathcal{B}_2(\Omega)}\\
    & \le C(\|R' - R(g_0)\|_{\mathcal{B}_{0}(\Omega)} + \|H' - H(g_0)\|_{\mathcal{B}_{1}(\Sigma)})^{2}<C \epsilon^{2}.
\end{aligned}
\end{equation*}
\end{proof}

We then proceed recursively as in the following theorem. Note that we only linearize about $g_0$.

\begin{thm}
Let $g_0, C$ and $\epsilon$ be as above. Then there is a constant $\delta \in (0,1)$ such that the following holds:\\
Suppose that $m \ge 1$ and we have recursively constructed $a_0, \dots, a_{m-1}\in\mathcal{B}_2(\Omega)$ and that $g_0, \dots, g_{m}$ are $\mathcal{C}^{k+2, \alpha}$-metrics, where $a_j$ is a solution of $\Psi_{g_0}(a_j) = (R' - R(g_j), 2(H' - H(g_j)))$ given by (\ref{weak solution}) and $g_j = g_{j-1} + a_{j-1}$. Assume $\|R' - R(g_0)\|_{\mathcal{B}_{0}(\Omega)} + \|H' - H(g_0)\|_{\mathcal{B}_{1}(\Sigma)} < \epsilon$ and that for all $0\le j\le m-1$,
$$\left\|a_{j}\right\|_{\mathcal{B}_2(\Omega)} \leq C(\|R' - R(g_0)\|_{\mathcal{B}_{0}(\Omega)} + \|H' - H(g_0)\|_{\mathcal{B}_{1}(\Sigma)})^{(1+j\delta)},$$
and for all $0\le j\le m$,
\begin{equation*}
\begin{aligned}
\|R' - R(g_j)\|_{\mathcal{B}_{0}(\Omega)} & + \|H' - H(g_j)\|_{\mathcal{B}_{1}(\Sigma)} \\
& \leq (\|R' - R(g_0)\|_{\mathcal{B}_{0}(\Omega)} + \|H' - H(g_0)\|_{\mathcal{B}_{1}(\Sigma)})^{(1+j\delta)}.
\end{aligned}
\end{equation*}
Then the iteration can proceed to the next step and the above inequalities persist for $j = m$.
\end{thm}
\begin{proof}
    Proposition \ref{base} gives the base case. Then by the induction hypothesis, we know the metrics $g_j$'s are near $g_0$, so the Taylor remainder constant $C$ can be taken independent of $m$. (\ref{weak solution}) gives a solution $a_m$ of 
    $$\Psi_{g_0}(a_m) = \left(
\begin{matrix} 
L_{g_0}(a_m)\\
B_{g_0}(a_m)
\end{matrix}
\right) = \left(
\begin{matrix} 
R' - R(g_m)\\
2(H' - H(g_m))
\end{matrix}
\right)$$
where $g_m = g_{m-1} + a_{m-1}$, and by Theorem \ref{weighted Schauder}, 
    \begin{align*}
    \|a_m\|_{\mathcal{B}_{2}(\Omega)} & \le C(\|R' - R(g_m)\|_{\mathcal{B}_{0}(\Omega)} + \|H' - H(g_m)\|_{\mathcal{B}_{1}(\Sigma)})\\
    & \le C(\|R' - R(g_0)\|_{\mathcal{B}_{0}(\Omega)} + \|H' - H(g_0)\|_{\mathcal{B}_{1}(\Sigma)})^{(1+m\delta)}.
\end{align*}
Then note by Taylor expansion,
\begin{align*}
\left(
\begin{matrix} 
R(g_{m+1})\\
H(g_{m+1})
\end{matrix}
\right) & = \left(
\begin{matrix} 
R(g_m) + L_{g_m}(a_m)\\
H(g_m) + \frac12B_{g_m}(a_m)
\end{matrix}
\right) + \left(
\begin{matrix} 
Q_{g_m,1}(a_m)\\
Q_{g_m,2}(a_m)
\end{matrix}
\right)\\
& = \left(
\begin{matrix} 
R'\\
H'
\end{matrix}
\right) + \sum^{m-1}_{j=0}\left(
\begin{matrix} 
L_{g_{j+1}}(a_m) - L_{g_j}(a_m)\\
\frac12B_{g_{j+1}}(a_m) - \frac12B_{g_j}(a_m)
\end{matrix}
\right)+ \left(
\begin{matrix} 
Q_{g_m,1}(a_m)\\
Q_{g_m,2}(a_m)
\end{matrix}
\right).
\end{align*}
So we have
\begin{align*}
& \|R' - R(g_{m+1})\|_{\mathcal{B}_{0}(\Omega)} + \|H' - H(g_{m+1})\|_{\mathcal{B}_{1}(\Sigma)}\\ 
& \leq C'\left(\left\|a_{m}\right\|^2_{\mathcal{B}_2(\Omega)} + \left\|a_{m}\right\|_{\mathcal{B}_2(\Omega)}\sum^{m-1}_{j=0}\left(\|L_{g_{j+1}} - L_{g_j}\| + \frac12\|B_{g_{j+1}} - B_{g_j}\|\right)\right)\\ 
& \leq C'\left(\left\|a_{m}\right\|^2_{\mathcal{B}_2(\Omega)} + \left\|a_{m}\right\|_{\mathcal{B}_2(\Omega)}\sum^{m-1}_{j=0}\left\|a_{j}\right\|_{\mathcal{B}_2(\Omega)}\right)\\ 
& \leq C'C^2\left(\|(R' - R(g_0), H' - H(g_0))\|_{\mathcal{B}_{0}(\Omega)\times\mathcal{B}_{1}(\Sigma)}^{(2+2m\delta)}\right.\\
& \quad \left.+ \|(R' - R(g_0), H' - H(g_0))\|_{\mathcal{B}_{0}(\Omega)\times\mathcal{B}_{1}(\Sigma)}^{(2+m\delta)}\sum^{m-1}_{j=0}\|(R' - R(g_0), H' - H(g_0))\|_{\mathcal{B}_{0}(\Omega)\times\mathcal{B}_{1}(\Sigma)}^{(j\delta)}\right)\\
& \leq 2C'C^2\epsilon^{1-\delta}(1-\epsilon^\delta)^{-1}\left(\|R' - R(g_0)\|_{\mathcal{B}_{0}(\Omega)} + \|H' - H(g_0)\|_{\mathcal{B}_{1}(\Sigma)}\right)^{(1+(m+1)\delta)}.
\end{align*}
Choose $\epsilon>0$ small enough so that $2C'C^2\epsilon^{1-\delta}(1-\epsilon^\delta)^{-1}\le1$.
\end{proof}

This shows that the series $\sum_{m=0}^{\infty} a_{m}$ converges geometrically to some “small” $a\in \mathcal{B}_{2}(\Omega)$, and hence $g_m$ converges in $\mathcal{M}^{k+2,\alpha}$ to $g=g_0 + a$ with $R(g)=R'$ and $H(g)=H'$. This completes the proof of our main Theorem \ref{main}.



\section*{Appendix. The linearization of the mean curvature}
We would like to use local coordinates to derive the formula for the linearization of the mean curvature $\dot{H}$. Consider the variation of metrics $g_t = g_0 + ta$ on $\overline{\Omega}$, where $a\in\mathcal{C}^\infty(\overline{\Omega})$. In the following discussion, we use a dot to mean the derivative at $t = 0$. 

Let us consider the inclusion map
$$F = i: \p\Omega \to \overline{\Omega}; \qquad x = (x^1,\dots, x^{n-1}) \mapsto F(x) = (F^1(x),\dots, F^n(x)).$$
Here we are using local coordinates on $\overline{\Omega}$ that need not be adapted to $\p\Omega$. Note that the inclusion map and local coordinates are invariant under the variation. Then the induced metric on $\p\Omega$ is
$$\hat{g}_{ij} (t) = g_{\alpha\beta}(t) \frac{\p F^\alpha}{\p x^i} \frac{\p F^\beta}{\p x^j}$$
where $i, j = 1, \dots, n-1; \, \, \alpha, \beta = 1, \dots, n$; and we have
$$\dot{\hat{g}}_{ij} = \left.\frac{d}{d t}\right|_{t=0}\left(g_{0}+t a\right)_{\alpha \beta} \frac{\partial F^{\alpha}}{\partial x^{i}} \frac{\partial F^{\beta}}{\partial x^j}=a_{\alpha \beta} \frac{\partial F^{\alpha}}{\partial x^{i}} \frac{\partial F^{\beta}}{\partial x^{j}}.$$
Here we are using coordinates for which the Christoffel symbols for $g_0$ vanish at a point $p\in\p\Omega$.

Let $\nu(t)$ be the outer unit normal (w.r.t. $g_t$) to $\p\Omega$. We choose a local field of frames $e_1,\dots, e_n$ in $\Omega$ such that restricted to $\p\Omega$, we have $e_i = \p F/\p x^i, e_n(t) = \nu(t)$. Then

$$
\left\{\begin{aligned} 
\|\nu(t)\|^2_{g_t} = g_{\alpha\beta}(t) \nu^\alpha(t) \nu^\beta(t) & \equiv 1  \\ 
\left<\nu(t), e_i\right>_{g_t} = g_{\alpha\beta}(t) \nu^\alpha(t) e^\beta_i & \equiv 0 \quad (i=1,\dots, n-1).
\end{aligned}\right.
$$
Differentiating them on both sides and evaluating at $t = 0$, we will get
\begin{align}\label{normal}
\left\{\begin{aligned} 
2g_{\alpha\beta}^0 \dot{\nu}^\alpha \nu^\beta + a_{\alpha\beta} \nu^\alpha \nu^\beta & = 0\\ 
g_{\alpha\beta}^0 \dot{\nu}^\alpha e^\beta_i + a_{\alpha\beta} \nu^\alpha e^\beta_i & = 0.
\end{aligned}\right.
\end{align}
Here and in the following calculation $g^0 = g_0$. If we write 
$$\dot{\nu}=c \nu + c^i e_i,$$
then from (\ref{normal}) we will get 
$$
\left\{\begin{aligned} 
-\frac12 a(\nu, \nu) & = \left<\dot{\nu}, \nu\right>_{g_0} = \left<c \nu, \nu\right>_{g_0} = c\\
-a(\nu, e_i) & = \left<\dot{\nu}, e_i\right>_{g_0} = \left<c^j e_j, e_i\right>_{g_0} = c^j \hat{g}_{ij}.
\end{aligned}\right.
$$
So $$\dot{\nu}=-\frac12 a(\nu, \nu) \, \nu - \hat{g}^{ij} a(\nu, e_j) \, e_i.$$
The second fundamental form $h_{ij}$ is
$$h_{ij}(t) = -\left<\nu(t), D_{e_i}e_j\right>_{g_t} = -g_{\alpha\theta}(t)\nu^\theta(t)\left(\frac{\partial^{2} F^{\alpha}}{\partial x^{i} \partial x^{j}} + \Gamma_{\beta \gamma}^{\alpha}(t) \frac{\partial F^{\beta}}{\partial x^{i}} \frac{\partial F^{\gamma}}{\partial x^{j}}\right)$$
where the Levi-Civita connection 
$
\Gamma_{\beta \gamma}^{\alpha} = \frac12 g^{\alpha\sigma} \left(\p_\beta g_{\gamma\sigma} + \p_\gamma g_{\beta\sigma} - \p_\sigma g_{\beta\gamma}\right)
$.
Note that we assumed $\p_\alpha g^0_{\beta\gamma} \equiv 0$, so
$$\dot{\Gamma}_{\beta \gamma}^{\alpha} = \frac12 g_0^{\alpha\sigma}\left(\p_\beta a_{\gamma\sigma} + \p_\gamma a_{\beta\sigma} - \p_\sigma a_{\beta\gamma}\right).$$
Thus,
\begin{align*}
\dot{h}_{ij} 
		& = -a(\nu, D_{e_i}e_j) - \left<-\frac12 a(\nu, \nu) \nu - \hat{g}^{kl} a(\nu, e_l) e_k, D_{e_i}e_j\right>_{g_0}\\
		& \quad - g^0_{\alpha\theta}\, \nu^\theta \frac12 g_0^{\alpha\sigma}\left(\p_\beta a_{\gamma\sigma} + \p_\gamma a_{\beta\sigma} - \p_\sigma a_{\beta\gamma}\right)e^\beta_i e^\gamma_j\\
		& = -a(\nu, D_{e_i}e_j) - \frac12 a(\nu, \nu) h_{ij} + \hat{g}^{kl} a(\nu, e_l) \left<e_k, D_{e_i}e_j\right>_{g_0}\\
		& \quad - \frac12\nu^\sigma \left(\p_\beta a_{\gamma\sigma} + \p_\gamma a_{\beta\sigma} - \p_\sigma a_{\beta\gamma}\right)e^\beta_i e^\gamma_j.
\end{align*}
By the Gauss-Weingarten Formulae,
\begin{align*}
D^\alpha_{e_i}e_j - \nabla^\alpha_{e_i}e_j & = \frac{\partial^{2} F^{\alpha}}{\partial x^{i} \partial x^{j}} + \Gamma_{\beta \gamma}^{\alpha} \frac{\partial F^{\beta}}{\partial x^{i}} \frac{\partial F^{\gamma}}{\partial x^{j}} - \hat{\Gamma}_{i j}^{k} \frac{\partial F^{\alpha}}{\partial x^{k}} = -h_{i j} \nu^{\alpha},\\
D^\alpha_{e_i}\nu & = \frac{\partial \nu^{\alpha}}{\partial x^{i}}+\Gamma_{\beta \gamma}^{\alpha} \frac{\partial F^{\beta}}{\partial x^{i}} \nu^{\gamma} = h_{i j} \hat{g}^{j l} \frac{\partial F^{\alpha}}{\partial x^{l}},
\end{align*}
we can rewrite the first three terms of $\dot{h}_{ij}$ as
\begin{align*}
\text{I} + \text{II} + \text{III} &= -a(\nu, D_{e_i}e_j) - \frac12 a(\nu, \nu) h_{ij} + \hat{g}^{kl} a(\nu, e_l) \left<e_k, \hat{\Gamma}_{i j}^p e_p - h_{i j} \nu\right>_{g_0}\\
		&= -a(\nu, D_{e_i}e_j) - \frac12 a(\nu, \nu) h_{ij} + a(\nu, e_l) \hat{\Gamma}_{i j}^l\\
		&= a(\nu, e_l \hat{\Gamma}_{i j}^l - D_{e_i}e_j) - \frac12 a(\nu, \nu)h_{ij}\\
		&= a(\nu, h_{ij}\nu) - \frac12 a(\nu, \nu) h_{ij}\\
		& = \frac12 a(\nu, \nu) h_{ij}.
\end{align*}
So
$$\dot{h}_{ij} = \frac12 a(\nu, \nu)h_{ij} - \frac12\nu^\sigma \left(\p_\beta a_{\gamma\sigma} + \p_\gamma a_{\beta\sigma} - \p_\sigma a_{\beta\gamma}\right)e^\beta_i e^\gamma_j.$$

Finally, the linearization $\dot{H}$ of the mean curvature operator is
\begin{align*}
\dot{H} & = \dot{h}_{ij}\hat{g}^{ij} - h_{ij}\hat{g}^{ik}\dot{\hat{g}}_{kl}\hat{g}^{lj}\\
	& = \left(\frac12 a(\nu, \nu) h_{ij} - \frac12\nu^\sigma \left(\p_\beta a_{\gamma\sigma} + \p_\gamma a_{\beta\sigma} - \p_\sigma a_{\beta\gamma}\right)e^\beta_i e^\gamma_j\right)\hat{g}^{ij} - a(e_k, e_l)h^{kl}\\
	& = - \frac12 \left(2\p_\beta a_{\gamma\sigma} - \p_\sigma a_{\beta\gamma}\right)\nu^\sigma e^\beta_i e^\gamma_j \hat{g}^{ij} + \frac12 a(\nu, \nu) H - \left<a, h\right>_{\hat{g}}.
\end{align*}
We note that
$$
\nu(\hat{g}_{ij}) = \nu\left<e_i, e_j\right>_{g_0} = \left<D_\nu e_i, e_j\right>_{g_0} + \left<e_i, D_\nu e_j\right>_{g_0} = \left<D_{e_i} \nu, e_j\right>_{g_0} + \left<e_i, D_{e_j} \nu\right>_{g_0} = 2h_{ij},
$$
and
$$
\nu(\hat{g}^{ij}) = -\hat{g}^{ik}\nu(\hat{g}_{kl})\hat{g}^{lj} = -\hat{g}^{ik}(2h_{kl})\hat{g}^{lj} = -2h^{ij}.
$$
Here we are using Fermi coordinates so that $\nu$ and $e_i$ commute.
So
\begin{align*}
\frac12 \p_\sigma a_{\beta\gamma}\nu^\sigma e^\beta_i e^\gamma_j \hat{g}^{ij}& = \frac12 \nu(a_{\beta\gamma} e^\beta_i e^\gamma_j) \hat{g}^{ij} - a_{\beta\gamma} \nu(e^\beta_i) e^\gamma_j \hat{g}^{ij}\\
		& = \frac12 \nu(a(e_i, e_j)) \hat{g}^{ij} - a(D_\nu e_i, e_j) \hat{g}^{ij}\\
		& = \frac12 \nu(a(e_i, e_j) \hat{g}^{ij}) - \frac12 a(e_i, e_j) \nu(\hat{g}^{ij}) - a(h_{ik} \hat{g}^{kl}e_l, e_j) \hat{g}^{ij}\\
		& = \frac12 \nu(\operatorname{tr}_{\hat{g}}a) + a(e_i, e_j)h^{ij} - a(e_l, e_j) h^{jl}\\
		& = \frac12 \nu(\operatorname{tr}_{\p\Omega} a),
\end{align*}
and
\begin{align*}
-\p_\beta a_{\gamma\sigma}\nu^\sigma e^\beta_i e^\gamma_j \hat{g}^{ij}& = -e_i(a_{\gamma\sigma}\nu^\sigma e^\gamma_j)\hat{g}^{ij} + a_{\gamma\sigma}\nu^\sigma e_i(e^\gamma_j)\hat{g}^{ij} + a_{\gamma\sigma}e_i(\nu^\sigma) e^\gamma_j\hat{g}^{ij}\\
		& = -e_i(a(\nu, e_j))\hat{g}^{ij} + a(\nu, D_{e_i}e_j)\hat{g}^{ij} + a(D_{e_i}\nu, e_j)\hat{g}^{ij}\\
		& = -e_i(a(\nu, e_j))\hat{g}^{ij} + a(\nu, \nabla_{e_i}e_j)\hat{g}^{ij} - a(\nu, h_{ij}\nu)\hat{g}^{ij} + a(e_l, e_j) h^{jl}\\
		& = -\operatorname{div}_{\p\Omega} a(\nu, \cdot) - a(\nu, \nu) H + \left<a, h\right>_{\hat{g}},
\end{align*}
where $\operatorname{div}_{\p\Omega} a(\nu,\cdot) = \left<\nabla_{e_i} (a(\nu,\cdot)), e_j\right>\hat{g}^{ij}$. Thus,
\begin{align*}
\dot{H} & = -\operatorname{div}_{\p\Omega} a(\nu, \cdot) - a(\nu, \nu) H + \left<a, h\right>_{\hat{g}} + \frac12 \nu(\operatorname{tr}_{\p\Omega} a) + \frac12 a(\nu, \nu) H - \left<a, h\right>_{\hat{g}}\\
		& = \frac12 \nu(\operatorname{tr}_{\p\Omega} a) - \operatorname{div}_{\p\Omega} a(\nu, \cdot) - \frac12 a(\nu, \nu) H.
\end{align*}
Strictly speaking, the term $\nu(\operatorname{tr}_{\partial\Omega} a)$ is well-defined only with respect to a foliation near $\partial\Omega$. However, the variation $\dot{H}$ itself is independent of this choice.




\bibliographystyle{amsplain}

\begin{bibdiv}
\begin{biblist}


\bib{A-D-N}{article}{
   author={Agmon, S.},
   author={Douglis, A.},
   author={Nirenberg, L.},
   title={Estimates near the boundary for solutions of elliptic partial
   differential equations satisfying general boundary conditions. I},
   journal={Comm. Pure Appl. Math.},
   volume={12},
   date={1959},
   pages={623--727},
   issn={0010-3640},
   review={\MR{125307}},
   doi={10.1002/cpa.3160120405},
}

\bib{A-B-dL}{article}{
   author={Almaraz, S\'{e}rgio},
   author={Barbosa, Ezequiel},
   author={de Lima, Levi Lopes},
   title={A positive mass theorem for asymptotically flat manifolds with a
   non-compact boundary},
   journal={Comm. Anal. Geom.},
   volume={24},
   date={2016},
   number={4},
   pages={673--715},
   issn={1019-8385},
   review={\MR{3570413}},
   doi={10.4310/CAG.2016.v24.n4.a1},
}

\bib{A-dL}{article}{
   author={Almaraz, S\'{e}rgio},
   author={de Lima, Levi Lopes},
   title={The mass of an asymptotically hyperbolic manifold with a
   non-compact boundary},
   journal={Ann. Henri Poincar\'{e}},
   volume={21},
   date={2020},
   number={11},
   pages={3727--3756},
   issn={1424-0637},
   review={\MR{4163306}},
   doi={10.1007/s00023-020-00954-w},
}

\bib{A-dL2}{article}{
   author={Almaraz, S\'{e}rgio},
   author={de Lima, Levi Lopes},
   title={Rigidity of non-compact static domains in hyperbolic space via positive mass theorems},
   date={2022},
   eprint={arXiv 2206.09768},
}


\bib{Aubin}{book}{
   author={Aubin, Thierry},
   title={Some nonlinear problems in Riemannian geometry},
   series={Springer Monographs in Mathematics},
   publisher={Springer-Verlag, Berlin},
   date={1998},
   pages={xviii+395},
   isbn={3-540-60752-8},
   review={\MR{1636569}},
   doi={10.1007/978-3-662-13006-3},
}






\bib{B-E-M}{article}{
   author={Bourguignon, Jean-Pierre},
   author={Ebin, David G.},
   author={Marsden, Jerrold E.},
   title={Sur le noyau des op\'erateurs pseudo-diff\'erentiels \`a{} symbole
   surjectif et non injectif},
   journal={C. R. Acad. Sci. Paris S\'er. A-B},
   volume={282},
   date={1976},
   number={16},
   pages={Aii, A867--A870},
   issn={0151-0509},
   review={\MR{0402829}},
}



\bib{Brendle-C}{article}{
   author={Brendle, Simon},
   author={Chen, Szu-Yu Sophie},
   title={An existence theorem for the Yamabe problem on manifolds with
   boundary},
   journal={J. Eur. Math. Soc. (JEMS)},
   volume={16},
   date={2014},
   number={5},
   pages={991--1016},
   issn={1435-9855},
   review={\MR{3210959}},
   doi={10.4171/JEMS/453},
}

\bib{Carlotto-S}{article}{
   author={Carlotto, Alessandro},
   author={Schoen, Richard},
   title={Localizing solutions of the Einstein constraint equations},
   journal={Invent. Math.},
   volume={205},
   date={2016},
   number={3},
   pages={559--615},
   issn={0020-9910},
   review={\MR{3539922}},
   doi={10.1007/s00222-015-0642-4},
}



\bib{C-D}{article}{
   author={Chru\'{s}ciel, Piotr T.},
   author={Delay, Erwann},
   title={On mapping properties of the general relativistic constraints
   operator in weighted function spaces, with applications},
   language={English, with English and French summaries},
   journal={M\'{e}m. Soc. Math. Fr. (N.S.)},
   number={94},
   date={2003},
   pages={vi+103},
   issn={0249-633X},
   review={\MR{2031583}},
   doi={10.24033/msmf.407},
}

\bib{C}{article}{
   author={Corvino, Justin},
   title={Scalar curvature deformation and a gluing construction for the
   Einstein constraint equations},
   journal={Comm. Math. Phys.},
   volume={214},
   date={2000},
   number={1},
   pages={137--189},
   issn={0010-3616},
   review={\MR{1794269}},
   doi={10.1007/PL00005533},
}

\bib{C-E-M}{article}{
   author={Corvino, Justin},
   author={Eichmair, Michael},
   author={Miao, Pengzi},
   title={Deformation of scalar curvature and volume},
   journal={Math. Ann.},
   volume={357},
   date={2013},
   number={2},
   pages={551--584},
   issn={0025-5831},
   review={\MR{3096517}},
   doi={10.1007/s00208-013-0903-8},
}

\bib{C-H}{article}{
   author={Corvino, Justin},
   author={Huang, Lan-Hsuan},
   title={Localized deformation for initial data sets with the dominant
   energy condition},
   journal={Calc. Var. Partial Differential Equations},
   volume={59},
   date={2020},
   number={1},
   pages={Paper No. 42, 43},
   issn={0944-2669},
   review={\MR{4062040}},
   doi={10.1007/s00526-019-1679-9},
}

\bib{C-S}{article}{
   author={Corvino, Justin},
   author={Schoen, Richard M.},
   title={On the asymptotics for the vacuum Einstein constraint equations},
   journal={J. Differential Geom.},
   volume={73},
   date={2006},
   number={2},
   pages={185--217},
   issn={0022-040X},
   review={\MR{2225517}},
}

\bib{C-V}{article}{
   author={Cruz, Tiarlos},
   author={Vit\'{o}rio, Feliciano},
   title={Prescribing the curvature of Riemannian manifolds with boundary},
   journal={Calc. Var. Partial Differential Equations},
   volume={58},
   date={2019},
   number={4},
   pages={Paper No. 124, 19},
   issn={0944-2669},
   review={\MR{3977557}},
   doi={10.1007/s00526-019-1584-2},
}


\bib{Dain-J-K}{article}{
   author={Dain, Sergio},
   author={Jaramillo, Jos\'{e} Luis},
   author={Krishnan, Badri},
   title={Existence of initial data containing isolated black holes},
   journal={Phys. Rev. D (3)},
   volume={71},
   date={2005},
   number={6},
   pages={064003, 11},
   issn={0556-2821},
   review={\MR{2138833}},
   doi={10.1103/PhysRevD.71.064003},
}

\bib{Delay}{article}{
   author={Delay, Erwann},
   title={Localized gluing of Riemannian metrics in interpolating their
   scalar curvature},
   journal={Differential Geom. Appl.},
   volume={29},
   date={2011},
   number={3},
   pages={433--439},
   issn={0926-2245},
   review={\MR{2795849}},
   doi={10.1016/j.difgeo.2011.03.008},
}



\bib{E}{article}{
   author={Escobar, Jos\'{e} F.},
   title={Uniqueness theorems on conformal deformation of metrics, Sobolev
   inequalities, and an eigenvalue estimate},
   journal={Comm. Pure Appl. Math.},
   volume={43},
   date={1990},
   number={7},
   pages={857--883},
   issn={0010-3640},
   review={\MR{1072395}},
   doi={10.1002/cpa.3160430703},
}

\bib{E1}{article}{
   author={Escobar, Jos\'{e} F.},
   title={The Yamabe problem on manifolds with boundary},
   journal={J. Differential Geom.},
   volume={35},
   date={1992},
   number={1},
   pages={21--84},
   issn={0022-040X},
   review={\MR{1152225}},
}

\bib{E2}{article}{
   author={Escobar, Jos\'{e} F.},
   title={Conformal deformation of a Riemannian metric to a scalar flat
   metric with constant mean curvature on the boundary},
   journal={Ann. of Math. (2)},
   volume={136},
   date={1992},
   number={1},
   pages={1--50},
   issn={0003-486X},
   review={\MR{1173925}},
   doi={10.2307/2946545},
}

\bib{E3}{article}{
   author={Escobar, Jos\'{e} F.},
   title={Conformal metrics with prescribed mean curvature on the boundary},
   journal={Calc. Var. Partial Differential Equations},
   volume={4},
   date={1996},
   number={6},
   pages={559--592},
   issn={0944-2669},
   review={\MR{1416000}},
   doi={10.1007/BF01261763},
}

\bib{E4}{article}{
   author={Escobar, Jos\'{e} F.},
   title={Conformal deformation of a Riemannian metric to a constant scalar
   curvature metric with constant mean curvature on the boundary},
   journal={Indiana Univ. Math. J.},
   volume={45},
   date={1996},
   number={4},
   pages={917--943},
   issn={0022-2518},
   review={\MR{1444473}},
   doi={10.1512/iumj.1996.45.1344},
}






\bib{F-M}{article}{
   author={Fischer, Arthur E.},
   author={Marsden, Jerrold E.},
   title={Deformations of the scalar curvature},
   journal={Duke Math. J.},
   volume={42},
   date={1975},
   number={3},
   pages={519--547},
   issn={0012-7094},
   review={\MR{380907}},
}

\bib{G}{book}{
   author={Grubb, Gerd},
   title={Distributions and operators},
   series={Graduate Texts in Mathematics},
   volume={252},
   publisher={Springer, New York},
   date={2009},
   pages={xii+461},
   isbn={978-0-387-84894-5},
   review={\MR{2453959}},
}

\bib{Han-Li}{article}{
   author={Han, Zheng-Chao},
   author={Li, Yanyan},
   title={The Yamabe problem on manifolds with boundary: existence and
   compactness results},
   journal={Duke Math. J.},
   volume={99},
   date={1999},
   number={3},
   pages={489--542},
   issn={0012-7094},
   review={\MR{1712631}},
   doi={10.1215/S0012-7094-99-09916-7},
}

\bib{Han-Li1}{article}{
   author={Han, Zheng-Chao},
   author={Li, Yanyan},
   title={The existence of conformal metrics with constant scalar curvature
   and constant boundary mean curvature},
   journal={Comm. Anal. Geom.},
   volume={8},
   date={2000},
   number={4},
   pages={809--869},
   issn={1019-8385},
   review={\MR{1792375}},
   doi={10.4310/CAG.2000.v8.n4.a5},
}

\bib{H-S}{article}{
   author={Hardt, Robert},
   author={Simon, Leon},
   title={Area minimizing hypersurfaces with isolated singularities},
   journal={J. Reine Angew. Math.},
   volume={362},
   date={1985},
   pages={102--129},
   issn={0075-4102},
   review={\MR{809969}},
   doi={10.1515/crll.1985.362.102},
}

\bib{H-M-R}{article}{
   author={Hijazi, Oussama},
   author={Montiel, Sebasti\'{a}n},
   author={Raulot, Simon},
   title={A positive mass theorem for asymptotically hyperbolic manifolds
   with inner boundary},
   journal={Internat. J. Math.},
   volume={26},
   date={2015},
   number={12},
   pages={1550101, 17},
   issn={0129-167X},
   review={\MR{3432532}},
   doi={10.1142/S0129167X15501013},
}

\bib{Ku84}{book}{
   author={Kufner, Alois},
   title={Weighted Sobolev spaces},
   series={Teubner-Texte zur Mathematik [Teubner Texts in Mathematics]},
   volume={31},
   note={With German, French and Russian summaries},
   publisher={BSB B. G. Teubner Verlagsgesellschaft, Leipzig},
   date={1980},
   pages={151},
   review={\MR{0664599}},
}

\bib{L-M}{book}{
   author={Lions, J.-L.},
   author={Magenes, E.},
   title={Non-homogeneous boundary value problems and applications. Vol. I},
   series={Die Grundlehren der mathematischen Wissenschaften},
   volume={Band 181},
   note={Translated from the French by P. Kenneth},
   publisher={Springer-Verlag, New York-Heidelberg},
   date={1972},
   pages={xvi+357},
   review={\MR{0350177}},
}

\bib{Marques}{article}{
   author={Marques, Fernando C.},
   title={Conformal deformations to scalar-flat metrics with constant mean
   curvature on the boundary},
   journal={Comm. Anal. Geom.},
   volume={15},
   date={2007},
   number={2},
   pages={381--405},
   issn={1019-8385},
   review={\MR{2344328}},
}

\bib{Mayer-N}{article}{
   author={Mayer, Martin},
   author={Ndiaye, Cheikh Birahim},
   title={Barycenter technique and the Riemann mapping problem of
   Cherrier-Escobar},
   journal={J. Differential Geom.},
   volume={107},
   date={2017},
   number={3},
   pages={519--560},
   issn={0022-040X},
   review={\MR{3715348}},
   doi={10.4310/jdg/1508551224},
}

\bib{M}{article}{
   author={Miao, Pengzi},
   title={Positive mass theorem on manifolds admitting corners along a
   hypersurface},
   journal={Adv. Theor. Math. Phys.},
   volume={6},
   date={2002},
   number={6},
   pages={1163--1182 (2003)},
   issn={1095-0761},
   review={\MR{1982695}},
   doi={10.4310/ATMP.2002.v6.n6.a4},
}

\bib{M-T}{article}{
   author={Miao, Pengzi},
   author={Tam, Luen-Fai},
   title={On the volume functional of compact manifolds with boundary with
   constant scalar curvature},
   journal={Calc. Var. Partial Differential Equations},
   volume={36},
   date={2009},
   number={2},
   pages={141--171},
   issn={0944-2669},
   review={\MR{2546025}},
   doi={10.1007/s00526-008-0221-2},
}



\bib{Schoen}{article}{
   author={Schoen, Richard},
   title={Conformal deformation of a Riemannian metric to constant scalar
   curvature},
   journal={J. Differential Geom.},
   volume={20},
   date={1984},
   number={2},
   pages={479--495},
   issn={0022-040X},
   review={\MR{0788292}},
}

\bib{S-Y1}{article}{
   author={Schoen, Richard},
   author={Yau, Shing Tung},
   title={On the proof of the positive mass conjecture in general
   relativity},
   journal={Comm. Math. Phys.},
   volume={65},
   date={1979},
   number={1},
   pages={45--76},
   issn={0010-3616},
   review={\MR{526976}},
}

\bib{S-Y3}{article}{
   author={Schoen, Richard},
   author={Yau, Shing Tung},
   title={Proof of the positive mass theorem. II},
   journal={Comm. Math. Phys.},
   volume={79},
   date={1981},
   number={2},
   pages={231--260},
   issn={0010-3616},
   review={\MR{612249}},
}

\bib{Sheng}{article}{
   author={Sheng, Hongyi},
   title={Static Manifolds with Boundary and Rigidity of Scalar Curvature
   and Mean Curvature},
   journal={Int. Math. Res. Not. IMRN},
   date={2025},
   number={7},
   pages={rnaf086},
   issn={1073-7928},
   review={\MR{4888700}},
   doi={10.1093/imrn/rnaf086},
}

\bib{Sheng-Zhao}{article}{
   author={Sheng, Hongyi},
   author={Zhao, Kai-Wei},
   title={A Gluing Construction for the Einstein Constraint Equations in Half Space}, 
   journal={in preparation},
}

\bib{S-T}{article}{
   author={Shi, Yuguang},
   author={Tam, Luen-Fai},
   title={Positive mass theorem and the boundary behaviors of compact
   manifolds with nonnegative scalar curvature},
   journal={J. Differential Geom.},
   volume={62},
   date={2002},
   number={1},
   pages={79--125},
   issn={0022-040X},
   review={\MR{1987378}},
}

\bib{bS}{book}{
   author={Simon, Barry},
   title={A Comprehensive Course in Analysis, Part 4: Operator Theory},
   publisher={American Mathematical Society, Providence, RI},
   date={2015},
}






\bib{lS4}{article}{
   author={Simon, Leon},
   title={Schauder estimates by scaling},
   journal={Calc. Var. Partial Differential Equations},
   volume={5},
   date={1997},
   number={5},
   pages={391--407},
   issn={0944-2669},
   review={\MR{1459795}},
   doi={10.1007/s005260050072},
}

\bib{Sm}{article}{
   author={Smale, Nathan},
   title={Generic regularity of homologically area minimizing hypersurfaces
   in eight-dimensional manifolds},
   journal={Comm. Anal. Geom.},
   volume={1},
   date={1993},
   number={2},
   pages={217--228},
   issn={1019-8385},
   review={\MR{1243523}},
   doi={10.4310/CAG.1993.v1.n2.a2},
}

\bib{Trudinger}{article}{
   author={Trudinger, Neil S.},
   title={Remarks concerning the conformal deformation of Riemannian
   structures on compact manifolds},
   journal={Ann. Scuola Norm. Sup. Pisa Cl. Sci. (3)},
   volume={22},
   date={1968},
   pages={265--274},
   issn={0391-173X},
   review={\MR{0240748}},
}

\end{biblist}
\end{bibdiv}

\end{document}